\RequirePackage{fix-cm}
\documentclass{svjour3}                     
\smartqed  
\usepackage{graphicx}
\usepackage{verbatim}   
\usepackage{color}      
\usepackage{subfigure}  
\usepackage{hyperref}   
\usepackage{amsmath}
\usepackage{algorithm}
\usepackage{slashbox}
\usepackage{setspace}
\usepackage[noend]{algpseudocode}
\usepackage{mathtools}
\usepackage{multirow}
\usepackage{amsfonts}
\usepackage{amssymb}
\usepackage[pdftex,dvipsnames]{xcolor}  
\usepackage{xargs}                      
\usepackage[colorinlistoftodos,prependcaption,textsize=tiny]{todonotes}

\newcommand{\eqnum}{\refstepcounter{equation}\textup{\tagform@{\theequation}}}

\newcommandx{\info}[2][1=]{\todo[linecolor=OliveGreen,backgroundcolor=OliveGreen!25,bordercolor=OliveGreen,#1]{#2}}

\newcommand\abs[1]{\left|#1\right|}
\newcommand\norm[1]{\left\|#1\right\|}

\usepackage{xparse}
\NewDocumentCommand{\ceil}{s O{} m}{%
  \IfBooleanTF{#1} 
    {\left\lceil#3\right\rceil} 
    {#2\lceil#3#2\rceil} 
}

\DeclareMathOperator{\diam}{diam}
 \journalname{myjournal}

\smartqed

\begin{document}

\title{Analysis and Performance of the Barzilai-Borwein Step-Size Rules for Optimization Problems in Hilbert Spaces\thanks{The work of Karl Kunisch has been partially supported the ERC advanced grant 668998 (OCLOC) under the EU's H2020 research program.}
}

\titlerunning{BB-Method  for Optimization Problems in Hilbert Spaces}        

\author{Behzad Azmi         \and
        Karl Kunisch 
}


\institute{Behzad Azmi \at
              Johann Radon Institute for Computational and Applied Mathematics (RICAM),  Austrian Academy of Sciences\\
             Altenbergerstra{\ss}e 69, A-4040 Linz, Austria\\
              \email{behzad.azmi@ricam.oeaw.ac.at}           
           \and
           Karl Kunisch \at
              Institute for Mathematics and Scientific Computing,  University of Graz\\
              Heinrichstra{\ss}e 36, 8010 Graz, Austria\\
              \email{karl.kunisch@uni-graz.at}\\
             Johann Radon Institute for Computational and Applied Mathematics (RICAM), Austrian Academy of Sciences\\
             Altenbergerstra{\ss}e 69, A-4040 Linz, Austria
}

\date{Received: date / Accepted: date}

\maketitle

\begin{abstract}
Due to simplicity, computational cheapness, and efficiency, the Barzilai and Borwein (BB) gradient method has received a significant amount of attention in different fields of optimization.  In the first part of this paper, based on spectral analysis, R-linear global convergence for the BB-method is proven for strictly convex quadratic problems posed in infinite-dimensional Hilbert spaces. Then this result is strengthened to R-linear local convergence for a class of twice continuously Fre\'chet-differentiable functions.  In the second part,  aiming at problems governed by partial differential equations (PDE), the mesh-independent principle is investigated for the BB-method. The applicability of these results is demonstrated  for three different types of PDE-constrained optimization problems. Numerical experiments illustrate the theoretical results.
\end{abstract}

\section{Introduction}
First-order (gradient) methods are progressively getting more attention since it has been realized that for a suitable choice of the step-length, using the negative gradient as the search direction may give rise to very efficient algorithmic behavior. As a pioneering work, we can refer to the method proposed by Barzilai and Borwein in \cite{MR967848} abbreviated as the BB-method.  In this work, the authors demonstrated that choosing an appropriate step-length leads to a significant acceleration over the steepest descent method. The BB-method incorporates the quasi-Newton property, by approximating the Hessian matrix by a scalar times the identity which satisfies the secant condition. Despite the simplicity and cheapness, this method has exhibited a surprisingly efficient numerical behaviour. This stimulated a significant amount of research. In the original work \cite{MR967848}, the authors established R-superlinear convergence for two-dimensional strictly convex quadratic problems. Later,  Raydan \cite{MR1225468} and  Dai and Liao \cite{MR1880051} proved, respectively, global convergence and  R-linear convergence rate of the BB-method for any finite-dimensional strictly convex quadratic problem. One of the important feature of this method  is the nonmonotonicity in the values of the objective function and gradient norm. To preserve this feature, some authors \cite{MR1937087,MR1430555} managed to prove the global convergence of the BB-method for finite-dimensional  unconstrained optimization problems based on the nonmonotone line search techniques introduced in \cite{MR849278}.  A deep analysis of the asymptotic behaviour of the BB-method was given in  \cite{MR2166548,MR2144378}. In these works, the surprising computational efficiency of the algorithm in relation to its nonmonotonicity was discussed and several circumstances were presented under which the performance of the BB-method (without globalization) is competitive, or even, superior to conjugate gradient methods. This occurs, for instance,  when a low accuracy for the solution of problem is required, or when significant round-off errors are present,  and the objective functions is made up of a quadratic function plus a small
non-quadratic term (near quadratic). Since then,  inspired by the BB-method, many authors designed and analysed  several  step-length rules for the gradient method by investigating the role and behaviour of the eigenvalues of the Hessian matrix, rather than  the decrease of the function values see e.g., \cite{MR3483103,MR2166548,MR3274863,MR2914422,MR2968263,MR3627452,MR2204453,MR3596866}. Due to simplicity and efficiency, the BB step-sizes  have been being widely used in various fields of  mathematical optimization and applications, including nonsmooth optimization  \cite{MR3543170,MR3628901,MR3564785,MR3258520,MR3258526,MR3627010,MR2678081,MR3305896}, inverse problems \cite{MR3180413,MR2486523,MR3652258,MR3010235,MR3627052,MR3684638,MR3562271,MR3711800,MR3626799}  constrained optimization \cite{MR3016291,MR3757115,MR3569614,MR3397072,MR3041752,MR3604055}.

In this work we aim to study the BB-method within the scope of PDE-constrained optimization. For optimization problems governed by partial different equations,  every  function evaluation is typically carried out through solving a partial differential equation (state equation). Hence function evaluations can be computationally very expensive  and it is desirable to avoid them as far as possible.  Moreover,  due to numerical discretization,  the presence of round-off and truncation errors is inevitable and,  depending on the discretization procedure,  the finite-dimensional approximation for the gradient of the original problem need not coincide with the gradient of the finite-dimensional approximation for the original problem (optimization and discretization do not commute).  A wide range of models arising from industry and natural science are formulated as  optimization problems governed by linear and semilinear partial differential equations.  For these problems, the corresponding reduced formulations lead to infinite-dimensional quadratic and near quadratic unconstrained optimization problems.  In this respect, we mention \cite{AzmiKunisch,MR3721863,azmisemi} in which the BB-method was efficiently employed in the context of the model predictive control for PDEs.  In view of the above discussion, we are motivated to study the BB-method for a more general class of problems, namely, unconstrained problems posed in infinite-dimensional Hilbert spaces. Here we focus on the following unconstrained optimization problem
\begin{equation}
\label{p}
\min_{u \in \mathcal{H}} \mathcal{F}(u),
\end{equation}
 where $\mathcal{F}:\mathcal{H}  \to \mathbb{R}$ is a twice continuously Fr\'echet differentiable function defined on an abstract Hilbert space $\mathcal{H}$ with the inner product $(\cdot,\cdot)$  and its associated norm  $\|\cdot\|$.   The Barzilai-Borwein iterations for solving \eqref{p} are defined by
\begin{equation}
\label{e1}
u_{k+1} = u_k - \frac{1}{\alpha_k}\mathcal{G}_k,
\end{equation}
where $\mathcal{G}_k:=\mathcal{G}(u_k)$ and $\mathcal{G} :\mathcal{H} \to \mathcal{H}$ stands for the gradient of $\mathcal{F}$. This gradient is defined by $\mathcal{G}:=\mathcal{R} \circ \mathcal{F}'$, where $\mathcal{F}': \mathcal{H}\to \mathcal{H'}$ is the first derivative of $\mathcal{F}$,  and $\mathcal{R}: \mathcal{H'} \to \mathcal{H}$ is the Riesz isomorphism,  with $\mathcal{H}'$ denoting  the dual space of $\mathcal{H}$. Thus for every $\delta u \in \mathcal{H}$, we have $\mathcal{F}'(u)\delta u = (\mathcal{G}(u), \delta u)$, with $(\cdot,\cdot)$ denoting the inner product in $\mathcal{H}$.  Furthermore, the step-size $\alpha_k>0$ is chosen according to either
\begin{equation}
\label{e1a}
\alpha^{BB1}_k := \frac{(\mathcal{S}_{k-1},\mathcal{Y}_{k-1})}{(\mathcal{S}_{k-1},\mathcal{S}_{k-1})}, \quad \text{ or } \quad
\alpha^{BB2}_k := \frac{ (\mathcal{Y}_{k-1},\mathcal{Y}_{k-1})}{(\mathcal{S}_{k-1},\mathcal{Y}_{k-1})},
\end{equation}
where $\mathcal{S}_{k-1}:=u_k-u_{k-1}$ and $\mathcal{Y}_{k-1}:=\mathcal{G}_k-\mathcal{G}_{k-1}$. With these specifications we are prepared to specify Algorithm \ref{BBa}  which will be investigated in this paper.
\begin{algorithm}[htbp!]
\caption{BB-gradient}\label{BBa}
\begin{spacing}{1.1}
\begin{algorithmic}[1]
\Require Let one of the following initial conditions be satisfied:
\begin{itemize}
\item{C1:} Initial iterates $u_0, u_1 \in \mathcal{H}$ with $u_0\neq u_1$ have been given.
\item{C2:} An initial iterate $u_1 \in \mathcal{H}$ and an initial step-size $\alpha_1$ with $\alpha_1>0$  have been given.
\end{itemize}
\State Set $k=1$.
\State If $\|\mathcal{G}_k\|=0$ stop.
\State  If $k=1$ and C2 holds go to Step 4,  otherwise choose  $\alpha_k$ equal either $\alpha^{BB1}_k$ or $\alpha^{BB2}_k$.
\State  Set $u_{k+1} = u_k- \frac{1}{\alpha_k}\mathcal{G}_k$, $k=k+1$, and go to Step 2.
\end{algorithmic}
\end{spacing}
\end{algorithm}
As mentioned before,  numerous results have been published on the BB-method, but, to the best of our knowledge, for optimization problems posed in infinite-dimensional spaces, there still does not exist a rigorous theory. Here we take a step in this direction and, as a first contribution, we analyse  the convergence of Algorithm \ref{BBa}.  Inspired by the result in \cite{MR1880051} and based on the spectral theorem, we establish the R-linear global convergence of  Algorithm \ref{BBa} when it is applied to  strictly convex quadratic problems defined by bounded uniformly positive self-adjoint operators. Then this result will be extended to a local convergence result for twice continuously Fr\'echet differentiable functions.

As the second contribution, we analyse  the mesh independence principle (MIP) for Algorithm \ref{BBa}.  This important property roughly states that the algorithm shows a similar convergence behaviour for the infinite-dimensional problem and its finite-dimensional approximations (discretized problems), independent of the mesh size. This concept of MIP was initially introduced in \cite{MR821912} for  Newton's method. Since then, MIP was studied for many different optimization algorithms and problem formulations. From these,  we can mention generalized equations \cite{MR1852505,MR1080795,MR1177240},  Newton methods \cite{MR2982712,MR2139225}, SQP methods \cite{MR1756894}, shape design problems \cite{MR1680928}, constrained Gauss-Newton methods \cite{MR1202003},  gradient projection methods \cite{MR1149080},  quasi-Newton methods \cite{MR912453,MR1049770,MR1150407}, and semi-smooth Newton methods \cite{MR2369205,MR2085262}. The convergence analysis of Algorithm \ref{BBa} will show  that, depending on the spectrum of the Hessian,   the sequence $\{ \| {\mathcal{G}}_k \|  \}_k$ can be nonmonotone. This is the main reason which distinguishes our analysis from that in \cite{MR912453,MR1049770,MR1150407}.

Our theoretical framework is supported by three optimizations problems with partial differential equations as constraints, including  linear elliptic (Poisson equation),  second-order linear hyperbolic (wave equation), and  semilinear parabolic equations (viscous Burger equation). We show that our results are applicable to these problems and report our numerical experience for them.

The rest of paper is organized as follows: In Section 2,  we first recall some concepts
from the spectral theory for bounded self-adjoint operators. We then deal with the global convergence analysis for strictly convex quadratic functions defined by bounded self-adjoint operators. Relying on this analysis, the local convergence of a class of nonlinear  functions is discussed. Section 3 is devoted for developing the mesh-independent principle for Algorithm \ref{BBa}. In Section 4, the PDE-constrained optimal control problems  alluded to above are investigated. Finally,  Section 5 presents the  numerical results.
\section{Convergence Analysis}
\label{section2}
In this section, we are concerned with the convergence analysis of Algorithm \ref{BBa}. The section is divided in two parts. The first part deals with strictly convex quadratic problems defined by bounded self-adjoint operators. In particular, the case in which the operator is a compact perturbation of the identity will  be treated in more detail.  Strictly convex quadratic problems are of great importance, not only in their own right,  but also as a model to study the behaviour of the algorithm for twice continuously Fr\'echet-differentiable functions in a neighbourhood of strong minima.  In the second part, relying on the analysis of the first part, we discuss the local convergence of Algorithm \ref{BBa} for twice continuously Fr\'echet-differentiable functions with Lipschitz continuous second derivatives.
\subsection{Quadratic Functions}
\subsubsection{General Case}
\label{secNoncompact}
In this subsection, we are concerned with the following quadratic programming  in an abstract Hilbert space $\mathcal{H}$
\begin{equation}
\label{QP}
\tag{QP}
\min_{u \in \mathcal{H}} \mathcal{F}(u):= \frac{1}{2}( \mathcal{A} u, u)-(b ,u) ,
\end{equation}
where $\mathcal{A}: \mathcal{H} \to \mathcal{H}$ is a bounded self-adjoint uniformly positive operator and $b \in \mathcal{H}$.
 In this case $\mathcal{G}_k:=\mathcal{G}(u_k)=\mathcal{A}u_k-b$ and it can easily be shown that
\begin{align}
\alpha^{BB1}_k = \frac{(\mathcal{S}_{k-1}, \mathcal{A}\mathcal{S}_{k-1})}{(\mathcal{S}_{k-1},\mathcal{S}_{k-1})}= \frac{(\mathcal{G}_{k-1}, \mathcal{A}\mathcal{G}_{k-1})}{(\mathcal{G}_{k-1},\mathcal{G}_{k-1})}, \label{e2}\\
\alpha^{BB2}_k = \frac{(\mathcal{S}_{k-1},\mathcal{A}^2\mathcal{S}_{k-1})}{(\mathcal{S}_{k-1},\mathcal{A}\mathcal{S}_{k-1})}=  \frac{(\mathcal{G}_{k-1},\mathcal{A}^2\mathcal{G}_{k-1})}{(\mathcal{G}_{k-1},\mathcal{A}\mathcal{G}_{k-1})}\label{e36},
\end{align}
where we have used that $\mathcal{S}_k = -\frac{1}{\alpha_k}\mathcal{G}_k$. We define the numerical range $\mathcal{W}(\mathcal{A})\subset \mathbb{R}$ of  $\mathcal{A}$ by
\begin{equation*}
\mathcal{W}(\mathcal{A}) := \{ (u, \mathcal{A}u) : u \in \mathcal{H},  \|u\| = 1 \}.
\end{equation*}
This set is convex and contains all the eigenvalues of $\mathcal{A}$. Moreover using  \eqref{e2}, \eqref{e36},  and the fact that
\begin{equation*}
\alpha^{BB2}_k = \frac{(\mathcal{G}_{k-1},\mathcal{A}^2\mathcal{G}_{k-1})}{(\mathcal{G}_{k-1},\mathcal{A}\mathcal{G}_{k-1})} = \frac{(\bar{\mathcal{G}}_{k-1},\mathcal{A} \bar{\mathcal{G}}_{k-1})}{(\bar{\mathcal{G}}_{k-1},\bar{\mathcal{G}}_{k-1})},
\end{equation*}
with $\bar{\mathcal{G}}_{k-1}:= \mathcal{A}^{\frac{1}{2}}\mathcal{G}_{k-1}$,  we infer that $\alpha^{BB1}_{k},  \alpha^{BB2}_{k} \in \mathcal{W}(\mathcal{A})$ for all $k \geq 1$. Therefore, if we define the  strictly positive constants  $\delta_{\inf}$ and  $\delta_{\sup}$ by
\begin{align*}
\delta_{\inf} := \inf\mathcal{W}(\mathcal{A}), & \quad
\delta_{\sup}  := \sup\mathcal{W}(\mathcal{A}),
\end{align*}
we can write
\begin{equation}\label{eqkk11}
 \alpha^{BB1}_{k},  \alpha^{BB2}_{k} \in [\delta_{\inf}, \delta_{\sup}]   \quad  \text{ for all }  k \geq 1.
\end{equation}
To exclude trivial cases we assume throughout that $\delta_{\inf} < \delta_{\sup}$.  For the following analysis we recall some facts from spectral theory. The spectrum $\sigma(\mathcal{A})$ of $\mathcal{A}$ is a closed strict subset of the interval $[ \delta_{\inf}, \delta_{\sup} ]$  with $\delta_{\inf}, \delta_{\sup} \in \sigma(\mathcal{A})$ and since $\mathcal{A}$ is a normal operator, we have $\overline{\mathcal{W}(\mathcal{A})}=\mathbf{conv}(\sigma(\mathcal{A}))=[ \delta_{\inf}, \delta_{\sup} ]$, where $\mathbf{conv}(S)$ denotes the convex hull of the set $S$.  Hence the interval $[ \delta_{\inf}, \delta_{\sup} ]$ is completely determined by the spectrum $\sigma(\mathcal{A})$.

Further, due to the spectral theorem   \cite{MR3112817,MR1787146}, there exists a unique spectral measure  $E$ on $\mathbb{R}$ which is supported on $\sigma(\mathcal{A})$, and whose range is the set of orthogonal projections in $\mathcal{H}$, such that
\begin{equation*}
\mathcal{\mathcal{A}} = \int_{\sigma(\mathcal{A})}\lambda\, dE_{\lambda}.
\end{equation*}
Moreover, for every bounded measurable function $f: \sigma(\mathcal{A}) \to \mathbb{R}$, the operator  $f(\mathcal{A})$ is defined by
\begin{equation}
\label{e18}
f(\mathcal{\mathcal{A}}) = \int_{\sigma(\mathcal{A})}f(\lambda) \, dE_{\lambda},
\end{equation}
and for every $x,y \in \mathcal{H}$ we have
\begin{equation}
\label{spect}
(f(\mathcal{A})x,y) = \int_{\sigma(\mathcal{A})}f(\lambda)d(E_{\lambda}x,y),
\end{equation}
where $d(E_{\lambda}x,y)$ stands for the integration with respect to the Borel measure $A \mapsto (E_{A}x,y)$ where $A \subseteq \sigma(\mathcal{A})$ is an arbitrary Borel set.

From \eqref{e1} we have
 \begin{equation}
 \label{e17}
 \mathcal{G}_{k+1} = \frac{1}{\alpha_k }(\alpha_k \mathcal{I}-\mathcal{A})\mathcal{G}_k   \quad \text{ for all  }  k = 1,2,\dots.
 \end{equation}
 For $\mathcal{G}_1 \in \mathcal{H}$ we find
\begin{equation*}
\mathcal{G}_1 = \int_{\sigma(\mathcal{A})} \, dE_{\lambda} \,\mathcal{G}_1,   \quad \text{ and  } \quad \|\mathcal{G}_1\|^2 = \int_{\sigma(\mathcal{A})} \, d(E_{\lambda}\mathcal{G}_1,\mathcal{G}_1).
\end{equation*}
Using \eqref{e18} and  \eqref{e17}, we have
\begin{equation*}
\mathcal{G}_2 = \frac{1}{\alpha_1} (\alpha_1-\mathcal{A})\mathcal{G}_1 = \int_{\sigma(\mathcal{A})}\frac{1}{\alpha_1} (\alpha_1-\lambda)\, dE_{\lambda}\mathcal{G}_1,
\end{equation*}
and, in a similar manner, we obtain
\begin{equation*}
\mathcal{G}_{k} =   \int_{\sigma(\mathcal{A})}\left[ \prod^{k-1}_{p=1}\left( \frac{\alpha_p-\lambda }{\alpha_p} \right) \right] \, dE_{\lambda}\mathcal{G}_1 \quad \text{ for every }  k = 1,2,\dots.
\end{equation*}
where $\prod^{0}_{p=1} = 1$.  Moreover, we can write for  $k = 1,2,\dots$
\begin{equation}
\label{e114}
\begin{split}
\|\mathcal{G}_{k+1}\|^2 =& \left(\frac{1}{\alpha_k }(\alpha_k \mathcal{I}-\mathcal{A})\mathcal{G}_k, \frac{1}{\alpha_k }(\alpha_k \mathcal{I}-\mathcal{A})\mathcal{G}_k\right) \\
&= \left(\frac{1}{\alpha^2_k }(\alpha_k \mathcal{I}-\mathcal{A})^2\mathcal{G}_k, \mathcal{G}_k\right)= \int_{\sigma(\mathcal{A})}\left( \frac{\alpha_k-\lambda}{\alpha_k} \right)^2 d(E_{\lambda}\mathcal{G}_k, \mathcal{G}_k).
\end{split}
\end{equation}
Similarly, we have
\begin{equation}
\begin{split}
\label{e19}
\|\mathcal{G}_{k}\|^2 &= \left(  \left[ \prod^{k-1}_{p=1}\left( \frac{\alpha_p-\mathcal{A}}{\alpha_p} \right) \right] \mathcal{G}_1,   \left[ \prod^{k-1}_{p=1}\left( \frac{\alpha_p-\mathcal{A}}{\alpha_p} \right)\right] \mathcal{G}_1 \right)\\
& = \left( \left[ \prod^{k-1}_{p=1}\left( \frac{\alpha_p-\mathcal{A}}{\alpha_p} \right)^2 \right]\mathcal{G}_1 , \mathcal{G}_1 \right)= \int_{\sigma(\mathcal{A})} \left[ \prod^{k-1}_{p=1}\left( \frac{\alpha_p-\lambda}{\alpha_p} \right)^2 \right]\, d(E_{\lambda}\mathcal{G}_1, \mathcal{G}_1).
\end{split}
\end{equation}
We define $\gamma_{\mathcal{A}}:= \frac{\delta_{\sup}-\delta_{\inf}}{\delta_{\inf}} $ and  $\rho_{\mathcal{A}} : =\frac{\delta_{\sup}-\delta_{\inf}}{\delta_{\sup}}$. These quantities will be used frequently in the proofs.  First we investigate the special case in which $\delta_{\sup}< 2\delta_{\inf}$. In this case, it can be shown that $\gamma_{\mathcal{A}}<1$.
 \begin{theorem}
 \label{Theo4}
Let $\delta_{\sup}< 2\delta_{\inf}$.  Then the sequence $\{u_k\}_k$ generated by Algorithm \ref{BBa} converges $Q$-linearly to the solution $u^*$  of  \eqref{QP} with the rate $\gamma_{\mathcal{A}}$.
\end{theorem}
\begin{proof}
Recall that by \eqref{e114},  we have for $k\geq 1$ that
\begin{equation}
\label{e19a}
\begin{split}
\|\mathcal{G}_{k+1}\|^2= \int_{\sigma(\mathcal{A})} \left( \frac{\alpha_k-\lambda}{\alpha_k} \right)^2 \, d(E_{\lambda}\mathcal{G}_k, \mathcal{G}_k).
\end{split}
\end{equation}
Since $\delta_{\sup} < 2\delta_{\inf}$, it follows for every $k\geq 1$  and $\lambda \in \sigma(\mathcal{A})$ that
\begin{equation}
\label{e20}
\abs{\frac{\alpha_k-\lambda}{\alpha_k}}^2 \leq \left( \frac{\delta_{\sup}-\delta_{\inf}}{\delta_{\inf}} \right)^2=\gamma^2_{\mathcal{A}} < 1.
\end{equation}
Using \eqref{e19a} and \eqref{e20}, we obtain
\begin{equation}
\label{e110}
\begin{split}
\|\mathcal{G}_{k+1}\|^2 \leq & \left( \frac{ \delta_{\sup}-\delta_{\inf}}{\delta_{\inf}}\right)^{2} \int_{\sigma(\mathcal{A})} \, d(E_{\lambda}\mathcal{G}_k, \mathcal{G}_k) =  \gamma^2_{\mathcal{A}} \norm{\mathcal{G}_k}^2  \quad     \text{ for every } k\geq 1.
\end{split}
\end{equation}
Therefore, we can conclude that
\begin{equation*}
\|\mathcal{G}_{k+1}\|^2 \leq  \gamma^{2k}_{\mathcal{A}} \norm{\mathcal{G}_1}^2  \quad     \text{ for every } k\geq 1,
\end{equation*}
and this completes the proof.
\qed \end{proof}
If we lift the condition $\delta_{\sup} <  2\delta_{\inf}$, we attain the following result.
\begin{theorem}
\label{RConvergence}
Let $\{  u_k \}_k$ be the sequence generated by Algorithm \ref{BBa} for finding the global minimum $u^*$ of $\eqref{QP}$. Then either $u_k = u^*$ for a finite $k$, or the sequence  $\{   u_k \}_k$ converges $R$-linearly to $u^*$.
\end{theorem}
The proof requires several lemmas and will be given in the remainder of this subsection.  First, we need to define some quantities that will be used throughout the results. For any given $\eta>0$, we denote $a_{i}: = \delta_{\inf}+(i-1)\eta$ for every $i $ with $1 \leq i \leq n^{u}_{\eta}$, and
\begin{equation*}
  b_{i}: =  \begin{cases}
 \delta_{\inf}+i\eta  & \text{ for } 1 \leq i \leq n^{u}_{\eta}-1,\\
 \delta_{\sup}           & \text{ for }  i = n^{u}_{\eta},
\end{cases}
\end{equation*}
where $n^{u}_{\eta}: =\lfloor\frac{\delta_{\sup}-\delta_{\inf}}{\eta} \rfloor +1$. Then, clearly,  $b_{i-1} = a_{i}$ for every $i =2,\dots, n^{u}_{\eta}$ and we can define the following family of pairwise disjoint intervals
\begin{equation}
\label{e30}
I_i =
\begin{cases}
[a_i, b_i)  & \text{ for } 1 \leq i \leq n^{u}_{\eta}-1,\\
[a_{n^{u}_{\eta}}, b_{n^{u}_{\eta}}] & \text{ for }  i = n^{u}_{\eta}.
\end{cases}
\end{equation}
By construction it is clear that $| I_i| \leq \eta$ for every $i = 1,\dots,n^{u}_{\eta}$, and
\begin{equation}
\label{e126}
\sigma(\mathcal{A}) \subseteq [\delta_{\inf} ,\delta_{\sup}] = \bigcup^{n^{u}_{\eta}}_{i = 1} I_i.
\end{equation}
For  $i=1,\cdots,n^{u}_{\eta}$, we define
\begin{equation}\label{e29}
(g^{k+1}_{i})^2: = \int_{I_i} \left[ \prod^{k}_{p=1}\left( \frac{\alpha_p-\lambda}{\alpha_p} \right)^2 \right]\, d(E_{\lambda}\mathcal{G}_1, \mathcal{G}_1),
\end{equation}
and attain
\begin{equation}
\label{e21}
\begin{split}
\|\mathcal{G}_{k+1}\|^2 &= \int_{\sigma(\mathcal{A})} \left[ \prod^{k}_{p=1}\left( \frac{\alpha_p-\lambda}{\alpha_p} \right)^2 \right]  \, d(E_{\lambda}\mathcal{G}_{1}, \mathcal{G}_{1}) = \int^{\delta_{\sup}}_{\delta_{\inf}} \left[ \prod^{k}_{p=1}\left( \frac{\alpha_p-\lambda}{\alpha_p} \right)^2 \right]  \, d(E_{\lambda}\mathcal{G}_{1}, \mathcal{G}_{1}) \\
&= \sum^{n^{u}_{\eta}}_{i = 1} \int_{I_i}  \left[ \prod^{k}_{p=1}\left( \frac{\alpha_p-\lambda}{\alpha_p} \right)^2 \right] \, d(E_{\lambda}\mathcal{G}_{1}, \mathcal{G}_{1}) = \sum^{n^{u}_{\eta}}_{i=1} (g^{k+1}_i )^2.
\end{split}
\end{equation}
Moreover, we define
\begin{equation}
\label{e109}
G(k,\ell):= \sum^{\ell}_{i=1} (g^k_{i})^2   \quad \text{ for every }  k\geq 1  \text{ and } 1 \leq \ell \leq n^{u}_{\eta},
\end{equation}
where $n^{u}_{\eta}$ is defined with respect to an interval length $\eta > 0$, and  $g^k_i$ is defined in \eqref{e29}. Then it is clear  that
\begin{equation*}
G(k,n^{u}_{\eta})=\sum^{n^{u}_{\eta}}_{i=1} (g^k_{i})^2  =\|\mathcal{G}_{k}\|^2  \quad \text{ for every } k \geq 1.
\end{equation*}
In the following lemma we show that there exists an index $n^{l}_{\eta}$ such that the sequences $\{g^{k}_{i}\}_k$ with $1\leq i \leq n^{l}_{\eta}$ converge to zero $Q$-linearly as $k$ tends to infinity.
\begin{lemma}
\label{lem3}
For every $\eta \in (0, \rho_{\mathcal{A}}\delta_{\inf}]$, there exists a positive integer  $n^{l}_{\eta}$ with $ 1\leq n^{l}_{\eta} \leq n^{u}_{\eta} $ such that for every $ 1\leq  i  \leq n^{l}_{\eta}$ ,  the sequences $\{g^{k}_{i}\}_k$ converge to zero $Q$-linearly with the factor $\rho_{\mathcal{A}}$ as $k$ tends to infinity.
\end{lemma}
\begin{proof}
Choose $  n^{l}_{\eta} \in \{1,\cdots,n^{u}_{\eta}\}$ as the largest integer such that
\begin{equation*}
\bigcup^{n^{l}_{\eta}}_{i=1} I_i \subseteq [\delta_{\inf},  (1+\rho_{\mathcal{A}})\delta_{\inf}].
\end{equation*}
Observe that this is well-defined since  $\eta \leq \rho_{\mathcal{A}}\delta_{\inf}$. Moreover, for every  $\lambda \in I_i$  with $1\leq  i \leq n^{l}_{\eta}$ and every $p \geq 1$ we have the following two cases:
\begin{enumerate}
\item  If  $\alpha_p -\lambda \geq 0$, then we have 
\begin{equation*}
\abs{ \frac{\alpha_p-\lambda}{\alpha_p}} = \frac{\alpha_p-\lambda}{\alpha_p} \leq \rho_{\mathcal{A}} <1.
\end{equation*}  
\item  If  $\alpha_p -\lambda < 0$, then clearly both of  $\alpha_p$ and $\lambda$ belong to  $[\delta_{\inf},  (1+\rho_{\mathcal{A}})\delta_{\inf}]$ and  we can write
\begin{equation*}
\abs{ \frac{\alpha_p-\lambda}{\alpha_p}} = \frac{\lambda-\alpha_p}{\alpha_p} \leq  \frac{(1+\rho_{\mathcal{A}})\delta_{\inf}- \delta_{\inf}}{\delta_{\inf}} = \rho_{\mathcal{A}}.
\end{equation*}  
\end{enumerate}
Therefore, we obtain
\begin{equation}
\label{e136}
\begin{split}
(g^{k+1}_{i})^2   & = \int_{I_i} \left[ \prod^{k}_{p=1}\left( \frac{\alpha_p-\lambda}{\alpha_p} \right)^2 \right]\, d(E_{\lambda}\mathcal{G}_1, \mathcal{G}_1) \leq \int_{I_i}\left( \frac{\alpha_k-\lambda}{\alpha_k} \right)^2 \left[ \prod^{k-1}_{p=1}\left( \frac{\alpha_p-\lambda}{\alpha_p} \right)^2 \right]\, d(E_{\lambda}\mathcal{G}_1, \mathcal{G}_1)\\
& \leq \rho_{\mathcal{A}}^{2}\int_{I_i} \left[ \prod^{k-1}_{p=1}\left( \frac{\alpha_p-\lambda}{\alpha_p} \right)^2 \right]\, d(E_{\lambda}\mathcal{G}_1, \mathcal{G}_1)=\rho_{\mathcal{A}}^{2}(g^{k}_{i})^2.
\end{split}
\end{equation}
This concludes the proof.
\qed \end{proof}
Next we prove the following useful lemmas, which will be used later.
 \begin{lemma}
 \label{alem1}
For any interval length $\eta \in (0,\frac{\delta_{\inf}}{2})$, every integer $\ell$ with  $ n^{l}_{\eta} \leq \ell  \leq  n^{u}_{\eta}$,  and $k\geq 1$,  the following property holds:

If the following condition
\begin{equation}
\label{e3}
G(k+j,\ell) \leq \overline{\zeta}\|\mathcal{G}_k\|^2  \quad \text{ for all } j\geq \overline{r}
\end{equation}
holds for some positive  $\overline{r} \in \mathbb{N}$ and $\overline{\zeta} \in \mathbb{R}_+$, then there exists an integer $\hat{j} \in \{ \overline{r},\cdots ,\overline{r}+\Theta +1\}$ such that
\begin{equation*}
(g^{k+\hat{j}}_{\ell+1})^2 \leq 2\overline{\zeta}\| \mathcal{G}_k\|^2,
\end{equation*}
where $\Theta =  \Theta (\overline{\zeta},\overline{r}):=\ceil[\Big]{\frac{\log(2\overline{\zeta}{\gamma_{\mathcal{A}}}^{-2(\overline{r}+1)})}{2\log c}}$ with $c := \max \{ \rho_{\mathcal{A}}, \frac{1}{2}+\frac{\eta}{\delta_{\inf}} \} $.
\end{lemma}
\begin{proof}
Supposing that
\begin{equation}
\label{e111}
(g^{k+j}_{\ell +1})^2 > 2\overline{\zeta}\| \mathcal{G}_k\|^2  \quad \text{ for all  } j \in \{\overline{r}, \cdots , \overline{r} +\Theta\},
\end{equation}
we will show that
\begin{equation*}
(g^{k+\overline{r}+\Theta+1}_{\ell +1})^2 \leq  2\overline{\zeta}\| \mathcal{G}_k\|^2.
\end{equation*}
Due to \eqref{e29}, we have for every $k\geq 1$ that
\begin{equation}
\begin{split}
\label{e112}
(g^{k+\overline{r}+1}_{\ell +1})^2 &= \int_{I_{\ell +1}}\left[ \prod^{k+\overline{r}}_{p =1}\left( \frac{\alpha_{p}-\lambda}{\alpha_{p}} \right)^2 \right] \, d(E_{\lambda}\mathcal{G}_{1}, \mathcal{G}_{1})\\
& =  \int_{I_{\ell +1}}\left[ \prod^{k+\overline{r}}_{p =k}\left( \frac{\alpha_{p}-\lambda}{\alpha_{p}} \right)^2 \right]\left[ \prod^{k-1}_{p =1}\left( \frac{\alpha_{p}-\lambda}{\alpha_{p}} \right)^2 \right]  \, d(E_{\lambda}\mathcal{G}_{1}, \mathcal{G}_{1}) \\
& \leq \left(\frac{\delta_{\sup}-\delta_{\inf}}{\delta_{\inf}} \right)^{2(\overline{r}+1)}\int_{I_{\ell +1}}\left[ \prod^{k-1}_{p =1}\left( \frac{\alpha_{p}-\lambda}{\alpha_{p}} \right)^2 \right]  \, d(E_{\lambda}\mathcal{G}_{1}, \mathcal{G}_{1}) \\
&= \gamma_{\mathcal{A}}^{2(\overline{r}+1)}(g^{k}_{\ell +1})^2\leq \gamma_{\mathcal{A}}^{2(\overline{r}+1)}\| \mathcal{G}_k\|^2.
\end{split}
\end{equation}
Due to Algorithm \ref{BBa},  for every $j \in \{\overline{r}, \cdots , \overline{r} +\Theta\}$ we have one of the cases $\alpha_{k+j} = \alpha^{BB1}_{k+j}$  or $\alpha_{k+j} = \alpha^{BB2}_{k+j}$. Further, using \eqref{e17}, the fact that $\mathcal{A}$ is self-adjoint, and  the spectral property \eqref{spect}, we have for every $k\geq1$ and $q=0,1,2$, that
\begin{equation}
\label{e23b}
\begin{split}
&(\mathcal{G}_{k},\mathcal{A}^q\mathcal{G}_{k}) = \left(  \left[ \prod^{k-1}_{j=1}\left( \frac{\alpha_p-\mathcal{A}}{\alpha_p} \right) \right] \mathcal{G}_1,   \mathcal{A}^q\left[ \prod^{k-1}_{p=1}\left( \frac{\alpha_p-\mathcal{A}}{\alpha_p} \right)\right] \mathcal{G}_1 \right)\\
& = \left( \mathcal{A}^q \left[ \prod^{k-1}_{p=1}\left( \frac{\alpha_p-\mathcal{A}}{\alpha_p} \right)^2 \right] \mathcal{G}_1, \mathcal{G}_1 \right) = \int_{\sigma(\mathcal{A})}  \lambda^q\left[ \prod^{k-1}_{p=1}\left( \frac{\alpha_p-\lambda}{\alpha_p} \right)^2 \right]  \, d(E_{\lambda}\mathcal{G}_{1}, \mathcal{G}_{1})\\
& = \int_{\bigcup^{n^{u}_{\eta}}_{i = 1} I_i}\lambda^q\left[ \prod^{k-1}_{p=1}\left( \frac{\alpha_p-\lambda}{\alpha_p} \right)^2 \right]  \, d(E_{\lambda}\mathcal{G}_{1}, \mathcal{G}_{1}).
\end{split}
\end{equation}
Now, by using \eqref{e2}, \eqref{e36}, and  \eqref{e23b},  we can write for $j \in \{\overline{r}, \cdots , \overline{r} +\Theta\}$ that
\begin{equation}
\label{e23}
\alpha^{BB1}_{k+j+1} = \frac{(\mathcal{G}_{k+j},\mathcal{A}\mathcal{G}_{k+j})}{\mathcal{G}_{k+j},\mathcal{G}_{k+j})} =\frac{\int_{\bigcup^{n^{u}_{\eta}}_{i = 1} I_i}\lambda\left[ \prod^{k+j-1}_{p=1}\left( \frac{\alpha_p-\lambda}{\alpha_p} \right)^2 \right]  \, d(E_{\lambda}\mathcal{G}_{1}, \mathcal{G}_{1}) }{\int_{\bigcup^{n^{u}_{\eta}}_{i = 1} I_i}\left[ \prod^{k+j-1}_{p=1}\left( \frac{\alpha_p-\lambda}{\alpha_p} \right)^2 \right]  \, d(E_{\lambda}\mathcal{G}_{1}, \mathcal{G}_{1})},
\end{equation}
and
\begin{equation}
\label{e23a}
\alpha^{BB2}_{k+j+1}= \frac{(\mathcal{G}_{k+j},\mathcal{A}^2\mathcal{G}_{k+j})}{\mathcal{G}_{k+j},\mathcal{A}\mathcal{G}_{k+j})} = \frac{\int_{\bigcup^{n^{u}_{\eta}}_{i = 1} I_i}\lambda^2\left[ \prod^{k+j-1}_{p=1}\left( \frac{\alpha_p-\lambda}{\alpha_p} \right)^2 \right]  \, d(E_{\lambda}\mathcal{G}_{1}, \mathcal{G}_{1}) }{\int_{\bigcup^{n^{u}_{\eta}}_{i = 1} I_i}\lambda\left[ \prod^{k+j-1}_{p=1}\left( \frac{\alpha_p-\lambda}{\alpha_p} \right)^2 \right]  \, d(E_{\lambda}\mathcal{G}_{1}, \mathcal{G}_{1})}.
\end{equation}
Moreover, due to \eqref{e29} and \eqref{e3}, we have
\begin{equation}
\label{e24}
\begin{split}
 &\int_{\bigcup^{\ell}_{i = 1} I_i}\left[ \prod^{k+j-1}_{p=1}\left( \frac{\alpha_p-\lambda}{\alpha_p} \right)^2 \right]  \, d(E_{\lambda}\mathcal{G}_{1}, \mathcal{G}_{1}) = \sum^{\ell}_{i = 1}\int_{I_i}\left[ \prod^{k+j-1}_{p=1}\left( \frac{\alpha_p-\lambda}{\alpha_p} \right)^2 \right]  \, d(E_{\lambda}\mathcal{G}_{1}, \mathcal{G}_{1}) \\
 & = \sum^{\ell}_{i=1} (g^{k+j}_i)^2 = G(k+j,\ell)\leq  \overline{\zeta}\|\mathcal{G}_k\|^2  \quad \text{ for all   }   j \in \{ \overline{r},\cdots ,\overline{r}+\Theta\}.
\end{split}
\end{equation}
For every $\lambda \in \bigcup^{\ell}_{i = 1} I_i $, we have $\lambda \leq a_{\ell+1}$. Thus, by \eqref{e24}, we can write
\begin{equation}
\label{e24a}
\begin{split}
&\int_{\bigcup^{\ell}_{i = 1} I_i}\lambda\left[ \prod^{k+j-1}_{p=1}\left( \frac{\alpha_p-\lambda}{\alpha_p} \right)^2 \right]  \, d(E_{\lambda}\mathcal{G}_{1}, \mathcal{G}_{1})\leq  a_{\ell+1} \int_{\bigcup^{\ell}_{i = 1} I_i}\left[ \prod^{k+j-1}_{p=1}\left( \frac{\alpha_p-\lambda}{\alpha_p} \right)^2 \right]  \, d(E_{\lambda}\mathcal{G}_{1}, \mathcal{G}_{1})  \\
 &= a_{\ell+1}  G(k+j,\ell)\leq  a_{\ell+1} \overline{\zeta}\|\mathcal{G}_k\|^2  \quad \text{ for all   }   j \in \{ \overline{r},\cdots ,\overline{r}+\Theta\}.
\end{split}
\end{equation}
From \eqref{e23} and \eqref{e24}, we obtain
\begin{equation}
\label{e25}
\frac{ a_{\ell+1}\mathcal{Z} }{ \overline{\zeta}\|\mathcal{G}_k\|^2  + \mathcal{Z}} \leq \alpha^{BB1}_{k+j+1}  \leq \delta_{\sup}  \quad \text{ for all }  j \in \{ \overline{r},\cdots ,\overline{r}+\Theta\},
\end{equation}
where $\mathcal{Z}:=\int_{\bigcup^{n^{u}_{\eta}}_{i = \ell+1} I_i}\left[ \prod^{k+j-1}_{p=1}\left( \frac{\alpha_p-\lambda}{\alpha_p} \right)^2 \right]  \, d(E_{\lambda}\mathcal{G}_{1}, \mathcal{G}_{1})$. From \eqref{e23a}, \eqref{e24a}, and the fact that $\lambda \geq  a_{\ell +1}$  for every $\lambda \in \bigcup^{n^{u}_{\eta}}_{i = \ell+1} I_i $, it follows that
\begin{equation}
\label{e25a}
\begin{split}
&\frac{ a_{\ell+1}\mathcal{Z}}{\overline{\zeta}\|\mathcal{G}_k\|^2  +\mathcal{Z}}= \frac{ a^2_{\ell+1}\mathcal{Z}}{a_{\ell +1}\overline{\zeta}\|\mathcal{G}_k\|^2  +\alpha_{\ell +1}\mathcal{Z}} \leq \frac{ a_{\ell+1}\int_{\bigcup^{n^{u}_{\eta}}_{i = \ell+1} I_i}\lambda\left[ \prod^{k+j-1}_{p=1}\left( \frac{\alpha_p-\lambda}{\alpha_p} \right)^2 \right]  \, d(E_{\lambda}\mathcal{G}_{1}, \mathcal{G}_{1})}{a_{\ell +1}\overline{\zeta}\|\mathcal{G}_k\|^2  + \int_{\bigcup^{n^{u}_{\eta}}_{i = \ell+1} I_i}\lambda\left[ \prod^{k+j-1}_{p=1}\left( \frac{\alpha_p-\lambda}{\alpha_p} \right)^2 \right]  \, d(E_{\lambda}\mathcal{G}_{1}, \mathcal{G}_{1}) }\\\leq &  \alpha^{BB2}_{k+j+1}  \leq \delta_{\sup}  \quad \text{ for all }  j \in \{ \overline{r},\cdots ,\overline{r}+\Theta\}.
\end{split}
\end{equation}
Now, using the fact  that
\begin{equation*}
\mathcal{Z} \geq  \int_{ I_{\ell +1}}\left[ \prod^{k+j-1}_{p=1}\left( \frac{\alpha_p-\lambda}{\alpha_p} \right)^2 \right]  \, d(E_{\lambda}\mathcal{G}_{1}, \mathcal{G}_{1})=  (g^{k+j}_{\ell +1})^2,
\end{equation*}
 and  by \eqref{e111}, \eqref{e25}, and \eqref{e25a},  we infer that for a chosen $\alpha_{k+j+1} = \alpha^{BB1}_{k+j+1}$ or  $\alpha_{k+j+1} = \alpha^{BB2}_{k+j+1}$ that
 \begin{equation}
 \label{e27}
\frac{2}{3} a_{\ell+1} =\frac{ a_{\ell+1}\mathcal{Z}}{\frac{1}{2}\mathcal{Z}  +\mathcal{Z}}  \leq \frac{ a_{\ell+1}\mathcal{Z}}{\frac{1}{2}(g^{k+j}_{\ell+1})^2  +\mathcal{Z}}  \leq\frac{ a_{\ell+1}\mathcal{Z}}{\overline{\zeta}\|\mathcal{G}_k\|^2  +\mathcal{Z}} \leq \alpha_{k+j+1}\leq\delta_{\sup}  \quad \text{ for all } j \in \{ \overline{r},\cdots ,\overline{r}+\Theta\}.
 \end{equation}
Now for $\lambda \in [a_{\ell+1}, b_{\ell +1}]$ and for $j \in \{ \overline{r},\cdots ,\overline{r}+\Theta\}$ we have the following two cases:
\begin{enumerate}
\item If $\alpha_{k+j+1}- \lambda \geq 0$, then by \eqref{eqkk11} we have
\begin{equation*}
  \abs{1- \frac{\lambda}{\alpha_{k+j+1}}} = \left( 1- \frac{\lambda}{\alpha_{k+j+1}} \right) \leq \rho_{\mathcal{A}}  <  1.
\end{equation*}
\item If $\alpha_{k+j+1}- \lambda <0$, then by \eqref{e27} and  using the fact that $ \lambda \leq  b_{\ell+1} \leq a_{\ell+1}+\eta$ for $\lambda \in I_{\ell+1}$, we obtain
\begin{equation*}
\begin{split}
\abs{1- \frac{\lambda}{\alpha_{k+j+1}}} &=\left(\frac{\lambda}{\alpha_{k+j+1}}-1 \right) \leq \left( \frac{b_{\ell +1}}{\alpha_{k+j+1}}-1\right) \leq \left( \frac{a_{\ell +1}+\eta}{\alpha_{k+j+1}}-1\right)\\
& \leq \frac{3}{2}+\frac{\eta}{\alpha_{k+j+1}}-1 \leq  \frac{1}{2}+\frac{\eta}{\delta_{\inf}} < 1,
\end{split}
\end{equation*}
where in the last inequality we have used that $\eta < \frac{\delta_{\inf}}{2}$.
\end{enumerate}
Hence, by the fact that  $c = \max \{ \rho_{\mathcal{A}}, \frac{1}{2}+\frac{\eta}{\delta_{\inf}} \} $,  we have for every $j \in \{ \overline{r},\cdots ,\overline{r}+\Theta\}$  and $\lambda \in [a_{\ell+1}, b_{\ell +1}]$ that
\begin{equation}
\label{e28}
\abs{ 1- \frac{\lambda}{\alpha_{k+j+1}}} \leq  c <1.
\end{equation}
Finally, by using \eqref{e29} and \eqref{e28}  we obtain for  every $j \in \{ \overline{r},\cdots ,\overline{r}+\Theta\}$ that
\begin{equation}
\begin{split}
\label{e29aaa}
(g^{k+j+2}_{\ell +1})^2 &= \int_{I_{\ell +1}}\left[ \prod^{k+j+1}_{p =1}\left( \frac{\alpha_{p}-\lambda}{\alpha_{p}} \right)^2 \right] \, d(E_{\lambda}\mathcal{G}_{1}, \mathcal{G}_{1})\\
& =  \int_{I_{\ell +1}}\abs{ 1-\frac{\lambda}{\alpha_{k+j+1}}}^2 \left[ \prod^{k+j}_{p =1}\left( \frac{\alpha_{p}-\lambda}{\alpha_{p}} \right)^2 \right]  \, d(E_{\lambda}\mathcal{G}_{1}, \mathcal{G}_{1}) \\
& \leq  c^2\int_{I_{\ell +1}}\left[ \prod^{k+j}_{p =1}\left( \frac{\alpha_{p}-\lambda}{\alpha_{p}} \right)^2 \right]  \, d(E_{\lambda}\mathcal{G}_{1}, \mathcal{G}_{1}) =  c^2(g^{k+j+1}_{\ell +1})^2.
\end{split}
\end{equation}
Using  \eqref{e112}, \eqref{e29aaa}, and the definitions of $\Theta$, we obtain
\begin{equation*}
\begin{split}
(g^{k+\overline{r}+\Theta+1 }_{\ell +1})^2 \leq c^{2\Theta}(g_{\ell+1}^{k+\overline{r}+1})^2\leq  c^{2\Theta} \gamma_{\mathcal{A}}^{2(\overline{r}+1)}(g^{k}_{\ell +1})^2 \leq  c^{2\Theta}\gamma_{\mathcal{A}}^{2(\overline{r}+1)}\| \mathcal{G}_k\|^2\leq 2\overline{\zeta}\|\mathcal{G}_{k}\|^2,
\end{split}
\end{equation*}
and the proof is complete.
\qed \end{proof}
\begin{lemma}
\label{lem4}
Let $\delta_{\sup} \geq 2\delta_{\inf}$.  Moreover,  assume that for any $\eta \in (0,\frac{\delta_{\inf}}{2})$,  integer $\ell$ with  $ n^{l}_{\eta} \leq \ell  \leq  n^{u}_{\eta}$,  and $k\geq 1$, there exist $r_{\ell} \in \mathbb{N}$  and $\zeta_{\ell} \in \mathbb{R}_+$ such that the condition
\begin{equation}
\label{e3e}
G(k+j,\ell) \leq \zeta_{\ell}\|\mathcal{G}_k\|^2  \quad \text{ for all } j\geq r_{\ell}
\end{equation}
 holds. Then we show that for the choice of
\begin{equation*}
\zeta_{\ell+1}:=(1+2\gamma_{\mathcal{A}}^4)\zeta_{\ell},  \quad  \text{ and }  \quad  r_{\ell+1}:= r_{\ell}+\Theta_{\ell}+1,
\end{equation*}
with $\Theta_{\ell}:=\Theta( \zeta_{\ell}, r_{\ell})$  defined as in Lemma \ref{alem1}, we have
\begin{equation*}
G(k+j,\ell +1) \leq \zeta_{\ell+1} \| \mathcal{G}_k\|^2   \quad \text{ for all } j\geq r_{\ell +1}.
\end{equation*}
\end{lemma}
\begin{proof}
First, observe that
\begin{equation*}
G(k+j,\ell+1) = G(k+j,\ell)+(g^{k+j}_{\ell+1})^2.
\end{equation*}
Therefore, using \eqref{e3e} we only need to show that for every $ j \geq r_{\ell+1}$
\begin{equation}
\label{e32c}
(g^{k+j}_{\ell+1})^2 \leq 2 \gamma_{\mathcal{A}}^4\zeta_{\ell}\|\mathcal{G}_k\|^2.
\end{equation}
Due to Lemma \ref{alem1} for  $\overline{\zeta} = \zeta_{\ell}$ and $\overline{r}=r_{\ell}$, there exists an integer $j_1 \in \{r_{\ell},\cdots,r_{\ell}+\Theta_{\ell}+1\}$ such that
\begin{equation*}
(g^{k+j_1}_{\ell+1})^2 \leq 2\zeta_{\ell}\|\mathcal{G}_k\|^2.
\end{equation*}
Now let us introduce a shifting variable which we initialize by $j_s=j_1$. Assume that $j_2\geq j_s= j_1$ is an  index, for which we have
\begin{equation}
\label{e32a}
(g^{k+j}_{\ell+1})^2 \leq  2\zeta_{\ell}\|\mathcal{G}_k\|^2 \quad \text{ for  all }  j_1 \leq   j \leq j_2,
\end{equation}
and
\begin{equation}
\label{e32aa}
(g^{k+j_2+1}_{\ell+1})^2 > 2\zeta_{\ell}\|\mathcal{G}_k\|^2.
\end{equation}
Note that if this case does not arise,  clearly, \eqref{e32c} holds for all $j\geq j_s = j_1$ and since $\gamma_{\mathcal{A}} \geq 1$ the proof is finished.   Further,  we can write
\begin{equation}
\label{e113}
(g^{k+j+1}_{\ell+1})^2 >  2\zeta_{\ell}\|\mathcal{G}_k\|^2 \quad \text{ for  all }  j_2 \leq   j \leq j_3-2,
\end{equation}
where $j_3\geq j_2+2$ is the first integer greater than $j_2$ for which we have
\begin{equation}
\label{e131}
 (g^{k+j_3}_{\ell +1})^2 \leq 2\zeta_{\ell}\|\mathcal{G}_k\|^2.
\end{equation}
Existence of such an index is justified using  Lemma \ref{alem1} for $\overline{r}=j_2$ and $\overline{\zeta} = \zeta_{\ell}$.  Now, by \eqref{e113} and using the same argument as in the proof of Lemma \ref{alem1}, where we have shown that from \eqref{e111} implies \eqref{e27}, we can infer that
\begin{equation*}
\frac{2}{3} a_{\ell+1} \leq \alpha_{k+j+2} \leq \delta_{\sup}   \quad  \text{ for every }   j_2 \leq j  \leq  j_3-2.
\end{equation*}
Continuing the argument from the proof of Lemma \ref{alem1} we infer that
\begin{equation}
\label{e31}
(g^{k+j+3}_{\ell +1})^2 \leq  c^2  (g^{k+j+2}_{\ell +1})^2  \quad  \text{ for every }   j_2 \leq j  \leq  j_3-2,
\end{equation}
where $c := \max \{ \rho_{\mathcal{A}} ,  \frac{1}{2}+\frac{\eta}{\delta_{\inf}} \} <1$. Finally, using \eqref{e29} and  \eqref{e32a},  we have for $r =1,2$
\begin{equation}
\begin{split}
\label{e32}
(g^{k+j_2+r}_{\ell +1})^2 &= \int_{I_{\ell +1}}\left[ \prod^{k+j_2+r-1}_{p =1}\left( \frac{\alpha_{p}-\lambda}{\alpha_{p}} \right)^2 \right] \, d(E_{\lambda}\mathcal{G}_{1}, \mathcal{G}_{1})\\
& =  \int_{I_{\ell +1}}\left[ \prod^{k+j_2+r-1}_{p =k+j_2}\left( \frac{\alpha_{p}-\lambda}{\alpha_{p}} \right)^2 \right]\left[ \prod^{k+j_2-1}_{p =1}\left( \frac{\alpha_{p}-\lambda}{\alpha_{p}} \right)^2 \right]  \, d(E_{\lambda}\mathcal{G}_{1}, \mathcal{G}_{1}) \\
& =  \int_{I_{\ell +1}}\left[ \prod^{r}_{p =1}\left( \frac{\alpha_{k+j_2+p-1}-\lambda}{\alpha_{k+j_2+p-1}} \right)^2 \right]\left[ \prod^{k+j_2-1}_{p =1}\left( \frac{\alpha_{p}-\lambda}{\alpha_{p}} \right)^2 \right]  \, d(E_{\lambda}\mathcal{G}_{1}, \mathcal{G}_{1}) \\
& \leq \left(\frac{\delta_{\sup}-\delta_{\inf}}{\delta_{\inf}} \right)^{2r}\int_{I_{\ell +1}}\left[ \prod^{k+j_2-1}_{p =1}\left( \frac{\alpha_{p}-\lambda}{\alpha_{p}} \right)^2 \right]  \, d(E_{\lambda}\mathcal{G}_{1}, \mathcal{G}_{1})  =  \gamma^{2r}_{\mathcal{A}} (g^{k+j_2}_{\ell +1})^2.
\end{split}
\end{equation}
Now since $c<1$ and $\gamma_{\mathcal{A}}\geq 1$ due the fact that $\delta_{\sup} \geq 2\delta_{\inf}$ , we obtain from \eqref{e32a}, \eqref{e31}, and  \eqref{e32} that
\begin{equation}
\label{e32b}
 (g^{k+j+3}_{\ell +1})^2 \leq  \gamma_{\mathcal{A}}^4(g^{k+j_2}_{\ell +1})^2 \leq  2\zeta_{\ell} \gamma_{\mathcal{A}}^4 \|\mathcal{G}_k\|^2 \quad  \text{ for every } j_2-2 \leq j  \leq  j_3-2,
\end{equation}
and as a consequence, we obtain
\begin{equation}
\label{e32d}
 (g^{k+j}_{\ell +1})^2 \leq  2\zeta_{\ell} \gamma_{\mathcal{A}}^4 \|\mathcal{G}_k\|^2 \quad  \text{ for every } j_2+1 \leq j  \leq  j_3+1.
\end{equation}
From \eqref{e32d} and \eqref{e32a}  we conclude that \eqref{e32c} holds for every $j \in \{j_1,\cdots, j_3 \}$. Finally, by setting $j_s=j_3$  and restart the process for $j_3$ justified in \eqref{e131} and repeating the same argument, it can be shown that \eqref{e32c} holds for every $j\geq j_1$.  Recall that $j_1 \in \{r_{\ell},\cdots,r_{\ell}+\Theta_{\ell}+1\}$. Therefore \eqref{e32c} holds for every $j\geq r_{\ell+1}$ and the proof is finished.
\qed \end{proof}
In the next lemma, we investigate both of the cases $\delta_{\sup} < 2\delta_{\inf}$ and $\delta_{\sup} \geq 2\delta_{\inf}$.
\begin{lemma}
\label{R-1}
Let $\{  u_k \}_k$ be the sequence generated by Algorithm \ref{BBa} for \eqref{QP}. Then there exists a positive integer $m$ depending on $\delta_{\inf}$ and $\delta_{\sup}$ such that we have
\begin{equation}
\label{e55}
\|\mathcal{G}_{k+m}\| \leq \frac{1}{2}\|\mathcal{G}_{k}\|  \quad \text{for all } k\geq 1,
\end{equation}
or equivalently,
\begin{equation}
\label{e56}
\|u_{k+m}-u^*\| \leq \frac{1}{2}\|u_{k}-u^*\|  \quad \text{for all } k\geq 1,
\end{equation}
for all initial iterates $u_0, u_1 \in  \mathcal{H}$ with $u_0\neq u_1$  in the condition C1, or every initial iterate $u_1 \in \mathcal{H}$ and every initial step-size $\alpha_1>0$ in the condition C2.
\end{lemma}
\begin{proof}
If  $\delta_{\sup} < 2\delta_{\inf}$, then $\gamma_{\mathcal{A}}<1$ and,  by \eqref{e110}  in the proof of Lemma \ref{Theo4}, we have
\begin{equation*}
\|\mathcal{G}_{k+1}\| \leq \gamma_{\mathcal{A}} \|\mathcal{G}_k\|   \quad  \text{ for every }   k\geq 1.
\end{equation*}
Therefore, \eqref{e55} follows for the choice of $m :=  \ceil[\Big]{\frac{-\log 2}{\log \gamma_{\mathcal{A}}}}$.

Now, we consider the case in which $\delta_{\sup} \geq 2\delta_{\inf}$. In this case we have for $\rho_{\mathcal{A}}$ that  $\frac{1}{2} \leq\rho_{\mathcal{A}}<1$. First we decompose the interval $[\delta_{\inf}, \delta_{\sup}]$  into  the finite family of intervals $\{I_i\}^{n^{u}_{\eta}}_i$  defined by \eqref{e30} with a fixed length  $\eta \in  (0,\frac{\delta_{\inf}}{2}) \subset (0, \rho_{\mathcal{A}}\delta_{\inf}]$. Then due to \eqref{e21} and \eqref{e109}, we have for every $k \geq 1$
\begin{equation*}
  G(k,n^{u}_{\eta})=\sum^{n^{u}_{\eta}}_{i=1} (g^k_i)^2=\| \mathcal{G}_k \|^2,
\end{equation*}
where $  (g^k_i)^2$  is defined by \eqref{e29}. Moreover due to \eqref{e136} in the proof of  Lemma \ref{lem3}, there exists an integer $n^{l}_{\eta}>0$ such that for every  $\ell$  with  $1 \leq \ell \leq  n^{l}_{\eta}$, we have
\begin{equation*}
(g^{k+j}_{\ell})^2 \leq \rho^{2j}_{\mathcal{A}}(g^{k}_{\ell})^2 \quad   \text{ for every  } j \geq  0  \text{ and } k\geq 1.
\end{equation*}
By summing over all $\ell$ with  $1 \leq \ell \leq  n^{l}_{\eta}$, we obtain
\begin{equation*}
G(k+j,n^{l}_{\eta}) \leq \rho^{2j}_{\mathcal{A}} G(k,n^{l}_{\eta}) \leq \rho^{2j}_{\mathcal{A}}\| \mathcal{G}_k\|^2 \quad   \text{ for every  } j \geq  0  \text{ and } k\geq 1.
\end{equation*}
Now for the choice of $r_{n^{l}_{\eta}}:= \ceil[\Big]{\frac{\log\zeta_{n^{l}_{\eta}}}{2\log \rho_{\mathcal{A}}}}$ for any given $\zeta_{n^{l}_{\eta}}>0$, we have
\begin{equation*}
G(k+j,n^{l}_{\eta}) \leq \zeta_{n^{l}_{\eta}} \| \mathcal{G}_k\|^2 \quad   \text{ for every  } j \geq  r_{n^{l}_{\eta}}  \text{ and } k\geq 1,
\end{equation*}
and thus, by choosing $\zeta_{n^{l}_{\eta}}:=\frac{1}{4}(1+2\gamma^4_{\mathcal{A}})^{-(n^{u}_{\eta}-n^{l}_{\eta})}$ we are in the position to use Lemma \ref{lem4}. By using this lemma  we have for  $\ell$ with $n^{l}_{\eta}\leq \ell \leq n^{u}_{\eta}-1$ that
\begin{equation*}
\zeta_{\ell+1}=(1+2\gamma_{\mathcal{A}}^4)\zeta_{\ell} = \frac{1}{4}(1+2\gamma^4_{\mathcal{A}})^{\ell+1-n^{u}_{\eta}}         ,  \quad  \text{ and }  \quad  r_{\ell+1}= r_{\ell}+\Theta_{\ell}+1,
\end{equation*}
where $\Theta_{\ell} =\Theta(\zeta_{\ell},r_{\ell})$ has been defined as in Lemma \ref{alem1}. To be more precise, by applying Lemma \ref{lem4} once,   for the first iteration, we obtain
\begin{equation*}
G(k+j,n^{l}_{\eta}+1) \leq \zeta_{n^{l}_{\eta} +1}\|\mathcal{G}_{k}\|^2 = (1+2\gamma_{\mathcal{A}}^4)\zeta_{n^{l}_{\eta}} \|\mathcal{G}_{k}\|^2 = \frac{1}{4}(1+2\gamma^4_{\mathcal{A}})^{1-(n^{u}_{\eta}-n^{l}_{\eta})}\|\mathcal{G}_{k}\|^2
\end{equation*}
for all $j \geq r_{n^{l}_{\eta}+1}:=r_{n^{l}_{\eta}}+\Theta_{n^{l}_{\eta}}+1$.  Applying this lemma repeatedly  we conclude after  $(n^{u}_{\eta}-n^{l}_{\eta})-1$ iterations  that
\begin{equation*}
\|\mathcal{G}_{k+j}\|^2 = G(k+j,n^{u}_{\eta}) \leq \zeta_{n^{u}_{\eta}}\|\mathcal{G}_{k}\|^2=\frac{1}{4}\|\mathcal{G}_{k}\|^2     \quad \text{ for all }  j\geq {r_{n^{u}_{\eta}}}.
\end{equation*}
By putting $m =r_{n^{u}_{\eta}}$, \eqref{e55} holds.

 Moreover, the equivalence of \eqref{e55} with \eqref{e56} is justified due the fact that, similarly to \eqref{e17} for $\mathcal{G}_k$, it can easily be  shown that
\begin{equation}
\label{e115}
(u_{k+1}-u^*) = \frac{1}{\alpha_k }(\alpha_k \mathcal{I}-\mathcal{A})(u_{k}-u^*)   \quad \text{ for all  }  k = 0,1,2,\dots.
\end{equation}
Hence, the same machinery can be used to derive \eqref{e56} and this completes the proof.
\qed \end{proof}
\textbf{Proof of Theorem \ref{RConvergence}.} We need only to consider the case in which for every $k\geq 0$ we have $u_k \neq u^*$. In this case, we will show that $u_k \to u^*$ $R$-linearly. Due to \eqref{e115} and with a similar argument as in \eqref{e114}, we can write
\begin{equation}
\label{e116}
\begin{split}
\|u_{k+1}&-u^* \|^2 = \int_{\sigma(\mathcal{A})}\left( \frac{\alpha_k-\lambda}{\alpha_k} \right)^2 d(E_{\lambda}(u_{k}-u^*), (u_{k}-u^*)) \\
&\leq \gamma^2_{\mathcal{A}} \int_{\sigma(\mathcal{A})} \, d(E_{\lambda}(u_{k}-u^*), (u_{k}-u^*)) =  \gamma^2_{\mathcal{A}} \norm{u_{k}-u^*}^2  \quad    \text{ for every } k\geq 1.
\end{split}
\end{equation}
Moreover, due to \eqref{e56} in Lemma \ref{R-1}, we obtain
\begin{equation}
\label{e117}
\|u_{jm+1}-u^*\| \leq (\frac{1}{2})^{j}\|u_{1}-u^*\|  \quad \text{for all } j\geq 0,
\end{equation}
where $m$ has been defined in Lemma \ref{R-1}. Now for every $k\geq 1$, there exists an integer $j$ such that $1+jm \leq k < 1+(j+1)m$. Therefore, it follows that $k-(jm+1)<m$ and $j \geq\frac{k}{m}-1$.  Using \eqref{e116} and \eqref{e117}, we obtain
\begin{equation*}
\begin{split}
\|u_{k}-u^* \| &\leq \gamma^m_{\mathcal{A}}\|u_{jm+1}-u^* \|\leq \gamma^m_{\mathcal{A}}(\frac{1}{2})^{j}\|u_{1}-u^*\| \leq  \gamma^m_{\mathcal{A}}(\frac{1}{2})^{\frac{k}{m}-1}\|u_{1}-u^*\|\\
& = c_1c^k_2\|u_{1}-u^*\|    \qquad \text{ for all } k\geq 1,
\end{split}
\end{equation*}
where $c_1:=\gamma^m_{\mathcal{A}}(\frac{1}{2})^{-1}$ and $c_2:=(\frac{1}{2})^{\frac{1}{m}}<1$, and this completes the proof.

\begin{remark}
\label{Remark3}
If  $\sigma(\mathcal{A})$ is finite,  we can infer that $\sigma(\mathcal{A})=\{ \lambda_i :  i= 1,\dots, m \}$ with $\lambda_{i+1}>\lambda_{i}$ for $i=1,\dots,m-1$,  $\lambda_1 = \delta_{\inf}$,  and $\lambda_m=\delta_{sup}$. Then for every arbitrary $\eta >0$ and  partitioning $\{  I_i \}^{n^{u}_{\eta}}_{i=1}$ of $[\delta_{\inf},\delta_{\sup}]$,  we obtain for  $k \geq 1$ that
\begin{equation}
\label{e132}
\begin{split}
\|\mathcal{G}_{k+1}\|^2& = \int^{\delta_{\sup}}_{\delta_{\inf}} \left[ \prod^{k}_{p=1}\left( \frac{\alpha_p-\lambda}{\alpha_p} \right)^2 \right]  \, d(E_{\lambda}\mathcal{G}_{1}, \mathcal{G}_{1}) = \sum^{n^{u}_{\eta}}_{i = 1} \int_{I_i}  \left[ \prod^{k}_{p=1}\left( \frac{\alpha_p-\lambda}{\alpha_p} \right)^2 \right] \, d(E_{\lambda}\mathcal{G}_{1}, \mathcal{G}_{1})\\
& = \sum^{m}_{i=1}\left[ \prod^{k}_{p=1}\left( \frac{\alpha_p-\delta_{\inf}}{\alpha_p} \right)^2\right]\, \|E_{\{ \lambda _i \}}\mathcal{G}_1\|^2=   \sum^{m}_{i=1} (g^{k+1}_{i})^2,
\end{split}
\end{equation}
where $(g^{k+1})^2_{i}:=\left[ \prod^{k}_{p=1}\left( \frac{\alpha_p-\delta_{\inf}}{\alpha_p} \right)^2\right]\, \|E_{\{ \lambda _i \}}\mathcal{G}_1\|^2$. Then  the statements of Lemma \ref{lem3} is true for $n^{l}_{\eta} = 1$. Further, Lemma \ref{R-1} and Theorem \ref{RConvergence} are applicable. Moreover, in the proof of Lemma \ref{alem1}, similarly to \eqref{e132}, all the integrations  are replaced by finite sum and  it follows that $c := \max \{ \rho_{\mathcal{A}}, \frac{1}{2} \}$. See \cite{MR1880051,MR1225468} for more details.
\end{remark}
\begin{remark}
\label{Remark5}
Note that, due to Theorem \ref{Theo4}, the numerical behaviour of Algorithm \ref{BBa} is strongly depending on $\sigma(\mathcal{A})$.  In fact, this relation can be explained based on the value of the spectral condition number $\kappa(\mathcal{A}):=\|\mathcal{A}\|\|\mathcal{A}^{-1}\| = \frac{\delta_{sup}}{\delta_{\inf}}$.  It can be seen that  $\gamma_{\mathcal{A}}= \kappa(\mathcal{A})-1 $ and  $\rho_{\mathcal{A}}=1-\frac{1}{\kappa(\mathcal{A})} < 1$. Further, depending on the value of $\kappa(\mathcal{A})$, we can summarize the following cases:
\begin{enumerate}
\item $\kappa(\mathcal{A}) < 2$ : In this case, due to Theorem \ref{Theo4},  Algorithm \ref{BBa} is Q-linearly convergent with the rate $\gamma_{\mathcal{A}}<1$.  Moreover, from \eqref{e110}, we infer that the sequence  $\{ \|\mathcal{G}_k\| \}_k$  in monotone decreasing.
\item $\kappa(\mathcal{A}) \geq 2$ : This case is  more delicate. Recall from \eqref{e21} that  for every fixed $\eta \in (0,\frac{\delta_{\inf}}{2})$,   and $k\geq 1$,  we have $\|\mathcal{G}_{k+1}\|^2  = \sum^{n^{u}_{\eta}}_{i=1} (g^{k+1}_i )^2 $ where  the values $(g^{k+1}_{i})^2$ with  $i =1,\dots ,n^{u}_{\eta}$ are defined by
\begin{equation}
\label{e179}
(g^{k+1}_{i})^2  = \int_{I_i}\left( 1-\frac{\lambda}{\alpha_k} \right)^2 \left[ \prod^{k-1}_{p=1}\left( \frac{\alpha_p-\lambda}{\alpha_p} \right)^2 \right]\, d(E_{\lambda}\mathcal{G}_1, \mathcal{G}_1).
\end{equation}
 Due to Lemma \ref{lem3}, there exists an index $n^{l}_{\eta}\geq1$ such that the sequences $ \{|g^{k}_{i}|\}_k$  with $i =1,\dots,n^{l}_{\eta}$  are Q-linearly monotonically decreasing with factor $\rho_{\mathcal{A}}<1$. Therefore it remains only to consider the values of $ |g^{k}_{i}|$ for  $i=n^{l}_{\eta}+1,\dots,n^{u}_{\eta}$.  From \eqref{e179}, it can be shown that for every interval $I_i$ with $\alpha_k \in I_i$ it holds that  $ |g^{k+1}_{i}| \leq \frac{\eta}{\alpha_k}|g^{k}_{i}| < \frac{1}{2}|g^{k}_{i}|$. On the other hand, if for  an interval $I_i$ it holds that $a_i > 2\alpha_k$, then we obtain   $|g^{k+1}_{i}| > |g^{k}_{i}|$.    Further, for the last interval $I_{n^{u}_{\eta}}$,  we have $\frac{|g^{k+1}_{n^{u}_{\eta}}|}{|g^{k}_{n^{u}_{\eta}}|} \leq \kappa(\mathcal{A})-1$. These facts explain the potential nonmonotonic behaviour of the sequence $ \{\| \mathcal{G}_k\| \}_k$ and its dependence on $\kappa(A)$.
\end{enumerate}
\end{remark}
\begin{remark}[Preconditioning]
Due to Remark \ref{Remark5}, the convergence of Algorithm \ref{BBa} depends strongly on $\kappa(\mathcal{A})$. Analogously to the case of the conjugate gradient methods, the problem  \eqref{QP} can, by using an appropriate uniformly positive, self-adjoint, and continuous operator $\mathcal{C}:\mathcal{H} \to \mathcal{H}$, be transformed to the following equivalent problem
\begin{equation*}
\min_{z \in \mathcal{H}}\tilde{\mathcal{F}}(z):= \frac{1}{2}( \tilde{\mathcal{A}} z, z)-(\tilde{b} ,z),
\end{equation*}
where $\tilde{\mathcal{A}}: =\mathcal{C}^{-\frac{1}{2}}\mathcal{A}\mathcal{C}^{-\frac{1}{2}}$, $\tilde{b}:=\mathcal{C}^{-\frac{1}{2}}b$ and $z:=\mathcal{C}^{\frac{1}{2}}u$.  Clearly, $\sigma(\tilde{\mathcal{A}})=\sigma(\mathcal{C}^{-1}\mathcal{A})$ and, as a consequnce, the spectrum of $\tilde{\mathcal{A}}$ is completely determined by $\mathcal{C}$ and $\mathcal{A}$. Thus, the operator $\mathcal{C}$ can be chosen such that the application of Algorithm \ref{BBa} yields  faster convergence. In \cite{MR1417682},  preconditioning has been studied for Algorithm \ref{BBa} in the case of the Euclidean space $\mathbb{R}^n$.  For the case of infinite-dimensional Hilbert spaces,  preconditioning methods have been studied for the conjugate gradient methods. Among them we can mention \cite{MR1914497,MR2338397,MR2082565,MR3348129,MR2769031}.
\end{remark}
\subsubsection{The Case of a Positive Compact Perturbation}
\label{SecCompact}
In many situations of practical importance, we are faced with  problems of the form \eqref{QP}, in which $\mathcal{A}: \mathcal{H} \to \mathcal{H}$ is a compact perturbation of the identity. Therefore, it is of interest to consider this case separately.  Here, we have
\begin{equation}
\tag{poco}
\label{poco}
\mathcal{A} := \mathcal{T}+\beta \mathcal{I}
\end{equation}
with a positive self-adjoint compact operator $\mathcal{T}:\mathcal{H} \to \mathcal{H}$ and a constant $\beta>0$.  In this subsection we show how the special form of the spectrum of operators with the form \eqref{poco} allows us to simplify the proof of Theorems \ref{Theo4} and \ref{RConvergence} of the previous section. For $\mathcal{T}$ as above we have $\sigma(\mathcal{T})=\{0\}\cup\Sigma_{\mathcal{T}}$, where   $ \Sigma_{\mathcal{T}}: = \{ \delta_i : i \in N\}$ admits an enumeration for a countable set $N$. This set contains an ordered sequence of nonzero pairwise distinct eigenvalues, i.e, $\delta_{i+1} < \delta_i$ for $i\in N$. Moreover  $\delta_1=\|\mathcal{T}\|$ and for every $i\in N $, we have $\dim(\ker(\delta_i-\mathcal{T}))<\infty$ where $\ker(L):=\{ u \in \mathcal{H}: Lu =0\}$ for a given linear operator $L:\mathcal{H} \to \mathcal{H}$. Further, if $N$ is infinite, $\delta_n \to 0$ and  $N$ can be taken to be $\mathbb{N}:=\{1,2,3,\dots\}$. Then we have
\begin{equation*}
\mathcal{T} = \sum_{i \in N}\delta_i E_{\{ \delta_i\}},
\end{equation*}
where $E_{\{\delta_i\}}$ is the orthogonal projection to the space $\ker(\mathcal{T}-\delta_{i})$. In the case that $N$ is finite, the convergence can be proven as explained in Remark \ref{Remark3}, therefore we assume here that $N$ is infinite. Then, due to spectral mapping Theorem, we have
\begin{equation}
\label{e150}
\sigma(\mathcal{A}) = \{ \beta \} \cup  \{ \lambda_i : \lambda_i =\delta_i+\beta \text{ for } i \in N  \},
\end{equation}
where $\beta$ is a cluster point of the spectrum with $\lambda_{i} \to \beta$, and $\lambda_{i+1}< \lambda_{i}$ for ever $i\geq 1$. Then, for every measurable function $f: \sigma(\mathcal{A}) \to \mathbb{R}$, the operator  $f(\mathcal{A})$ is defined by
\begin{equation*}
f(\mathcal{\mathcal{A}}) =  \int_{\sigma(\mathcal{A})}f(\lambda) \, dE_{\lambda}=\sum_{i = 1}^{\infty}f(\lambda_i) E_{\{ \lambda_i\}} +f(\beta) E_{\{ \beta\}},
\end{equation*}
where $E_{\{ \beta \}}$ is the orthogonal projection to the space $\ker(\mathcal{T})$. Note that $E_{\{ \beta \}}=0$ unless zero is an eigenvalue of $\mathcal{T}$.  For convenience in notation, we denote $\lambda_0 = \beta$. Then we have
 \begin{equation}
\label{e14}
f(\mathcal{\mathcal{A}}) =  \sum_{i = 0}^{\infty}f(\lambda_i) E_{\{ \lambda_i\}},
\end{equation}
and similarly, for every $x,y \in \mathcal{H}$ we have
\begin{equation}
\label{e129}
(f(\mathcal{A})x,y) = \int_{\sigma(\mathcal{A})}f(\lambda)d(E_{\lambda}x,y) = \sum_{i = 0}^{\infty}f(\lambda_i) (E_{\{ \lambda_i \}}x,y).
\end{equation}
 Due to structure of $\mathcal{A}$ and the definition of $\delta_{\inf}$ and  $\delta_{\sup}$, we have
\begin{equation}
\label{e174}
\delta_{\inf} =\lambda_{0} =  \beta,    \quad   \delta_{\sup} =\lambda_1 = \| \mathcal{T} \| +\beta.
\end{equation}
Moreover, due to \eqref{e14} and \eqref{e129}, for  $\mathcal{G}_1 \in \mathcal{H}$  we can write
\begin{equation*}
\mathcal{G}_1 = \sum_{i = 0}^{\infty} E_{\{ \lambda_i\}}\mathcal{G}_1 \quad \text{ and  } \quad \|\mathcal{G}_1\|^2 = \sum_{i = 0}^{\infty}(E_{\{ \lambda_i \}}\mathcal{G}_1,\mathcal{G}_1) =  \sum_{i = 0}^{\infty}\| E_{\{ \lambda_i \}}\mathcal{G}_1\|^2.
\end{equation*}
 Using \eqref{e17} and  \eqref{e14}, we have
\begin{equation*}
\mathcal{G}_2 = \frac{1}{\alpha_1} (\alpha_1-\mathcal{A})\mathcal{G}_1 = \sum^{\infty}_{i = 0} \frac{1}{\alpha_1} (\alpha_1-\lambda_i)E_{\{\lambda_i\}}\mathcal{G}_1,
\end{equation*}
and, in a similar manner by induction, we obtain for every $ k \geq 1$ that
\begin{equation}
\label{e128}
\begin{split}
\mathcal{G}_{k} =  \sum^{\infty}_{i=0} \left[ \prod^{k-1}_{p=1}\left( \frac{\alpha_p-\lambda_i}{\alpha_p} \right) \right] E_{\{\lambda_i\}}\mathcal{G}_1 \quad \text{ and } \quad
\|\mathcal{G}_{k}\|^2 = \sum^{\infty}_{i=0} \left[ \prod^{k-1}_{p=1}\left( \frac{\alpha_p-\lambda_i}{\alpha_p} \right)^2 \right]\| E_{\{\lambda_i\}}\mathcal{G}_1\|^2 .
\end{split}
\end{equation}
For every $i\geq 0$ and $k\geq 1$ we define
\begin{equation}
\label{e4a}
(g^{k}_i)^2:= \left[ \prod^{k-1}_{p=1}\left( \frac{\alpha_p-\lambda_i}{\alpha_p} \right)^2 \right]\| E_{\{\lambda_i\}}\mathcal{G}_1\|^2=\left[ \prod^{k-1}_{p=1}\left( \frac{\alpha_p-\lambda_i}{\alpha_p} \right)^2 \right](g^1_i)^2,
\end{equation}
and conclude that
\begin{equation}
\label{e6}
\|\mathcal{G}_{k}\|^2 = \sum^{\infty}_{i = 0} (g^{k}_{i})^2 \quad  \text{ and  } \quad (g^{k+1}_{i})^2 = \left( \frac{\alpha_k-\lambda_i}{\alpha_k}\right)^2 (g^{k}_{i})^2.
\end{equation}
Using  \eqref{eqkk11} and \eqref{e128}, we can write for every $k \geq 1$ and any chosen  $\alpha_{p} =\alpha^{BB1}_{p}$ or  $\alpha_{p} =\alpha^{BB2}_{p}$  with $p = 1,2,\dots,k $   that
\begin{equation}
\label{e4}
\begin{split}
\|\mathcal{G}_{k+1}\|^2=& \sum^{\infty}_{i=0} \left[ \prod^k_{p=1}\left( \frac{\alpha_p-\lambda_i}{\alpha_p} \right)^2 \right]\| E_{\{\lambda_i\}}\mathcal{G}_1\|^2 \leq  \gamma^{2}_{\mathcal{A}}\sum^{\infty}_{i=0}\left[ \prod^{k-1}_{p=1}\left( \frac{\alpha_p-\lambda_i}{\alpha_p} \right)^2 \right]\| E_{\{\lambda_i\}}\mathcal{G}_1\|^2 =  \gamma^{2}_{\mathcal{A}}\|\mathcal{G}_{k}\|^2.
\end{split}
\end{equation}
Hence, analogously to Theorem \ref{Theo4}, we can conclude that for \eqref{QP} with an operator $\mathcal{A}$ of the form  \eqref{poco}, Algorithm \ref{BBa} is $Q$-linearly convergent, provided that   $\delta_{\sup} < 2\delta_{\inf}$.  Next we consider the general case.  In a similar manner as in the previous subsubsection, we define
\begin{equation*}
G(k,\ell):= \sum_{i  \in \{0\}\cup\{i: i\geq \ell\}} (g^k_{i})^2    \quad \text{ for every }  \ell, k\geq 1.
\end{equation*}
Then, by \eqref{e128}  we have
 \begin{equation}
\label{e123}
G(k,1)= \sum_{i  \in \{0\}\cup\{i: i\geq1\}} (g^k_{i})^2=\sum^{\infty}_{i =1} (g^k_{i})^2 = \|\mathcal{G}_k\|^2     \quad \text{ for every }  k\geq 1.
\end{equation}
First, analogously to Lemma \ref{lem3}, we will show that
that there exists an index $n^{u}$ depending on $\delta_{\inf}$ and $\delta_{\sup}$ such that the sequences $\{g^{k}_{i}\}_k$ with $ i \geq n^{u}$ converge to zero $Q$-linearly.

\begin{lemma}
\label{lem1}
There exists a positive integer $n^{u}$ such that for any  $i \in \{0\}\cup \{i: i\geq n^{u}\}$, the sequences  $\{g^{k}_{i}\}_k$ converge to zero $Q$-linearly with the factor $\rho_{\mathcal{A}}$ as $k$ tends to infinity and we also have
 \begin{equation}
\label{e35}
\lim_{k \to \infty}G(k,n^{u}) =0.
\end{equation}
\end{lemma}
\begin{proof}
Since  $\{ \lambda_i \}_i$ is a positive and decreasing sequence and $\lambda_i \to \beta = \delta_{\inf}$, there exists a positive integer $n^{u}$ such that for every $i \geq n^{u}$ we have $\lambda_i - \delta_{\inf} \leq \rho_{\mathcal{A}}\delta_{\inf}$.  For every $i \in \{0\}\cup \{i: i\geq n^{u}\}$ and $k\geq 1$, we have the following two cases:
\begin{enumerate}
\item If $\alpha_k-\lambda_i \geq 0$ then we have
\begin{equation*}
\abs{\frac{\alpha_k-\lambda_i}{\alpha_k}} = \left(1- \frac{\lambda_i}{\alpha_k}\right) \leq  \rho_{\mathcal{A}} < 1.
\end{equation*}
\item  If $\alpha_k-\lambda_i < 0$ then we have  $\alpha_k \in [\delta_{\inf}, \lambda_i )$ and thus
\begin{equation*}
\abs{\frac{\alpha_k-\lambda_i}{\alpha_k}} = \frac{\lambda_i-\alpha_k}{\alpha_k} \leq \frac{\lambda_i-\delta_{\inf}}{\delta_{\inf}} \leq \rho_{\mathcal{A}}.
\end{equation*}
\end{enumerate}
Therefore, by using \eqref{e4a} and \eqref{e6}, we can infer for every $k\geq 1$ and  $i \in \{0\}\cup \{i: i\geq n^{u}\}$  that
\begin{equation}
\label{e33}
(g^{k+1}_{i})^2 = \left( \frac{\alpha_k-\lambda_i}{\alpha_k}\right)^2 (g^{k}_{i})^2 \leq  \rho^{2}_{\mathcal{A}} (g^{k}_{i})^2 \leq \cdots \leq  \rho^{2k}_{\mathcal{A}} (g^{1}_{i})^2.
\end{equation}
Now due to \eqref{e33}, we conclude that the sequences  $\{g^{k}_{i}\}_k$ converge to zero $Q$-linearly for all  $i \in \{0\}\cup \{i: i\geq n^{u}\}$. Moreover, using \eqref{e33} and summing up for every $i \in \{0\}\cup \{i: i\geq n^{u}\}$ we obtain for every $k \geq 1$  that
\begin{equation*}
G(k+1,n^{u}) =\sum_{i \in \{0\}\cup \{i: i\geq n^{u}\}} (g^{k+1}_i)^2 \leq \rho^{2k}_{\mathcal{A}}\sum_{i \in \{0\}\cup \{i: i\geq n^{u}\}} (g^{1}_i)^2 = \rho^{2k}_{\mathcal{A}}\|\mathcal{G}_1\|^2.
\end{equation*}
Thus \eqref{e35} follows.
\qed \end{proof}
Next, we consider results which are analogous to Lemmas \ref{alem1} and \ref{lem4} for the special case \eqref{poco}.
\begin{lemma}
 \label{clem1}
For each integer $\ell$ with  $1 < \ell  \leq  n^{u}$,  and $k\geq 1$,  the following property holds:

If the condition
\begin{equation}
\label{e118}
G(k+j,\ell) \leq \overline{\zeta}\|\mathcal{G}_k\|^2  \quad \text{ for all } j\geq \overline{r}
\end{equation}
holds for some positive  $\overline{r} \in \mathbb{N}$ and $\overline{\zeta} \in \mathbb{R}_+$, then there exists an integer $\hat{j} \in \{ \overline{r},\cdots ,\overline{r}+\Theta+1\}$  such that
\begin{equation*}
(g^{k+\hat{j}}_{\ell-1})^2 \leq 2\overline{\zeta}\| \mathcal{G}_k\|^2,
\end{equation*}
where $\Theta=\Theta(\overline{\zeta},\overline{r})$ is defined as in Lemma \ref{alem1} with $c := \max \{ \rho_{\mathcal{A}}, \frac{1}{2} \}$.
\end{lemma}
\begin{proof}
Similarly to the proof of Lemma \ref{alem1}, we assume that
\begin{equation}
\label{e119}
(g^{k+j}_{\ell -1})^2 > 2\overline{\zeta}\| \mathcal{G}_k\|^2  \quad \text{ for all  } j \in \{\overline{r}, \cdots , \overline{r} +\Theta\},
\end{equation}
and we show that
\begin{equation*}
(g^{k+\overline{r}+\Theta+1}_{\ell -1})^2 \leq  2\overline{\zeta}\| \mathcal{G}_k\|^2.
\end{equation*}
By Algorithm \ref{BBa},  for every $j \in \{\overline{r}, \cdots , \overline{r} +\Theta\}$ we have either $\alpha_{k+j} = \alpha^{BB1}_{k+j}$  or $\alpha_{k+j} = \alpha^{BB2}_{k+j}$.
Using  \eqref{e2}, \eqref{e36}, and \eqref{e129}, we can write for every $j \in \{\overline{r}, \cdots , \overline{r} +\Theta\}$ that
\begin{equation}
\label{e7}
\alpha^{BB1}_{k+j+1} = \frac{\sum^{\infty}_{i =0}(g^{k+j}_i)^2\lambda_i}{\sum^{\infty}_{i =0}(g^{k+j}_i)^2}, \text{ and } \quad  \alpha^{BB2}_{k+j+1} = \frac{\sum^{\infty}_{i =0}(g^{k+j}_i)^2\lambda^2_i}{\sum^{\infty}_{i =0}(g^{k+j}_i)^2\lambda_i}.
\end{equation}
Using \eqref{e118} and \eqref{e7}, we obtain for every  $j \in \{\overline{r}, \cdots , \overline{r} +\Theta\}$ that
\begin{equation}
\label{e10}
\frac{ \lambda_{\ell-1}\sum^{\ell -1}_{i=1} (g^{k+j}_i )^2}{ \overline{\zeta}\|\mathcal{G}_k\|^2 + \sum^{\ell -1}_{i=1} (g^{k+j}_i )^2} \leq \alpha^{BB1}_{k+j+1}  \leq \delta_{\sup}.
\end{equation}
Moreover, using  \eqref{e118} and the fact that $\lambda_{\ell-1} \geq \lambda_{i} $ for every $i \in \{ 0\} \cup \{ i:  i\geq \ell \}$, we obtain
\begin{equation}
\label{e120}
\sum_{ i \in \{ 0 \}\cup\{i: i\geq \ell\}} (g^{k+j}_i)^2\lambda_i \leq \lambda_{\ell-1}G(k+j,\ell) \leq \lambda_{\ell-1}\overline{\zeta}\| \mathcal{G}_k\|^2  \quad \text{ for every }j \in \{\overline{r}, \cdots , \overline{r} +\Theta\}.
\end{equation}
Then by using \eqref{e118} and \eqref{e7}, we have for every  $j \in \{\overline{r}, \cdots , \overline{r} +\Theta\}$ that
\begin{equation}
\label{e121}
\begin{split}
\frac{ \lambda_{\ell-1}\sum^{\ell -1}_{i=1} (g^{k+j}_i )^2}{ \overline{\zeta}\|\mathcal{G}_k\|^2 + \sum^{\ell -1}_{i=1} (g^{k+j}_i )^2} = \frac{ \lambda^2_{\ell-1}\sum^{\ell -1}_{i=1} (g^{k+j}_i )^2}{ \lambda_{\ell-1}(\overline{\zeta}\|\mathcal{G}_k\|^2 + \sum^{\ell -1}_{i=1} (g^{k+j}_i )^2)}&\leq \frac{ \lambda_{\ell-1}\sum^{\ell -1}_{i=1} (g^{k+j}_i )^2\lambda_{i}}{ \lambda_{\ell-1}\overline{\zeta}\|\mathcal{G}_k\|^2 + \sum^{\ell -1}_{i=1} (g^{k+j}_i )^2\lambda_{i}}\\& \leq \alpha^{BB2}_{k+j+1}  \leq \delta_{\sup}.
\end{split}
\end{equation}
From  \eqref{e119}, \eqref{e10}, \eqref{e121}, and the fact that  $\sum^{\ell-1}_{i=1} (g^k_i)^2 \geq  (g^k_{\ell -1})^2 $,  it follows, with a computations similar to those in the proof of Lemma \ref{alem1} which leads to \eqref{e27},   for a chosen $\alpha_{k+j+1} = \alpha^{BB1}_{k+j+1}$ or  $\alpha_{k+j+1} = \alpha^{BB2}_{k+j+1}$  that
\begin{equation}
\label{e11}
\frac{2}{3} \lambda_{\ell -1}  \leq    \alpha_{k+j+1} \leq \delta_{\sup}    \quad \text{ for all }  j \in \{\overline{r}, \cdots , \overline{r} +\Theta\}.
\end{equation}
Moreover, considering separately the cases $\frac{\lambda_{\ell-1}}{\alpha_{k+j+1}} <1$ and  $\frac{\lambda_{\ell-1}}{\alpha_{k+j+1}} \ge 1$, we obtain due to \eqref{e11} that
\begin{equation}
\label{e12}
\abs{ 1- \frac{\lambda_{\ell-1}}{\alpha_{k+j+1}}} \leq c  <  1 \quad \text{ for all }  j \in \{\overline{r}, \cdots , \overline{r} +\Theta\},
\end{equation}
where $c =\max\{\frac{1}{2}, \rho_{\mathcal{A}}\}$. Using  \eqref{e6} and \eqref{e12} we conclude that
\begin{equation}
\label{e122}
| g^{k+j+2}_{\ell -1} | = \abs{ 1- \frac{\lambda_{\ell-1}}{\alpha_{k+j+1}}} | g^{k+j+1}_{\ell -1} | \leq c |g^{k+j+1}_{\ell -1} |  \quad \text{for all }   j \in \{\overline{r}, \cdots , \overline{r} +\Theta\}.
\end{equation}
Finally, by \eqref{e6} and  \eqref{e122} we obtain
\begin{equation*}
\begin{split}
(g^{k+\overline{r}+\Theta+1 }_{\ell -1})^2 \leq c^{2\Theta}(g_{\ell-1}^{k+\overline{r}+1})^2\leq  c^{2\Theta} \gamma_{\mathcal{A}}^{2(\overline{r}+1)}(g^{k}_{\ell -1})^2 \leq  c^{2\Theta}\gamma_{\mathcal{A}}^{2(\overline{r}+1)}\| \mathcal{G}_k\|^2\leq 2\overline{\zeta}\|\mathcal{G}_{k}\|^2,
\end{split}
\end{equation*}
and the proof is complete.
\qed \end{proof}
Next since we have
\begin{equation*}
G(k+j ,\ell-1) = G(k+j,\ell)+(g^{k+j}_{\ell-1})^2   \quad \text{ for every } k,j \geq1,
\end{equation*}
 with a similar argument as in the proof of Lemma \ref{lem4}, the following lemma can be proven.
\begin{lemma}
\label{clem4}
Let $\delta_{\sup} \geq 2\delta_{\inf}$. Moreover,  assume that for any  integer $\ell$ with  $1 < \ell  \leq  n^{u}$,  and $k\geq 1$, there exist positive numbers  $r_{\ell}$  and $\zeta_{\ell}$ such that \eqref{e118} holds for $\overline{r} = r_{\ell}$ and $\overline{\zeta}=\zeta_{\ell}$. Then for the choice of $\zeta_{\ell-1}:=(1+2\gamma_{\mathcal{A}}^4)\zeta_{\ell}$ and $r_{\ell-1}:= r_{\ell}+\Theta_{\ell}+1$,
we have
\begin{equation*}
G(k+j,\ell -1) \leq \zeta_{\ell-1} \| \mathcal{G}_k\|^2   \quad \text{ for all } j\geq r_{\ell -1}.
\end{equation*}
\end{lemma}
Finally by using Lemma \ref{lem1} and \ref{clem4}, we can prove Lemma \ref{R-1} for the special case in which $\mathcal{A}$ is a compact perturbation of the identity.   For the case  of $\delta_{sup} < 2\delta_{inf}$ the proof is similar to the proof of Lemma \ref{R-1}. Here we give some hints about the other case, namely,  $\delta_{sup}\geq 2\delta_{inf}$.  First, due to Lemma \ref{lem1}, for $\zeta_{n^{u}}:=\frac{1}{4}(1+4\gamma^2_{\mathcal{A}})^{1-n^{u}}$ there exists an integer $r_{n^{u}}>0$ such that
\begin{equation*}
G(k,n^u) \leq \zeta_{n^u}\|\mathcal{G}_k\|^2 \quad \text{ for all } k\geq r_{n^u},
\end{equation*}
where $n^{u}=n^{u}(\delta_{\inf},\delta_{\sup})>0$  is defined in Lemma \ref{lem1}. Then, similarly to the  proof of  Lemma \ref{R-1}, by induction and using  Lemma \ref{clem4},  we have for  $\ell$ with $2 \leq \ell \leq n^{u}$ that
\begin{equation*}
\zeta_{\ell-1}=(1+2\gamma_{\mathcal{A}}^4)\zeta_{\ell} = \frac{1}{4}(1+2\gamma^4_{\mathcal{A}})^{1-(\ell-1)}         ,  \quad  \text{ and }  \quad  r_{\ell-1}= r_{\ell}+\Theta_{\ell}+1.
\end{equation*}
Therefore,  due to \eqref{e123}, \eqref{e55} holds for $m: = r_1$.

Now we can conclude that Theorems \ref{Theo4} and \ref{RConvergence} hold for \eqref{poco}.

\begin{remark}
\label{Remark6}
 The explanations in Remark \ref{Remark5} also apply for operators $\mathcal{A}$ of the form \eqref{poco}. Due to \eqref{e6}, the case $\kappa(\mathcal{A}) < 2$  is analogous to that in Remark \ref{Remark5}. For the  case $\kappa(\mathcal{A}) \geq 2$, we can gain more information since the spectrum of $\sigma(\mathcal{A})$ is discrete. Due to the right equality in \eqref{e6} we have
 \begin{equation}
\label{e180}
|g^{k+1}_{i}| = \abs{ \frac{\alpha_k-\lambda_i}{\alpha_k}}|g^{k}_{i}|.
\end{equation}
Therefore, for $i=0,1,\dots$ we obtain  $g^{k+1}_{i}= 0$ if $\alpha_k = \lambda_i$. Moreover, due to Lemma \ref{lem1}, there exists an index $n^{u}\geq1$ such that the sequences $ \{|g^{k}_{i}|\}_k$  with $i \in \{0\}\cup \{i: i\geq n^{u}\}$  are Q-linearly convergent with factor $\rho_{\mathcal{A}}<1$. Therefore, we need only to consider the values of $ |g^{k}_{i}|$ with  $i=1,\dots,n^{u}-1$.  From \eqref{e180},  for any  $\lambda_i$ close  to $\alpha_k $ we have a significant reduction and  $ |g^{k+1}_{i}| \ll  |g^{k}_{i}|$, while  for  $\lambda_i > 2\alpha_k$,  we obtain  $|g^{k+1}_{i}| > |g^{k}_{i}|$. These facts clarify the potential nonmonotonic behaviour in the sequence $ \{\| \mathcal{G}_k\| \}_k$.   In fact, for  $\alpha_k$ close to $\lambda_1$, the  coefficients $g_i$
decrease in modulus, but the changes in $g_i$  with $i \in \{0\}\cup \{i: i\geq n^{u}\}$ are negligible provided that  $\kappa(\mathcal{A})$  is large.  Furthermore,  small values of $\alpha_k$, tend to diminish the components $|g^k_i|$ for small $i\neq 0$ and thus, enhance the relative contribution of components for large $i$.
\end{remark}

\subsection{General Objective Function}\label{sec2.2}
 In this section,  we will prove the local  $R$-linear convergence of Algorithm \ref{BBa}, in the case that this algorithm is applied for finding a local minimum $u^* \in \mathcal{H}$ of a not necessarily quadratic function $\mathcal{F}: \mathcal{H} \to \mathbb{R}$.  More precisely, $\mathcal{F}$  is twice continuously Fr\'echet-differentiable at $u^*$ with Lipschitz continuous second derivative $\mathcal{F}''  $  in a neighbourhood of  $u^* \in \mathcal{H}$. Then if we identify the first derivative $\mathcal{F}'$ by its corresponding representation $\mathcal{G}$, we have the following first-order optimality condition
\begin{equation}
 \label{e39}
 \tag{EP}
  \mathcal{G} (u^*) = 0 \quad \text{ in } \mathcal{H}.
 \end{equation}
 Due to the continuity of the bilinear map  $\mathcal{F}''(u^*)$, the exists a positive constant  $\delta_{\sup}$  such that
\begin{equation}
\label{e130cont}
\mathcal{F}''(u^*)(v,u) \leq \delta_{\sup} \|v\|\|u\|   \quad   \text{ for all }   u, v \in \mathcal{H}.
\end{equation}
 Moreover, we assume that the continuous bilinear map  $\mathcal{F}''(u^*)$ is uniformly positive, that is
\begin{equation}
\label{e130}
\delta_{\inf} \|v\|^2 \leq \mathcal{F}''(u^*)(v,v) \leq \delta_{\sup} \|v\|^2   \quad   \text{ for all }  v \in \mathcal{H},
 \end{equation}
where $\delta_{\sup} \geq \delta_{\inf}>0$.  Then, due to the Riesz representation theorem, there exists a unique self-adjoint bounded operator  $\mathcal{A}^{\mathcal{F}}_{u^*}$ (see \cite{MR1070713}[Theorem 2.2, page 31] ) such that
\begin{equation*}
\mathcal{F}''(u^*)(v,u)= (\mathcal{A}^{\mathcal{F}}_{u^*}v,u) \quad  \text{for all } v,u \in \mathcal{H}.
\end{equation*}
  Similarly to the analysis of \cite{MR2241317,MR1865673}, the $R$-linearly convergence result is proven by comparing the sequences $\{u_k\}_k$ and $\{\hat{u}_k\}_k$ which are generated by Algorithm \ref{BBa} applied to,  respectively, $\mathcal{F}$ and its second-order Taylor approximation $\hat{\mathcal{F}}$ defined by
\begin{equation}
\label{quadmodel}
\hat{\mathcal{F}}(u) = \mathcal{F}(u^*) +\frac{1}{2}(\mathcal{A}^{\mathcal{F}}_{u^*}(u-u^*),u-u^*).
\end{equation}

Throughout this section,  all notations with the accent  ``  \string^ ''   are related to the quadratic approximation \eqref{quadmodel}. For instance with $\hat{\mathcal{G}}(\cdot)$ and $\hat{\alpha}_k$, we denote the gradient and the step-sizes of Algorithm \ref{BBa} applied to $\hat{\mathcal{F}}$, respectively.

Since $\mathcal{F}'' : \mathcal{H} \to \mathcal{L}(\mathcal{H},\mathcal{L}(\mathcal{H},\mathbb{R}))$ is locally Lipschitz continuous and $\mathcal{F}''(u^*): \mathcal{H}\times \mathcal{H} \to \mathbb{R}$ is continuous and uniformly positive, there exist a ball $\mathcal{B}_{\tau}(u^*)$  centered at $u^*$ with a radius $\tau$,   positive constants $\alpha_{\inf}$, $\alpha_{\sup}$ depending on $\tau$,  and $L$ such that
\begin{equation}
\tag{$L1$}
\label{e45}
\|\mathcal{G}(u) - \mathcal{A}^{\mathcal{F}}_{u^*}(u-u^*)\| \leq L\|u-u^*\|^2   \quad \text{ for all } u\in \mathcal{B}_{\tau}(u^*),
\end{equation}
 and
\begin{equation}
\tag{$L2$}
\label{e46}
\alpha_{\inf} \|v\|^2 \leq \mathcal{F}''(u)(v,v) \leq  \alpha_{\sup} \|v\|^2   \quad \text{ for all } v \in \mathcal{H} \text{ and }  u\in \mathcal{B}_{\tau}(u^*).
\end{equation}
Moreover, due to the mean value theorem we have
\begin{equation}
\label{e80}
\alpha_{\inf} \leq \alpha^{BB1}_k, \alpha^{BB2}_k \leq  \alpha_{\sup},
\end{equation}
provided that $u_k$ and $u_{k-1}$ belong to $\mathcal{B}_{\tau}(u^*)$.  Moreover if the iterations of Algorithm \ref{BBa} applied to $\hat{\mathcal{F}}$ lie in $\mathcal{B}_{\tau}(u^*)$, we will also have
\begin{equation}
\label{e81}
\alpha_{\inf} \leq  \delta_{\inf}  \leq  \hat{\alpha}^{BB1}_k, \hat{\alpha}^{BB2}_k \leq \delta_{\sup} \leq  \alpha_{\sup}.
\end{equation}
Further, by the  fundamental theorem of calculus, we infer that
\begin{equation}
\label{e60}
\alpha_{\inf}\|u-u^*\| \leq \|\mathcal{G}(u)\| = \| \mathcal{G}(u)-\mathcal{G}(u^*)\| \leq \alpha_{\sup}\|u-u^*\|  \quad   \text{ for all } u \in \mathcal{B}_{\tau}(u^*).
\end{equation}

In the next lemma we study the distance of the sequences  $\{u_k \}_k$ and $\{\hat{u}_k\}_k$.
\begin{lemma}
\label{lem5}
Let $u^*$ be a local minimizer of  $\mathcal{F}$ with $\mathcal{F}\in C^2(\mathcal{H},\mathbb{R})$ and assume that \ref{e45} and \ref{e46} hold for a radius $\tau$ and constants $\alpha_{\inf}$ and  $\alpha_{\sup}$, and for the bilinear form $ \mathcal{F}''(u^*)$ estimate \eqref{e130} holds with the constants $\delta_{\sup}$ and  $\delta_{\inf}$.  Further, let $\{ u_j\}_j$ be a sequence generated by Algorithm \ref{BBa} applied to $\mathcal{F}$,  and  $\{ \hat{u}^k_j\}_j$ be the sequence generated by Algorithm \ref{BBa} applied to the quadratic approximation \eqref{quadmodel} of $\mathcal{F}$ at $u^*$ with an initial iterate $u_k$ and an initial step-size $\alpha_k$ with $k\geq 1$. Then for any fixed positive integer $m$, there exist positive constants $\eta \leq \tau$ and $\lambda$ such that the following property holds:

If $u_k \in \mathcal{B}_{\eta}(u^*)$, $\alpha_k \in [\alpha_{\inf},\alpha_{\sup}] $,   and if for some  $\ell \in \{0, \dots, m\}$,  the following condition holds
\begin{equation}
\label{e67}
\|\hat{u}^k_j-u^*\| \geq \frac{1}{2} \| u_k-u^*\|  \quad \text{ for all  }  j \in \{0, \dots, \max\{0,\ell-1\}\},
\end{equation}
then we have
\begin{equation}
\label{e62}
u_{k+j} \in \mathcal{B}_{\tau}(u^*)  \quad \text{ and } \quad \|u_{k+j}-\hat{u}^k_j\| \leq \lambda \|u_k-u^*\|^2
\end{equation}
for all $j \in \{ 0,\dots,\ell\}$.
\end{lemma}
\begin{proof}
The proof is given in Appendix \ref{Ap1}.
\qed \end{proof}

In the next theorem, we present the main result of this section which is the local $R$-linearly convergence of Algorithm \ref{BBa} applied to twice continuously Fr\'echet differentiable objective functions.
 \begin{theorem}
\label{Theo3}
Let $u^*$ be a local minimizer of a twice continuously Fr\'echet differentiable function $\mathcal{F}$, with a locally Lipschitz continuous second-derivative. Further suppose that the bilinear mapping $\mathcal{F}''(u^*)$ satisfies estimate \eqref{e130} for constants $\delta_{\sup}$ and  $\delta_{\inf}$.  Then there exist positive constants $\zeta$, $\lambda_1$, $\lambda_2$, and $\theta<1$ such that the sequence  $\{u_k\}_k$, generated by Algorithm \ref{BBa}, satisfies
\begin{equation}
\label{e58}
\|u_k-u^*\| \leq \lambda_1  \theta^k \| u_1- u^*\|   \quad  \text{ for all  } k \geq 1,
\end{equation}
and
\begin{equation}
\label{e59}
\|\mathcal{G}_k\| \leq \lambda_2  \theta^k \| \mathcal{G}_1\| \quad  \text{ for all  } k \geq 1,
\end{equation}
for all initial iterates $u_0, u_1 \in \mathcal{B}_{\zeta}(u^*) \in \mathcal{H}$ with $u_0 \neq u_1$.
\end{theorem}
\begin{proof}
The assumptions on  $\mathcal{F}$ imply that \ref{e45} and \ref{e46} are satisfied for a radius $\tau$ and constants $\alpha_{\inf}$ and  $\alpha_{\sup}$.
The proof relies on  Lemma \ref{R-1} and Lemma \ref{lem5} in an essential manner. By Lemma \ref{R-1},  which we use for the sequences $\{\hat u^k_j\}_j$, for every  initial iterate $\hat{u}^k_0 := u_k$ and initial step-size $\hat{\alpha}^k_0 :=\alpha_k $ with
\begin{equation}
\label{e47}
\alpha_{\inf} \leq \alpha_k=\hat{\alpha}^k_0  \leq \alpha_{\sup},
\end{equation}
 we have
 \begin{equation}
\label{e49}
\| \hat{u}^k_m-u^*\| \leq \frac{1}{2} \| u_k -u^*\|.
\end{equation}
Given the constants $\eta\le \tau $ and $\lambda$ from Lemma \ref{lem5} we define  $\zeta := \min\{\eta, \tau_1\}$, where $\tau_1$ is chosen such that $c_2:=\frac{1}{2}+\lambda \tau_1 <1$. Then, due to Lemma \ref{lem5},  for the fixed integer $m$,  if $u_k\in \mathcal{B}_{\zeta}(u^*)$, if  $\alpha_k$ satisfies  \eqref{e47}, and if
\begin{equation}
\label{e53}
\|\hat{u}^k_j-u^*\| \geq \frac{1}{2} \| u_k-u^*\|  \quad \text{ for all  }  j \in \{0, \dots, \max \{0,\ell-1\}\} \mbox{ with } \ell\leq m,
\end{equation}
then we have
\begin{equation}
\label{e50}
u_{k+j} \in \mathcal{B}_{\tau}(u^*) \text{ and } \|u_{k+j}-\hat{u}^k_j\| \leq \lambda \|u_k-u^*\|^2 \quad \text{ for all } j\in \{ 1, \cdots, \ell\}.
\end{equation}

Next we show by induction that there exists a subsequence of indices $\{k_i\}_i$ with  $k_1=1$, for which we have
\begin{equation}
\label{e52}
k_{i+1}-k_{i} \leq m \quad \text{ and }  \quad \| u_{k_{i+1}}- u^* \| \leq c_2 \|u_{k_{i}}-u^* \|,
\end{equation}
for all  $i=1,2,\dots$.

For any $u_0,u_1 \in \mathcal{B}_{\zeta}(u^*)  \subset  \mathcal{B}_{\tau}(u^*)$  and $k_1 = 1$, due to \ref{e46} we obtain
\begin{equation*}
\alpha_{\inf} \leq \alpha_1=\hat{\alpha}^1_0  \leq \alpha_{\sup}.
\end{equation*}
Due to Lemma \ref{R-1}  and  \eqref{e49},  there exists a smallest integer $j_1\leq m$ such that
\begin{equation}
\label{e57}
\|\hat{u}^{k_1}_{j_1}-u^*\| \leq \frac{1}{2}\|\hat{u}^{k_1}_0 - u^* \| = \frac{1}{2} \| u_{k_1} - u^*\| .
\end{equation}
Defining $k_2: = k_1 +j_1 > k_1$, and using \eqref{e50} and \eqref{e57}, we have
\begin{equation}
\label{e51}
\begin{split}
\|u_{k_2}-u^*\| &= \| u_{k_1+j_1}-u^*\| \leq \| u_{k_1+j_1}-\hat{u}_{j_1}^{k_1}\|+\| \hat{u}_{j_1}^{k_1}-u^*\| \\
                        &\leq \lambda\| u_{k_1}-u^* \|^2+\frac{1}{2}\| \hat{u}_{0}^{k_1}-u^*\| \\
                        &\leq \lambda\tau_1\| u_{k_1}-u^* \|+\frac{1}{2}\| u_{k_1}-u^*\| \leq c_2 \| u_{k_1}-u^*\|,
\end{split}
\end{equation}
and hence \eqref{e52} follows for $i=1$. By \eqref{e51} and the fact that $u_{k_1}=u_1 \in \mathcal{B}_{\zeta}(u^*)$, it follows that  $u_{k_2} \in \mathcal{B}_{\zeta}(u^*) \subset \mathcal{B}_{\tau}(u^*)$. Together with the inclusion in \eqref{e50}  we obtain that   $u_{k} \in \mathcal{B}_{\tau}(u^*)$  for all $k \in \{0,1,\dots, k_2\}$.

 To carry out the induction step we assume that for an index $k_i$  we have  $u_{k_i} \in \mathcal{B}_{\zeta}(u^*)$ and,  $u_{k} \in \mathcal{B}_{\tau}(u^*)$  for all $k \in \{0,1,\dots, k_i\}$. We will show that there exists an index $k_{i+1}>k_{i} $ with  $k_{i+1}-k_{i} \leq m$ such that  $u_{k_{i+1}} \in \mathcal{B}_{\zeta}(u^*)$,   $u_{k} \in \mathcal{B}_{\tau}(u^*)$  for all $k \in \{0,1,\dots, k_{i+1}\}$,  and \eqref{e52} holds.

Since $u_{k_i}, u_{k_i-1} \in \mathcal{B}_{\tau}(u^*)$ we have $\alpha_{k_i} \in [\alpha_{\inf},\alpha_{\sup}]$. Moreover,  due to \eqref{e49}, there is an integer $j_i \leq m$  with the property that
\begin{equation*}
\|\hat{u}^{k_i}_{j_i}-u^*\| \leq \frac{1}{2} \|\hat{u}^{k_i}_{0}-u^* \| = \frac{1}{2} \|u_{k_i}-u^* \|.
\end{equation*}
Due to \eqref{e50},  by defining  $ k_{i+1} = k_i+j_i>k_i$ and using the similar argument as in \eqref{e51}, we can show that \eqref{e52} holds and, consequently, we have $u_{k_{i+1}} \in \mathcal{B}_{\zeta}(u^*)$, and  $u_{k} \in \mathcal{B}_{\tau}(u^*)$  for all $k \in \{0,1,\dots, k_{i+1}\}$.

Now,  due to \eqref{e92}, there is a positive constant  $c_1$ such that
\begin{equation}
\label{e54}
\|u_{k+j} - u^*\| \leq c_1\| u_k-u^*\| \quad  \text{ for all } j\in \{1,\dots, m\},
\end{equation}
where $c_1$ depends only on $m$ and the constants $\alpha_{\sup}$ and $\alpha_{\inf}$ which  have been defined in \ref{e46}. Further, for every $k \geq 1$, there exists an integer $i\geq 1$ such that  $ k_i \leq k < k_{i+1}$ with $k \leq k_i+m-1$  and  $k_i \leq m(i-1)+1$. Therefore, $i \geq \frac{k}{m}$ and also by \eqref{e54}, we obtain
\begin{equation*}
\begin{split}
\|u_k-u^*\|  &\leq c_1 \| u_{k_i}-u^*\| \leq c_1 (c_2)^{i-1} \| u_{k_1}-u^*\|  \leq c_1(c_2)^{\frac{k}{m}-1} \| u_{k_1}-u^*\|. 
\end{split}
\end{equation*}
By setting $\theta := (c_2)^{\frac{1}{m}}< 1$, and  $\lambda_1 := \frac{c_1}{c_2}$, we can conclude \eqref{e58}.

We turn to verification of \eqref{e59}. By using the fact that  for every $k \in \mathbb{N}$ the sequence $\{u_k\}_k$ lies in $\mathcal{B}_{\tau}(u^*)$,  the property \eqref{e60}, and \eqref{e58}, we obtain
\begin{equation*}
\begin{split}
\|\mathcal{G}_k\| \leq \alpha_{\sup}\|u_k-u^*\| \leq  \alpha_{\sup}\lambda_1\theta^k \| u_{1}-u^*\| \leq \frac{\alpha_{\sup} \lambda_1}{\alpha_{\inf}} \theta^{k} \|\mathcal{G}_1\|.
\end{split}
\end{equation*}
By setting $\lambda_2:= \frac{\alpha_{\sup} \lambda_1}{\alpha_{\inf}}$ we complete the proof.
\qed \end{proof}
\section{Mesh Independence Principle}
\label{Sec3}
In this section, we investigate finite-dimensional approximations of Algorithm \ref{BBa}. More specifically we investigate the dependence of the iteration count of the algorithm to achieve a desired accuracy of the residue under finite-dimensional approximations. We note that our objective here  is not to estimate the error between the solutions of the discretized problem and continuous one.

 Thus let $\{\mathcal{H}^h\}_h$  be a family  of finite-dimensional Hilbert spaces indexed by some real number $h>0$, and endowed with inner products and their associated norms denoted by $(\cdot,\cdot)_h$ and $\|\cdot\|_h$, respectively.  Let  $\mathcal{G}^h: \mathcal{H}^h \to \mathcal{H}^h$ denote continuous nonlinear mappings which will be required to approximate $\mathcal{G}$ in a sense to be made precise in Assumption A2 below. We then consider the family of problems:
\begin{equation}
\tag{${EP}^h$}
\label{EPN}
\text{ Find } u^{*h} \in \mathcal{H}^h  \text{ such that  } \quad  \mathcal{G}^h(u^{*h})=0.
\end{equation}
Throughout this section we pose the following assumption:
\begin{itemize}
\item[A0: ] The assumptions of Theorem \ref{Theo3} in Section \ref{sec2.2} hold and we denote by  $\{u_k\}_k$ the sequence  generated by Algorithm \ref{BBa} which enjoys the properties asserted in Theorem \ref{Theo3}.
\end{itemize}
In particular, it is assumed that  $\|u_0 -u^*\|$ and $\| u_1-u^*\|$ are sufficiently small ($<\zeta$ with $\zeta$ defined in Theorem \ref{Theo3}) unless $\mathcal{F}$ is a strictly convex quadratic function. For the case of strictly convex quadratic functions, $u_0$ and $ u_1$ can be chosen from the whole of $\mathcal{H}$.

To describe the family of approximating sequences we choose $u^h_0, u^h_1 \in \mathcal{H}^h$ and update $u^h_k$, for $k=1,\dots$ by
\begin{equation}
\label{e1d}
u^h_{k+1} = u^h_k - \frac{1}{\alpha^h_k}\mathcal{G}^h_k,
\end{equation}
where $\mathcal{G}^h_k := \mathcal{G}^h(u^h_k)$ and the step-size $\alpha^h_k$ is chosen according to either
\begin{equation}
\label{e2d}
\alpha^{BB1,h}_k := \frac{(\mathcal{S}^h_{k-1},\mathcal{Y}^h_{k-1})_h}{(\mathcal{S}^h_{k-1},\mathcal{S}^h_{k-1})_h}, \text{ or }
\alpha^{BB2,h}_k := \frac{ (\mathcal{Y}^h_{k-1},\mathcal{Y}^h_{k-1})_h}{(\mathcal{S}^h_{k-1},\mathcal{Y}^h_{k-1})_h}.
\end{equation}
 Here we have set $\mathcal{S}^h_{k-1}:=u^h_k-u^h_{k-1}$ and $\mathcal{Y}^h_{k-1}:=\mathcal{G}^h_k-\mathcal{G}^h_{k-1}$.  We should point out that  the inner product on  $\mathcal{H}^h$ will typically reflect the norm on $\mathcal{H}$. It should not be thought of as the canonical inner-product in
 $\mathbb{R}^{N(h)}$.

Let us now formulate some additional notation and assumptions that we require for the main result of this section.
Suppose that  $\{\mathbb{P}^h\}_h$ is  a family of linear  `prolongation' operators
\begin{equation*}
\mathbb{P}^h: \mathcal{H}^h \to \mathcal{H}.
\end{equation*}
We use the following notion of convergence in the space  $\mathcal{H}$. A sequence $u^h \in \mathcal{H}^h$ is $\mathcal{H}$-convergent to $u \in \mathcal{H}$ if
\begin{equation*}
\lim_{h \downarrow 0} \|\mathbb{P}^h u^h-u\| = 0.
\end{equation*}
We have to assume that the discrete inner products approximate the original one in the following sense:
\begin{itemize}
  \item[A1: ] If $u^h \overset{\mathcal{H}}{\rightarrow} u$  and $z^h \overset{\mathcal{H}}{\rightarrow} z$  for $u,z \in \mathcal{H}$, then
 \begin{equation}
 \label{e177}
 \lim_{h \downarrow 0} (u^h,z^h)_{h} = (u,z).
\end{equation}
\end{itemize}
Moreover we need the following approximation property of  $\mathcal{G}$  by the family  $\mathcal{G}^h$.
\begin{itemize}
\item[A2: ] Suppose that $\mathcal{G}(u^*) = 0$. Then,  if $u^h \overset{\mathcal{H}}{\rightarrow}u$ with $u$ in a neighborhood of $u^*$, then
\begin{equation}\label{eqkk10}
\mathcal{G}^h(u^h) \overset{\mathcal{H}}{\rightarrow} \mathcal{G}(u).
\end{equation}
 \end{itemize}

 \begin{remark}\label{rekk1}
In applications it can occur that the convergence specified in \eqref{eqkk10} requires additional regularity of $u$ and $\mathcal{G}(u)$. In this case one assumes the existence of a subspace $\mathcal{W}$  in $\mathcal{H}$ of more regular functions, and one needs to  assure that the limit of the iterations remains in $\mathcal{W}$. In this case Assumption A2 is replaced by A2' below. For details we refer to \cite{MR912453}, for instance.
 \begin{itemize}
\item[A2': ]  There is $u^* \in \mathcal{W}$ with  $\mathcal{G}(u^*) = 0$, such that $\mathcal{G}$ is well-defined for all $u \in \mathcal{W}$ sufficiently near $u^*$ with respect to the $\mathcal{H}$-norm. Moreover, if $u\in \mathcal{W}$ with $\|u-u^*\|$ sufficiently small and  $u^h \overset{\mathcal{H}}{\rightarrow}u$, then $\mathcal{G}(u)\in \mathcal{W} $ and
\begin{equation*}
\mathcal{G}^h(u^h) \overset{\mathcal{H}}{\rightarrow} \mathcal{G}(u).
\end{equation*}
 \end{itemize}
\end{remark}
\begin{theorem}
Suppose that Assumptions A0-A2 hold.  Moreover, let $u^{h}_i \overset{\mathcal{H}}{\rightarrow} u_i$  for $i=0,1$  with $u_1\neq u_0$ and $u^h_1\neq u^h_0$.  Then for any $k' \geq 1$, we have
\begin{equation}
\label{e40}
\lim_{h \downarrow 0} \max_{1\leq k \leq k'} \|\mathbb{P}^h u_{k}^h-u_k\| = 0.
\end{equation}
\end{theorem}
\begin{proof}
 Using \eqref{e1d} and the triangle inequality we obtain
\begin{equation}
\label{e175}
 \|\mathbb{P}^h u_{k+1}^h-u_{k+1}\|  \leq  \|\mathbb{P}^h u_{k}^h-u_{k}\|+\abs{\frac{1}{\alpha^h_k}-\frac{1}{\alpha_k}}\| \mathbb{P}^h\mathcal{G}^h_k \| +\abs{\frac{1}{\alpha_k}}\| \mathbb{P}^h \mathcal{G}^h_k-\mathcal{G}_k \|,
\end{equation}
for every $k\geq 1$.  Then, proceeding by induction,  using \eqref{e177} and \eqref{eqkk10}, and passing the limit in \eqref{e2d}  and  \eqref{e175}, it can be shown that \eqref{e40} is true for every $k'\geq 1$.
\qed \end{proof}
The termination condition for \ref{EPN} is based on the norm of the gradients for the  approximated and the original problem. Thus for $\epsilon > 0$ the iteration is terminated according to
\begin{equation}
\label{terminCond}
\begin{split}
\| \mathcal{G}^h_k \|_h < \epsilon, \quad  \text{ and } \quad \| \mathcal{G}_k \| < \epsilon,
\end{split}
\end{equation}
where $\epsilon$ is a sufficiently small positive number. In order to investigate the behaviour of convergence of the approximated problem with respect to the original problem, we consider the following quantities:
\begin{equation*}
\begin{split}
k^*(\epsilon) := \min \{ k \in \mathbb{N}:  \| \mathcal{G}_k\| < \epsilon \},   \qquad k^*_h(\epsilon) := \min \{ k \in \mathbb{N}:  \| \mathcal{G}^h_k\|_h < \epsilon \},
\end{split}
\end{equation*}
where $k^*(\epsilon)$ and $k^*_h(\epsilon)$ are the smallest iteration numbers for which the norm of corresponding gradients is less than $\epsilon$. In the following we study the relation between $k^*(\epsilon)$ and $k^*_h(\epsilon)$.

\begin{theorem}
\label{Theo2}
Suppose that Assumptions A0-A2 hold.  Further, let $u^{h}_i \overset{\mathcal{H}}{\rightarrow} u_i$  for $i=0,1$ with $u_1\neq u_0$ and $u^h_1\neq u^h_0$.  Then for any given numbers $\epsilon>0$ and  $\delta>0$, there exists a number $h_{\delta ,\epsilon} >0$ such that
\begin{equation}
\label{e43ac}
k^*(\epsilon+\delta) \leq k^*_h(\epsilon)\leq k^*(\epsilon)
\end{equation}
 for every $h \in (0, h_{\delta,\epsilon}]$.
\end{theorem}
\begin{proof}
Due to \eqref{e40} and  A2, we have for every $k$ that
\begin{equation}
\label{e94}
\mathcal{G}^h_k \overset{\mathcal{H}}{\rightarrow} \mathcal{G}_k
\end{equation}
and by  A1, we obtain
\begin{equation}
\label{e42}
\lim_{h \downarrow 0} \|\mathcal{G}^h_k \|_h = \| \mathcal{G}_k\|.
\end{equation}
Now,  we show that $\| \mathcal{G}^h_k\|_h < \epsilon$  for a sufficiently small $h>0$, provided that $\| \mathcal{G}_k\| < \epsilon$ holds for an iterate $k$. Since $\| \mathcal{G}_{k^*(\epsilon)}\| < \epsilon$, there exists a positive number $\zeta:=\zeta(\epsilon)$ such that  $\| \mathcal{G}_{k^*(\epsilon)}\| +\zeta < \epsilon$. Moreover, due to \eqref{e42}, there exists a positive number $h_{\epsilon}>0$ such that for every $h \in (0, h_{\epsilon}]$ we have
\begin{equation}
\label{e95}
\abs{\|\mathcal{G}^h_{k^*(\epsilon)} \|_h - \| \mathcal{G}_{k^*(\epsilon)}\|} \leq \zeta.
\end{equation}
Hence, for every $h \in (0, h_{\epsilon}]$, we obtain
\begin{equation*}
\begin{split}
\|\mathcal{G}^h_{k^*(\epsilon)}\|_h &=\| \mathcal{G}_{k^*(\epsilon)} \|+\|\mathcal{G}^h_{k^*(\epsilon)}\|_h - \| \mathcal{G}_{k^*(\epsilon)} \| \leq \|\mathcal{G}_{k^*(\epsilon)}\| + \zeta  <   \epsilon,
\end{split}
\end{equation*}
and, thus, we have
\begin{equation*}
k^*_{h}(\epsilon) \leq k^*(\epsilon)  \quad \text{ for every }  h \in (0, h_{\epsilon}],
\end{equation*}
which implies the second inequality  in \eqref{e43ac}.  Now assume that $\delta>0$ be given.  Then due to \eqref{e42} we have
\begin{equation}
\label{e43}
\lim_{h \downarrow 0} \max_{1 \leq k < k^*(\delta +\epsilon)} \abs{ \|\mathcal{G}_k^h\|_{h}-\| \mathcal{G}_k\| } = 0.
\end{equation}
By the definition of $k^*(\delta +\epsilon)$, we have
\begin{equation}
\label{e43aa}
\| \mathcal{G}_k\| \geq \delta + \epsilon    \quad  \text{ for all }  k < k^*(\delta +\epsilon).
\end{equation}
Moreover due to \eqref{e43}, there exists a positive number $h_{\delta}$ such that
\begin{equation}
\label{e43ab}
 \abs{ \|\mathcal{G}_k^h\|_{h}-\| \mathcal{G}_k\| } \leq \max_{1 \leq k' < k^*(\delta+\epsilon)} \abs{ \|\mathcal{G}_{k'}^h\|_{h}-\| \mathcal{G}_{k'}\| } \leq \delta  \quad  \text{ for all }  h \in (0 , h_{\delta}] \text{ and } k < k^*(\delta +\epsilon).
\end{equation}
Using \eqref{e43aa} and \eqref{e43ab} we infer for every $h \in (0, h_{\delta}]$ and $k < k^*(\delta + \epsilon)$  that
\begin{equation*}
\begin{split}
\| \mathcal{G}^h_k\|_h &\geq \|\mathcal{G}_k\|- \delta \geq  \delta +\epsilon - \delta  = \epsilon,
\end{split}
\end{equation*}
and, thus, $k^*_{h}(\epsilon) \geq k^*(\delta+\epsilon)$ for every $h \in (0, h_{\delta}]$.  Now for the choice of  $h_{\delta, \epsilon} :=\min\{h_{\delta}, h_{\epsilon}\}$, the relation \eqref{e43ac} holds for every $h \in   (0, h_{\delta, \epsilon}]$ and we are finished with the proof.

\qed \end{proof}
\begin{theorem}
\label{theo5}
Suppose that Assumptions A0-A2 hold.  Further  assume that  $u^{h}_i \overset{\mathcal{H}}{\rightarrow} u_i$  for $i=0,1$ with $u_1\neq u_0$ and $u^h_1\neq u^h_0$. Then for  each $\epsilon>0$ there exists  $h_{\epsilon}>0$ such that
\begin{equation*}
k^*(\epsilon)-\ell \leq k_h^*(\epsilon) \leq k^*(\epsilon) \quad \text{ for every } h \in (0, h_{\epsilon}],
\end{equation*}
where the integer $\ell>0$ is independent of $h$ and $\epsilon$.
 \end{theorem}
\begin{proof}
Theorem \ref{Theo3} implies  $R$-linear convergence of $u_k \to u^*$. It can be shown as in the proof of Theorem \ref{Theo3} that there exist a positive integer $m$,   positive numbers $c_2<1$ and $\zeta \leq\tau$,  and a subsequence of indices $\{k_i\}_i \in\mathbb{N}$ with  $k_1=1$, for which we have
\begin{equation}
\label{e61cc}
u_k \in \mathcal{B}_{\zeta}(u^*)   \quad \text{ for every }    k \geq k_1,
\end{equation}
and
\begin{equation}
\label{e61}
k_{i+1}-k_{i} \leq m \quad \text{ and }  \quad \| u_{k_{i+1}}- u^* \| \leq c_2 \|u_{k_{i}}-u^* \|,  \text{ for all } i \geq 1.
\end{equation}
Moreover, as mentioned in the proof of Theorem \ref{Theo3}, there exists a number $c_1>0$ such that
\begin{equation}
\label{e61a}
\|u_{k+j} - u^*\| \leq c_1\| u_k-u^*\| \quad  \text{ for all } j\in \{1,\dots, m\} \text{ and any }  k\geq k_1.
\end{equation}
Let us first denote the integer $q^*$ as the smallest integer for which $c_2^{q^*} < \frac{\alpha_{\inf}}{c_1\alpha_{\sup}}$ holds. The existence of  such $q^*$ is guaranteed since $c_2<1$.  Next, we show for every $k\geq k_1$ that there exists a positive integer $i^+(k)\leq m(q^*+1)-1=:\ell$ such that
\begin{equation}
\label{e61b}
\| \mathcal{G}_{k+i^+(k)}\| < \|\mathcal{G}_k\| .
\end{equation}
 For every  $k \geq k_1$,  the exists an index $i$ such that $k_i \leq k < k_{i+1}$.  Due to \eqref{e60}, \eqref{e61}, \eqref{e61a}, and the definition of $q^*$, we obtain
\begin{equation*}
\begin{split}
\| \mathcal{G}_{k_{i+q^*+1}}\| &\leq \alpha_{\sup} \|u_{k_{i+q^*+1}} -u^* \| \leq \alpha_{\sup} c^{q^*}_2 \| u_{k_{i+1}}-u^*\| \leq \alpha_{\sup} c_1c^{q^*}_2 \| u_{k}-u^*\| \\
& \leq \frac{\alpha_{\sup} c_1c^{q^*}_2}{\alpha_{\inf}}\| \mathcal{G}_k\| <\| \mathcal{G}_k\|.
\end{split}
\end{equation*}
By setting $i^+(k):=k_{i+q^*+1}-k$, we have $i^+(k) \leq \ell$ and we are finished with the verification of   \eqref{e61b}.

  Now, due to the definition of $k^*(\epsilon)$, we have  $\|  \mathcal{G}_k \| \geq \epsilon$ for every  $k < k^*(\epsilon)$. We will next show that  for every  $k^*(\epsilon)\geq \ell$  that
\begin{equation}
\label{e61c}
\|\mathcal{G}_k \| > \epsilon  \quad \text{ for every }  k < k^*(\epsilon)-\ell.
\end{equation}
Suppose on contrary that there exists an index $\bar{k}< k^*(\epsilon)-\ell$ with  $\|\mathcal{G}_{\bar{k}} \| = \epsilon$. Then due to \eqref{e61b} there exists an integer $i^+(\bar{k}) \leq \ell$ such that we have $\| \mathcal{G}_{\bar{k}+i^+(\bar{k})}\| < \|\mathcal{G}_{\bar{k}}\| = \epsilon$  with  $ \bar{k}+i^+(\bar{k}) \leq \bar{k}+\ell < k^*(\epsilon)$, and this contradicts the definition of $k^*(\epsilon)$. Hence, \eqref{e61c} holds.

Due to \eqref{e61c}, for  $k<k^*(\epsilon)-\ell$  there exist strictly positive numbers $\{\delta_k\}_k$ such that  $\|\mathcal{G}_k \| = \epsilon +\delta_k$ for $k<k^*(\epsilon)-\ell$. By setting $0<\delta :=\min\{\delta_k : k< k^*(\epsilon)-\ell\}$, we obtain
\begin{equation*}
\|\mathcal{G}_k \| \geq  \epsilon +\delta    \quad \text{ for every }  k < k^*(\epsilon)-\ell.
\end{equation*}
Therefore we conclude that $k^*(\epsilon)-\ell \leq k^*(\epsilon +\delta)$. Due to Theorem \ref{Theo2}, for $\epsilon>0$ and  $\delta>0$, there exists a  number $h_{\epsilon}>0$ such that we have
\begin{equation}
\label{e44}
k^*(\epsilon)-\ell \leq k^*(\epsilon+\delta) \leq k^*_h(\epsilon)\leq k^*(\epsilon)  \quad \text{ for every } h \in (0, h_{\epsilon}].
\end{equation}
This concludes the proof.
\qed \end{proof}
\begin{remark}
\label{Remark4}
In the case of quadratic functions \eqref{QP}, due to Lemma \ref{R-1},  inequality  \eqref{e44} holds for $\ell = m$ and all initial iterates $u_0, u_1 \in  \mathcal{H}$ with $u_0\neq u_1$. In particular, if also $\delta_{\sup} < 2\delta_{\inf}$, then \eqref{e44} holds for $\ell=1$.
\end{remark}

\begin{remark}
In general, the sequence  $\{\| \mathcal{G}_k\|  \}_k$ corresponding to Algorithm \ref{BBa} is not monotonically decreasing.  This is the reason why we  have to introduce  $\ell$ in Theorem \ref{theo5}  which can possibly be larger than $1$.  In the case that  $\delta_{\sup} < 2\delta_{\inf}$, $\{\| \mathcal{G}_k\| \}_k$ is monotone decreasing and as a consequence $\ell =1$.
\end{remark}

\section{Application to Optimal Control Problems with PDEs}
In this section, we will apply Algorithm \ref{BBa} to optimal control problems which are governed by three types of  partial differential equations, including an elliptic,  a hyperbolic, and a parabolic problem. We introduce these problems in reminder of this section. For the sake of brevity, finite-dimensional approximation is only discussed for the elliptic case.
\subsection{Dirichlet Optimal Control for the Poisson Equation}
\label{DirichletPoisson}
\subsubsection{Continuous Problem}
\label{SubDirichletPoisson}
In this subsection,  we consider the following elliptic Dirichlet boundary control problem
\begin{align}
\label{e147}
&\min_{u\in L^2(\Gamma)} J(u,y):= \frac{1}{2}\|y-y_d\|^2_{L^2(\Omega)}+ \frac{\beta}{2}\| u\|^2_{L^2(\Gamma)},   \\
&\text{subject to }
\begin{cases} \label{e148}
-\Delta y = f  &\text{ in }    \Omega,\\
 y = u  &\text{ on }   \Gamma,
\end{cases}
\end{align}
on an open convex bounded polygonal set $\Omega \subset \mathbb{R}^2$  with boundary denoted by $\Gamma :=\partial \Omega$. We assume that $f,  y_d\in L^2(\Omega)$ and $\beta >0$. Then, for a given  $(u,f) \in L^2(\Gamma)\times L^2(\Omega)$, the solution $y(u,f) \in L^2(Q)$ of \eqref{e148} exists in a very weak sense and it satisfies the following variational equation
\begin{equation*}
(y, -\Delta \varphi)_{L^2(\Omega)}+(u,\partial_{\nu}\varphi)_{L^2(\Gamma)} = (\varphi, f)_{L^2(\Omega)}     \quad  \text{ for all } \varphi \in H^2(\Omega)\cup H_0^1(\Omega).
\end{equation*}
The corresponding  solution operator defined by $(u,f) \mapsto y(u,f)$ is a continuous operator from $L^2(\Gamma) \times  L^2(\Omega)$ to $L^2(\Omega)$. See e.g., \cite{MR775683,MR1173209}. Moreover, the linear operators  $\mathcal{L}: L^2(\Gamma) \to L^2(\Omega)$ defined by $u \mapsto y(u,0)$, and  $\Pi: L^2(\Omega) \to L^2(\Omega)$ defined by $f \mapsto y(f,0)$ are continuous. Then by defining  $\mathcal{X}:=L^2(\Omega)$,  $\mathcal{H}:=L^2(\Gamma)$ and $\psi:= -\Pi f +  y_d$, we can express the optimal control problem \eqref{e147}-\eqref{e148} as the following linear least squares problem
\begin{equation}
\label{e86}
\tag{LS}
\min_{u \in \mathcal{H}}  \mathcal{F}(u) := \frac{1}{2} \| \mathcal{L}u-\psi \|^2_{\mathcal{X}} +\frac{\beta}{2}\|u\|^2_{\mathcal{H}}.
\end{equation}
By a short computation, it can be shown that the problem \eqref{e147}-\eqref{e148}  can be written in the form of \eqref{QP}, where  $\mathcal{A}:=\mathcal{L^*L}+\beta \mathcal{I} $  with $\mathcal{L}^* :\mathcal{X} \to \mathcal{H}$ defined as the adjoint operator of  $\mathcal{L}$, and $b:=\mathcal{L}^*\psi$.  Clearly, the operator $\mathcal{A}$ is uniformly  positive,  bounded,  and self-adjoint on the Hilbert space $\mathcal{H}$ and thus the existence and uniqueness  of the solution to the problem \eqref{e147}-\eqref{e148} can be obtained  due the fact that  $\mathcal{A}$ has a bounded inverse.
\begin{remark}
According to \cite{MR2084239}[Theorem 4.2], for each pair $(f,u) \in L^2(\Omega)\times L^2(\Gamma)$, the solution $y(f,u)$ to \eqref{e148} belongs to the space $H^{\frac{1}{2}}(\Omega)$, which is continuously and compactly embedded to $L^2(\Omega)$. Therefore the linear operator $\mathcal{L}: \mathcal{H} \to \mathcal{H}$ is compact, and we conclude that  $\mathcal{A}$ has the form \eqref{poco}.

For every $u \in \mathcal{H}$,  the derivative of $\mathcal{F}$ at $u$ in direction $\delta u \in  \mathcal{H}$ can be expressed by
\begin{equation}
\label{e155}
\mathcal{F}'(u)\delta u = (\mathcal{L}^* (\mathcal{L}u-\psi)+\beta u , \delta u  ),
\end{equation}
and the gradient of  $\mathcal{F}$ at $u$ is identified by  $\mathcal{G}(u) =\mathcal{L}^* (\mathcal{L}u-\psi)+\beta u$. Alternatively, if we consider  the solution $p(u)\in H^2(\Omega) \cap H^1_0(\Omega)$  of the adjoint equation
\begin{equation}
\begin{cases} \label{e154}
-\Delta p =y(u,f)-y_d     &\text{ in }    \Omega,\\
 p = 0  &\text{ on }   \Gamma,
\end{cases}
\end{equation}
where $y(u,f) \in L^2(\Omega)$ is the solution of \eqref{e148},  then the directional derivative \eqref{e155} and the corresponding  gradient $\mathcal{G}$ at point $u$ can be rewritten as
\begin{equation}
\label{e156}
\begin{split}
\mathcal{F}'(u) \delta u = ( \partial_{\nu}p(u)+\beta u , \delta u )  \text{ for all } \delta u \in \mathcal{H},\text{ and } \quad
\mathcal{G}(u)  =  \partial_{\nu}p(u)+\beta u  \text{ in }  \mathcal{H}.
\end{split}
\end{equation}
\end{remark}
 For the global minimizer $u^* \in \mathcal{H}$ to \eqref{e86}, the first-order optimality condition can be expressed as
\begin{equation}
\label{e90}
(\mathcal{\mathcal{L^*L}+\beta \mathcal{I}})u^*= \mathcal{L}^*\psi,
\end{equation}
 which can be rewritten, equivalently, as the following systems of equations
 \begin{equation*}
 \begin{cases}
 y^* = y(u^*,f)    &\text{ in } L^2(\Omega),   \\
 \partial_{\nu}p^*=-\beta u^*& \text{ in } L^2(\Gamma), \\
 -\Delta p^* = y^*-y_d   &\text{ in }  L^{2}(\Omega), \text{with } p^*=0  \text{ on } \Gamma.
 \end{cases}
 \end{equation*}
 \subsubsection{Discretized Problem }
 \label{subsubdis}
For the  discretization of \eqref{e147}-\eqref{e148}, we use finite elements. Let us consider the regular family of triangulations $\{\mathcal{T}_h\}_{h>0}$ of $\overline{\Omega}$ with $\overline{\Omega}=\cup_{T \in \mathcal{T}_h}T$ and the mesh-size defined by $h:= \max\{\diam(T) : T \in \mathcal{T}_h \}$. Let $\{x_j\}_{1 \leq j \leq N(h)}$ be the nodes which lies on the boundary with the numbering which starts at the origin in the counterclockwise and $x_{N(h)+1}=x_1$. Then we define the space of discretized control by
\begin{equation*}
\mathcal{H}^h :=\{u^h \in C(\Gamma) : u^h|_{[x_j,x_{j+1}]} \in \mathcal{P}^1 \text{ for } j=1,\dots, N(h)\},
\end{equation*}
and, we consider the space  $V^h \subset H^1(\Omega)$ defined by
\begin{equation*}
V^h :=\{y^h \in C(\bar{\Omega}) : y^h|_T \in \mathcal{P}^1 \text{ for every } T\in \mathcal{T}_h \},
\end{equation*}
where $\mathcal{P}^1$ is the space of polynomials of degree less than or equal to 1. Further we set $V^h_0 := V^h \cap H^1_{0}(\Omega)$.  The space $\mathcal{H}^{h}$ is formed by the  restriction of the functions of $V^h$ to $\partial\Omega$.  Clearly, we have $\mathcal{H}^h \subset \mathcal{H}$ and, as a result, the finite-dimensional space $\mathcal{H}^h$ is endowed with the inner product and the norm introduced by the space $\mathcal{H}= L^2(\Gamma )$.  Then, naturally, the prolongation operator $\mathbb{P}^h:\mathcal{H}^h \to \mathcal{H}$ is defined to be the canonical injection operator i.e., $\mathbb{P}^h(u^h) = u^h$ for every $u^h \in \mathcal{H}^h$.
Let us consider the orthogonal projection operator $\Pi^h:  \mathcal{H} \to \mathcal{H}^h$  defined by
\begin{equation*}
( \Pi^hv, u^h)_{ \mathcal{H}} = (v,u^h)_{ \mathcal{H}}     \text{ for all  } u^h\in \mathcal{H}^h.
\end{equation*}
It satisfies the following estimate
\begin{equation}
\label{e171}
\| u- \Pi^h u \|_{\mathcal{H}}  \leq ch^{\frac{1}{2}}\| u\|_{H^{\frac{1}{2}}(\Gamma)},
\end{equation}
for every $u \in  H^{\frac{1}{2}}(\Gamma)$, see, e.g., \cite{MR2084239,MR2272157}. For every $u \in  \mathcal{H}$  we consider the unique discrete solution  $y^h(u) \in V^h$ satisfying
\begin{equation}
\label{e151}
\begin{cases}(\nabla y^h, \nabla \phi^h )  = (f,  \phi^h )  \text{ for all } \phi^h \in V^h,  \\
y^h|_{\Gamma} = \Pi^h u.
 \end{cases}
\end{equation}
Then we can define the discrete objective function  in $\mathcal{H}$ by
\begin{equation}
\label{e152}
 J^h(u,y^h(u)):= \frac{1}{2}\|y^h(u)-y_d\|^2_{L^2(\Omega)}+ \frac{\beta}{2}\| u\|^2_{L^2(\Gamma)}.
\end{equation}
The finite-dimensional approximation of  \eqref{e147}-\eqref{e148} can be expressed as
\begin{equation}
\label{e153}
\min_{u^h\in \mathcal{H}^h} \mathcal{F}^h(u^h) = \min_{u^h\in \mathcal{H}^h} J^h(u^h,y^h(u^h)).
\end{equation}
Existence of a solution to \eqref{e153} follows by similar arguments as for the continuous problem. Given $u\in \mathcal{H}$ we consider the adjoint state  $p^h(u) \in V^h_0$ as  the solution of
\begin{equation}
\label{e157}
 (\nabla p^h(u),  \nabla \psi^h )_{L^2(\Omega)} = (y^h(u)-y_d, \psi^h )_{L^2(\Omega)}   \text{ for all } \psi^h \in V^h_0.
 \end{equation}
In order to compute the gradient of $\mathcal{F}^h$, analogously to the expression  \eqref{e156}, we need to characterise a discrete normal derivative $\partial^h_{\nu} p^h(u)$.  For every $ u \in\mathcal{H}$, similarly to  \cite{MR2272157}[ Proposition 4.2], $\partial^h_{\nu} p^h(u) \in \mathcal{H}^h$ is characterized as the unique solution of the following variational problem
\begin{equation*}
(\partial^h_{\nu} p^h(u),\varphi^h)_{\mathcal{H}} = (\nabla p^h(u),  \nabla \varphi^h )_{L^2(\Omega)}-(y^h(u)-y_d, \varphi^h )_{L^2(\Omega)}    \text{ for all }  \varphi^h \in V^h,
\end{equation*}
where $p^h(u) \in V^h_0$ is the solution of \eqref{e157}. Next, we prove the following useful estimate.
\begin{lemma}
There exists a constant $c$ depending on $f$ and  $y_d$, and independent of $h$ such that
\begin{equation}
\label{e158}
\|\partial_{\nu} p(u)- \partial^h_{\nu} p^h(v)\|_{\mathcal{H}}  \leq c \left( \| u-v\|_{\mathcal{H}}+h^{\frac{1}{2}}(1+ \| v\|_{\mathcal{H}}) \right) \text{ for all } u,v \in \mathcal{H}.
\end{equation}
\end{lemma}
\begin{proof}
This proof is based on the results from \cite{MR2272157}, where $u \in L^{\infty}(\Omega)$ was used in the context of semilinear elliptic equation.  First, using a similar argument as in \cite{MR2084239,MR2272157},  one can show that
\begin{equation}
\label{e161}
\| y(u)-y^h(u)\|_{\mathcal{H}} \leq c (1+\| u\|_{\mathcal{H}})^2h^{\frac{1}{2}},
\end{equation}
where the constant $c$ depends on $f$. From \eqref{e161}, it follows that
\begin{equation*}
\| y(u)-y^h(v)\|_{\mathcal{H}} \leq c \left(\|u-v\|_{\mathcal{H}}+h^{\frac{1}{2}}(1+\| u\|_{\mathcal{H}})\right) \text{ for all } u,v \in \mathcal{H},
\end{equation*}
Next, we show that
\begin{equation}
\label{e162}
\| \partial_{\nu}p(u) -\partial^h_{\nu}p^h(u) \|_{\mathcal{H}} \leq ch^{\frac{1}{2}}(1+ \| u\|_\mathcal{H}) \text{ for all  }   u\in \mathcal{H}.
\end{equation}
Recall that $p(u) \in H^2(\Omega)\cap H^1_0(\Omega)$ and therefore $\partial_{\nu}p(u) \in H^{\frac{1}{2}}(\Gamma)$. For the left hand-side of \eqref{e162} we obtain
\begin{equation}
\label{e173}
\| \partial_{\nu}p(u) -\partial^h_{\nu}p^h(u) \|_{\mathcal{H}}^2  = \int_{\Gamma}  \abs{ \partial_{\nu}p(u)-\Pi^h \partial_{\nu}p(u) }^2d\mathcal{S}+ \int_{\Gamma}  \abs{ \Pi^h \partial_{\nu}p(u)-\partial^h_{\nu}p^h(u) }^2d\mathcal{S} =: I_1+I_2.
\end{equation}
The last term can be equivalently be expressed as
\begin{equation}
\label{e178}
 I_2 =  \int_{\Gamma} (\partial_{\nu}p(u)-\partial^h_{\nu}p^h(u))(\Pi^h\partial_{\nu} p(u)-\partial_{\nu}^hp^h(u))d\mathcal{S}.
\end{equation}
Let  $w^h \in V^h$ be the solution of the following variational equation
\begin{equation}
\label{e163}
\begin{cases}(\nabla w^h, \nabla \phi^h )  = 0  \text{ for all } \phi^h \in V^h,  \\
w^h|_{\Gamma} = \Pi^h \partial_{\nu}p(u)-\partial^h_{\nu}p^h(u).
 \end{cases}
\end{equation}
Then,  by referring to \cite{MR842125}[Lemma 3.2], we have the following estimate for \eqref{e163}
\begin{equation}
\label{e164}
\| w^h\|_{H^1(\Omega)} \leq c \|  \Pi^h \partial_{\nu}p(u)-\partial^h_{\nu}p^h(u) \|_{H^{\frac{1}{2}}(\Gamma)},
\end{equation}
with a constant $c$ independent of $h$.  Using the definition of $\partial^h_{\nu}p^h(u)$ and  Green formula for  $\partial_{\nu}p(u)$, we obtain
\begin{equation}
\label{e165}
(\partial_{\nu}p(u)-\partial^h_{\nu}p^h(u), \phi^h )_{\mathcal{H}} =( \nabla (p(u)-p^h(u)), \nabla \phi^h )_{L^2(\Omega)}+(y^h(u)-y(u),\phi^h)_{L^2(\Omega)}
\end{equation}
for every $\phi^h \in V^h$. Using \eqref{e178}, \eqref{e163}, and \eqref{e165}, we find
 \begin{equation*}
 I_2 = ( \nabla (p(u)-p^h(u)), \nabla w^h )_{L^2(\Omega)}+(y^h(u)-y(u), w^h)_{L^2(\Omega)}.
 \end{equation*}
 Moreover, we have
\begin{equation}
\label{e166}
(\nabla p^h(u), \nabla w^h )_{L^2(\Omega)} = ( \nabla \mathcal{I}_h p(u), \nabla w^h )_{L^2(\Omega)}=0,
\end{equation}
where $\mathcal{I}_h \in \mathcal{L}(C(\overline{\Omega}),V_0^h)$ stands for the classical interpolation operator, see e.g., \cite{MR1278258}. Due to \eqref{e166} and the definition of $w^h$ from \eqref{e163},  we obtain
 \begin{equation}
 \label{e170}
 I_2 = ( \nabla (p(u)- \mathcal{I}_h  p(u)), \nabla w^h )_{L^2(\Omega)}+(y^h(u)-y(u),w^h)_{L^2(\Omega)}.
 \end{equation}
 Using \eqref{e164}, the interpolation estimate, and the following inverse estimate (see e.g., \cite{MR2084239})
 \begin{equation*}
 \| u^h\|_{H^{\frac{1}{2}}(\Gamma)} \leq Ch^{-\frac{1}{2}}\| u^h\|_{\mathcal{H}}   \text{ for all  } u^h \in \mathcal{H}^h,
\end{equation*}
 we infer that
 \begin{equation}
 \label{e167}
 \begin{split}
 |( \nabla (p(u)- \mathcal{I}_h  p(u)), \nabla w^h )_{L^2(\Omega)}| &\leq \| \nabla (p(u)- \mathcal{I}_h  p(u)) \|_{L^2(\Omega)}  \| w^h \|_{H^1(\Omega)} \\
&\leq ch\| p(u)\|_{H^2(\Omega)}\| w^h|_{\Gamma} \|_{H^{\frac{1}{2}}(\Gamma)}\\
&\leq ch^{\frac{1}{2}} (1+ \| u\|_{\mathcal{H}}) \| w^h|_{\Gamma}\|_{\mathcal{H}} \\
&\leq ch^{\frac{1}{2}}(1+ \| u\|_{\mathcal{H}}) \sqrt{I_2},
 \end{split}
 \end{equation}
where the constant $c$ from the third line of \eqref{e167} depends also on $y_d$. Moreover, due to \eqref{e161}, we can write
 \begin{equation}
 \label{e168}
 \begin{split}
 |(y^h(u)-y(u),w^h)_{L^2(\Omega)}| \leq \|y^h(u)-y(u)\|_{L^2(\Omega)} \| w^h\|_{L^2(\Omega)} \leq  ch^{\frac{1}{2}}(1+ \| u\|_{\mathcal{H}}) \sqrt{I_2}.
 \end{split}
 \end{equation}
From \eqref{e170}, \eqref{e167}, and \eqref{e168}, it follows that
\begin{equation}
\label{e169}
I_2 \leq ch(1+ \| u\|_{\mathcal{H}})^2.
\end{equation}
 Further, using \eqref{e171} we obtain
\begin{equation}
\label{e172}
I_1 \leq ch\| \partial_{\nu}p(u) \|^2_{H^{\frac{1}{2}}(\partial \Omega)} \leq ch\| p(u)\|^2_{H^2(\Omega)} \leq ch(1+ \| u\|_{\mathcal{H}})^2.
\end{equation}
Now, from \eqref{e173}, \eqref{e169}, and \eqref{e172},  we conclude \eqref{e162}. Finally, using \eqref{e162} we can write that
\begin{equation*}
\begin{split}
\|\partial_{\nu} p(u)-\partial^h_{\nu}p^h(v) \|_{\mathcal{H}} &\leq \| \partial_{\nu} p(u)-\partial_{\nu}p(v)\|_{\mathcal{H}} +\| \partial_{\nu} p(v)-\partial^h_{\nu}p^h(v)\|_{\mathcal{H}} \\
& \leq c\| p(u)-p(v)\|_{H^2(\Omega)} + ch^{\frac{1}{2}}(1+ \| v\|_{\mathcal{H}}) \\
& \leq c \left( \| u-v\|_{\mathcal{H}}+h^{\frac{1}{2}}(1+ \| v\|_{\mathcal{H}}) \right),
\end{split}
\end{equation*}
for every $u,v \in \mathcal{H}$ and we are finished with the verification of \eqref{e158}.
\qed \end{proof}

Now we are in the position in which we can verify the assumptions A1-A2 of Section \ref{Sec3}. A1 follows from the definition of $\mathcal{H}^h$ and $\mathbb{P}^h$. To verify A2,  assume that $u^h \overset{\mathcal{H}}{\rightarrow} u $ with $u^h \in \mathcal{H}^h$.  Similarly to \eqref{e156}, the directional derivative and its corresponding  gradient of $\mathcal{F}^h$ of the discretized problem \eqref{e153} at point $u^h$ can be rewritten as
\begin{equation}
\label{e159}
\begin{split}
{\mathcal{F}^h}'(u^h) {\delta u}^h = ( \partial^h_{\nu}p^h(u^h)+\beta u^h , {\delta u}^h )  \text{ for all } {\delta u}^h \in \mathcal{H}^h,\text{ and } \quad
\mathcal{G}^h(u^h)  =  \partial^h_{\nu}p^h(u^h)+\beta u^h  \text{ in }  \mathcal{H}^h.
\end{split}
\end{equation}
Then by  \eqref{e156},  \eqref{e159}, and  \eqref{e158}, we obtain
\begin{equation}
\label{e160}
\| \mathcal{G}(u)- \mathbb{P}^h\mathcal{G}^h(u^h)\|_{\mathcal{H}} \leq  \|\partial_{\nu} p(u)- \partial^h_{\nu} p^h(u^h) \|_{\mathcal{H}} +\beta \| u-u^h \|_{\mathcal{H}} \leq (c+\beta) \| u-u^h\|_{\mathcal{H}} +ch^{\frac{1}{2}}(1+\|u^h\|_{\mathcal{H}}).
\end{equation}
Hence,  $\mathcal{G}^h(u^h) \overset{\mathcal{H}}{\rightarrow} \mathcal{G}(u)$ follows by sending $h$ to zero in \eqref{e160}.
\begin{remark}
Due the fact that $\partial_{\nu}p(u) \in  H^{\frac{1}{2}}(\Gamma)$ for every $u \in \mathcal{H}$, using \eqref{e156} and  Step 4 in Algorithm \ref{BBa}, it is easy to see that for every $u_0,u_1 \in  H^{\frac{1}{2}}(\Gamma)$, the sequence $\{u_k\}_k$ stays in the space $H^{\frac{1}{2}}(\Gamma)$. Moreover, for given  $y_d,f \in L^{p^*}(\Omega)$ with $p^* >2$, we have $y\in W^{1,\overline{p}}(\Omega)$ and $p(y,y_d)\in W^{2,\overline{p}}(\Omega)$ for  $\overline{p} \in (2,p^*]$ depending on $\Omega$, see e.g., \cite{MR2272157}[Theorem 3.4]. Hence, $\{u_k\}_k \subset  W^{1-\frac{1}{\overline{p}},\overline{p}}(\Gamma) \subset C(\Gamma)$ provided that  $u_0,u_1 \in  W^{1-\frac{1}{\overline{p}},\overline{p}}(\Gamma)$. In this case $u^h_0,u^h_1$ can be chosen as $\mathcal{I}^{\Gamma}_hu_0, \mathcal{I}^{\Gamma}_hu_1 \in \mathcal{H}^h$ where $\mathcal{I}^{\Gamma}_h \in \mathcal{L}(W^{1-\frac{1}{\overline{p}},\overline{p}}(\Gamma),\mathcal{H}^h)$ is the standard interpolation operator.
\end{remark}

\subsection{Neumann Optimal Control for the  Linear Wave Equation}
\label{NeumannSubsec}
Let us consider the optimal control problem
\begin{align}
&\min_{u\in L^2(\Sigma_c)} J(u,y):= \frac{\alpha_1}{2}\|y-y_d\|^2_{L^2(Q)}+ \frac{\alpha_2}{2}\|y(T)-z_d\|^2_{L^2(\Omega)}+\frac{1}{2}\| u\|^2_{L^2(\Sigma_c)}, \label{e85}   \\  \label{e84}
&\text{subject to }
\begin{cases}
y_{tt}-\Delta y = f  &\text{ in }    Q,\\
\partial_{\nu}y =u   & \text{ on }  \Sigma_c,\\
y = 0  &\text{ on }     \Sigma_0,\\
y(0)=y^1_0,\quad y_t(0)=y^2_0 &\text{ on } \Omega,
\end{cases}
\end{align}
where $\alpha_1$, $\alpha_2$, and $\beta$ are positive constants,  the desired state $y_d$ and the desired finial state $z_d$ are smooth enough,  $Q:= (0,T)\times\Omega$, $\Sigma_c:= (0,T)\times  \Gamma_c$, $\Sigma_0 :=(0,T)\times \Gamma_0$, and  $\Omega \in \mathbb{R}^n$ is a bounded domain with the smooth boundary $\partial \Omega:=\overline{\Gamma_c} \cup \overline{\Gamma_0}$. Further,  two disjoint components $\Gamma_c$, $\Gamma_0$ are relatively open in $\partial \Omega$.

Before investigating the optimal control problem, we recall some useful results for equation \eqref{e84}. The  operator $A: L^2(\Omega) \supset \mathcal{D}(A) \to L^2(\Omega)$ defined by $Ah = -\Delta h$ is a positive self-adjoint operator with $\mathcal{D}(A):=\{ h \in H^2(\Omega), h|_{\Gamma_0} =  \partial_{\nu}h|_{\Gamma_{c}}=0 \}$.  Thus, we define the spaces  $H_{\Gamma_0}^{\alpha}(\Omega):=\mathcal{D}(A^{\frac{\alpha}{2}})$ for $0\leq\alpha \leq 1$, and by $(H^{\alpha}_{\Gamma_0}(\Omega))^*$ we denote the corresponding dual space.  These spaces are used throughout this subsection.  We use the following notion of  solution \cite{MR0271512,MR0350177}.
\begin{definition}[Very weak solution]  Let $T>0$, and  $(y^1_0, y^2_0,u,f) \in L^2(\Omega)\times (H^1_{\Gamma_0}(\Omega))^* \times L^2(\Sigma_c) \times L^2(0,T;(H^1_{\Gamma_0}(\Omega))^*)$ be given. A function $y \in L^{2}(Q)$ is referred to as the very weak solution of  \eqref{e84}, if the following inequality holds
\begin{equation}
\label{e89}
\begin{split}
\langle f,\varphi &\rangle_{(L^2(0,T;(H^1_{\Gamma_0}(\Omega))^*), L^2(0,T; H^1_{\Gamma_0}(\Omega)))}  =\\
&(g,y)_{L^2(Q)}+(y^1_0, \varphi_t(0))_{L^2(\Omega)}-\langle y^2_0, \varphi(0) \rangle_{ ((H^1_{\Gamma_0}(\Omega))^*,H_{\Gamma_0}^{1}(\Omega))}-(u,\varphi)_{L^2(\Sigma_0)}
\end{split}
\end{equation}
for all $g \in L^2(Q)$, where $\varphi(g) \in C^0([0,T];H^{1}_{\Gamma_0}(\Omega))\cap C^1([0,T]; L^2(\Omega)) $ is  the weak solution of the following backward in time problem
\begin{equation*}
\begin{cases}
\varphi_{tt}-\Delta \varphi = g  &\text{ in }  Q,\\
\partial_{\nu}\varphi =0   & \text{ on }  \Sigma_c,\\
\varphi = 0                       &\text{ on }    \Sigma_0,\\
\varphi(T)=0,\quad \varphi_t(T)=0 &\text{ on } \Omega.
\end{cases}
\end{equation*}
\end{definition}
We have the following existence and regularity results from \cite{MR1108480,MR1133544,MR2149171} for the solution of \eqref{e84}.
\begin{lemma}
For every  $(y^1_0,y^2_0,u,f) \in  L^2(\Omega) \times (H^1_{\Gamma_0}(\Omega))^* \times L^2(\Sigma_c) \times L^2(0,T; (H^{1}_{\Gamma_0}(\Omega))^*)$, equation \eqref{e84} admits a unique very weak solution $y (y^1_0,y^2_0,u,f)$ in the space $C^0([0,T];L^2(\Omega))\cap C^1([0,T]; (H^1_{\Gamma_0}(\Omega))^*)$ satisfying
\begin{equation}
\label{e88}
\begin{split}
\|y\|_{C^{0}([0,T]; L^2(\Omega))}&+\|y_t\|_{C^{0}([0,T];(H^1_{\Gamma_0}(\Omega))^*)} \\
 &\leq c  \left( \|y^1_0\|_{L^2(\Omega)}+\|y^2_0\|_{(H^{1}_{\Gamma_0}(\Omega))^*}+\|u\|_{L^2(\Sigma_c)}+ \|f\|_{L^2(0,T;(H^{1}_{\Gamma_0}(\Omega))^*)} \right),
\end{split}
\end{equation}
where the constant $c_1$ is independent of $y^1_0$, $y^2_0$, $u$, and $f$. Moreover, the solution operator  $L : L^2(\Sigma_c) \to C^0([0,T];H^{\frac{1}{2}}_{\Gamma_0}(\Omega))\cap C^1([0,T]; H^{-\frac{1}{2}}(\Omega))$ defined by $u \mapsto y(0,0,u,0)$  is bounded. Furthermore, the mapping $\Pi :  L^2(\Omega) \times (H^1_{\Gamma_0}(\Omega))^* \times L^2(0,T; (H^{1}_{\Gamma_0}(\Omega))^*) \to C^0([0,T];L^2(\Omega))\cap C^1([0,T]; (H^1_{\Gamma_0}(\Omega))^*)$ defined by $(y^1_0,y^2_0,f) \mapsto y(y^1_0,y^2_0,0,f)$ is  continuous.
\end{lemma}
By considering the following continuous embeddings
\begin{equation*}
\begin{split}
i_1: C^0([0,T];H^{\frac{1}{2}}_{\Gamma_0}(\Omega)) \hookrightarrow L^2(Q),  \qquad i_2: C^0([0,T]; L^2(\Omega)) \hookrightarrow	L^2(Q),
\end{split}
\end{equation*}
and the continuous operator $\delta_T : C^0([0,T]; L^2(\Omega)) \to L^2(\Omega)$ defined by $y \mapsto y(T)$, we can  rewrite the optimal control problem \eqref{e85}-\eqref{e84} in the form \eqref{e86},  where $\mathcal{H}:= L^2(\Sigma_c)$, $\mathcal{X}:=L^2(Q)\times L^2(\Omega)$, and the linear operator $\mathcal{L}: \mathcal{H} \to \mathcal{X}$ and  $\psi \in \mathcal{X}$ are defined as follows
\begin{equation}
\label{e87}
\mathcal{L}u := \begin{pmatrix}
\alpha_1(i_1 \circ L) (u)\\
\alpha_2(\delta_T \circ L) (u) 
\end{pmatrix},  \quad
\psi := \begin{pmatrix}
\alpha_1y_d - \alpha_1(i_2 \circ \Pi) (y^1_0,y^2_0,f)\\
\alpha_2z_d-\alpha_2(\delta_T \circ \Pi) (y^1_0,y^2_0,f) 
\end{pmatrix}.
\end{equation}
Similarly to the previous subsection, the optimal control problem \eqref{e85}-\eqref{e84},  can be also  rewritten in the form of \eqref{QP}, where  $\mathcal{A}:=\mathcal{L^*L}+\beta \mathcal{I} $  with $\mathcal{L}^* :\mathcal{X}^* \to \mathcal{H}$, and $b:=\mathcal{L}^*\psi$.  In addition, due the fact that  the operator $\mathcal{A}$ is uniformly positive,  bounded,  and self-adjoint, the existence and uniqueness  of the solution to optimal control problem \eqref{e85}-\eqref{e84} can be justified due the fact that  $\mathcal{A}$ has a bounded inverse.

\begin{remark}
In the optimal control problem \eqref{e85}-\eqref{e84}, the operator $\mathcal{A}$ is a compact perturbation of the identity, since $\mathcal{L}: \mathcal{H} \to \mathcal{H}$ is compact. Indeed, due to \cite{MR916688}[Corollary 5.], the continuous embedding from the space $C^0([0,T];H^{\frac{1}{2}}_{\Gamma_0}(\Omega))\cap C^1([0,T]; H^{-\frac{1}{2}}(\Omega))$ to the space $C^0([0,T];L^2(\Omega))$ is compact and this implies the compactness of $\delta_T \circ L$ and $i_1 \circ L$. Therefore, due to \eqref{e87}, $\mathcal{L}$ is compact with respect to  the product topology $L^2(Q)\times L^2(\Omega)$.
\end{remark}

Now assume that $u^* \in \mathcal{H}$ is the optimal solution of the optimal control problem \eqref{e85}-\eqref{e84}. Then,  the first-order optimality condition \eqref{e39} can be expressed as \eqref{e90} where the operator $\mathcal{L}$ and the function $\psi$ were defined in \eqref{e87}. Moreover, it can be shown (see \cite{MR2801222,MR0271512,MR2124277}) that \eqref{e90} is equivalent to the condition $\beta u^* =  p^* $ on $\Sigma_c$, where  $p^* \in C^1([0,T];L^2(\Omega))\cap C^0([0,T]; H^1_{\Gamma_0}(\Omega))$ is the weak solution of the following linear wave equation
\begin{equation*}
\begin{cases}
p^*_{tt}-\Delta p^* = -\alpha_1(y^*-y_d)  &\text{ in }  Q,\\
\partial_{\nu} p^* =0   & \text{ on }  \Sigma_c,\\
p^*= 0                       &\text{ on }    \Sigma_0,\\
p^*(T)=0,\quad  p^*_t(T)=\alpha_2(y^*(T)-z_d) &\text{ on } \Omega,
\end{cases}
\end{equation*}
and $y^*= y(y^1_0,y^2_0,f,u^* )$ is the very weak solution of \eqref{e84}.
\subsection{Distributed Optimal Control for the Burgers Equation}
Here we consider the following optimal control problem which consists of minimizing the performance index
\begin{equation}
\label{e98}
J(y,u):=\frac{\alpha_1}{2}\|y-y_d\|_{L^2(Q)}^2+ \frac{\alpha_2}{2}\|y(T)-z_d\|_{L^2(0,1)}^2 + \frac{\beta}{2}\|u\|_{L^2(\hat{Q})}^2,
\end{equation}
subject to the Burgers equation with homogeneous Dirichlet boundary condition.
\begin{equation}
\label{e99}
\begin{cases}
y_t-\vartheta y_{xx}+yy_x=Bu+f, \quad (t,x)\in Q ,\\
      y(t,0)=y(t,1)=0, \quad t \in (0,T),\\
      y(0,x)=y_0(x),\quad x \in (0,1).\\
\end{cases}
\end{equation}
where $\vartheta$, $\beta$ $\alpha_1$, $\alpha_2$, and  $T$ are positive constants, $y(t)=y(t,x) , u(t)=u(t,x)$,  $Q:=(0,T)\times(0,1)$, and $\hat{Q}=(0,T)\times \hat{\Omega}$ where $\hat{\Omega}$ is an open subset of  $(0,1)$. Moreover,  $y_0 \in L^2(0,1)$, $f \in L^2(0,T;H^{-1}(0,1))$, the desired states $y_d$ and $z_d$  are smooth enough,  and the extension operator $B \in \mathcal{L}(L^2(\hat{\Omega}),L^2(0,1))$ is defined by
\begin{equation*}
(Bu)(x) = \begin{cases} u(x), & x \in \hat{\Omega},\\
                        0     & x \in  (0,1)\backslash\hat{\Omega}.
                        \end{cases}
\end{equation*}
Considering the space
\begin{equation*}
W(0,T):=\{\phi: \phi \in L^2(0,T;H_0^1(0,1)) , \phi_t \in L^2(0,T;H^{-1}(0,1))\}.
\end{equation*}
as the space of solutions,  we have the following  notion of weak solution.
\begin{definition}
Let  $(y_0,u,f) \in  L^2(0,1) \times L^2(\hat{Q}) \times L^2(0,T;H^{-1}(0,1)) $ be given.  Then,  a function $y \in W(0,T)$ is referred as a weak solution to \eqref{e99} if $y(0)=y_0$ is satisfied in $L^2(0,1)$ and for almost every $t \in (0,T)$,  the following equality
\begin{equation*}
\langle y_t(t),\varphi \rangle_{H^{-1},H^1_0}+\vartheta( y(t),\varphi)_{H^1_0}+b(y(t),y(t),\varphi)=\langle Bu(t)+f,\varphi \rangle_{H^{-1},H^1_0} \quad \text{ for all }\varphi \in H^1_0(0,1)
\end{equation*}
holds, where the continuous trilinear form $b: H^1_0(0,1)\times H^1_0(0,1)\times H^1_0(0,1) \to \mathbb{R}_+$ is defined as
\begin{equation*}
b(\varphi,\psi,\phi)=\int_0^1\varphi\psi_x\phi dx \quad \text{ for all } \varphi, \psi, \phi \in H^1_0(0,1).
\end{equation*}
\end{definition}
It is known that, for every triple $(y_0,u,f) \in  L^2(0,1) \times L^2(\hat{Q}) \times L^2(0,T;H^{-1}(0,1))$, equation \eqref{e99} admits a unique weak solution $y(y_0,u,f)  \in W(0,T)$ and for this weak solution we have the following estimate
\begin{equation}
\label{e108}
\|y\|_{W(0,T)} \leq C   \left(   \|y_0\|_{L^2(\Omega)}+\|u\|_{L^2(\hat{Q})}+\|f\|_{L^2(0,T;H^{-1}(0,1))}  \right)^2,
\end{equation}
where the constant $C$ depends only on $T$ and $\vartheta$. Now, by setting $X= W(0,T)\times \mathcal{H}$ with $\mathcal{H} := L^2(\hat{Q})$, and $Y:=L^2(0,T;H^{-1}(0,1))\times L^2(0,1)$, we define $e: X \to Y$ by
\begin{equation*}
e(y,u):=\begin{pmatrix}y_t-\vartheta y_{xx}+yy_x-Bu-f\\
                       y(0)-y_0
\end{pmatrix}.
\end{equation*}
The mapping $e: X \to Y$ consists of a sum of continuous linear terms and a continuous bilinear term. Hence it can be shown that  it is infinitely Fr\'echet differentiable.  Moreover due to the unique solvability of \eqref{e99}, for every $u \in \mathcal{H}$ there exists a unique element $y = y(u)\in W(0,T)$ satisfying $e(y(u),u) =0$ and estimate \eqref{e108} holds. Therefore the control-to-state  $u \in \mathcal{H}  \mapsto y(u) \in W(0,T)$ is well-defined. Then we can rewrite the optimal control problem \eqref{e98}-\eqref{e99} in the following form
\begin{equation}
\label{Opburger}
 \min_{u \in \mathcal{H}}\mathcal{F}(u) = \min_{u \in \mathcal{H}}J(y(u),u) =\min_{(y,u)\in X } \{ J(y,u) : \text{ subject  to } e(y,u)=0 \}.
\end{equation}
Further, due to estimate \eqref{e108} and  the compact embedding from the space $W(0,T)$ to the space $L^2(Q)$, it follows from standard subsequential limit arguments that the optimal control problems \eqref{e98}-\eqref{e99} admits a solution, see e.g., \cite{MR1872392,MR1818917}. Before dealing with the  optimality conditions, we refer to the following  linearized Burgers equation at $y \in W(0,T)$ and its corresponding backward in time  adjoint equation
\begin{center}
\nopagebreak
\scalebox{0.8}{
\hspace{-3em}
\noindent\begin{minipage}{.6\columnwidth}
\begin{equation}
\label{e137}
\begin{cases}
q_t-\vartheta q_{xx} + (yq)_x= \phi  & (t,x)\in Q ,\\
q(t,0)=q(t,1)=0 & t \in (0,T),\\
q(0)=q_0  &x \in (0,1),\\
\end{cases}
\end{equation}
\end{minipage}%
\hspace{-0em}
\begin{minipage}{.64\columnwidth}
\begin{equation}
\label{e138}
\begin{cases}
-p_t-\vartheta p_{xx}-yp_x=\psi &(t,x)\in Q ,\\
      p(t,0)=p(t,1)=0 & t \in (0,T),\\
      p(T,x)=p_T   & x \in (0,1).\\
\end{cases}
\end{equation}
\end{minipage}}
\nopagebreak\vspace{\belowdisplayskip}
\end{center}
It can be shown that for every pairs $(\phi,q_0)$  and  $(\psi,p_T)$ in the space  $Y$, the solution operators $\mathcal{S}^y_{lin}: Y \to W(0,T)$ of \eqref{e137}, and $\mathcal{S}^y_{adj} : Y\to W(0,T)$  of \eqref{e138} defined by  $(\phi,q_0) \mapsto v$ and $(\psi,p_T) \mapsto p$, respectively, are well-defined and continuous. See e.g., \cite{MR1872392,MR1818917}.

Due to the definitions of $\mathcal{S}^{y(u)}_{lin}$ and $e_y(y(u),u)$, we can infer that  $e^{-1}_y(y(u),u) = \mathcal{S}^{y(u)}_{lin}$ and, as consequence, $e_y(y,u)$ is continuously invertible. In addition, since $e$ is infinitely continously Fr\'echet differentiable \cite{MR2516528}, the implicit function theorem implies that the control-to-state operator $u \mapsto y(u)$ is infinitely continuously Fr\'echet differentiable and its Fr\'echet derivatives of all orders are Lipschitz continuous on bounded sets.
Now we are in the position to derive the first-order optimality conditions.  First, by using the implicit function theorem, the first derivative of the mapping $u \mapsto y(u)$ at $u$ in direction of an arbitrary $\delta u \in \mathcal{H}$ is given by
\begin{equation}
\label{e139}
y'(u)\delta u = -e^{-1}_y(x)e_u(x)\delta u,
\end{equation}
where $x:=(y(u),u)\in X$. Then, by the chain rule we obtain
\begin{equation*}
\mathcal{F}'(u)\delta u  = (\mathcal{G}(u),\delta u) = ((y'(u))^*J_{y}(x)+J_{u}(x),\delta u),
\end{equation*}
where $(y'(u))^*$ stands for the adjoint operator of $y'(u)$. Since $\delta u$ is arbitrary, the first-order optimality condition $\eqref{e39}$ can be written as
\begin{equation}
\label{e140}
\mathcal{G}(u^*) = J_{u}(x^*)-e^{*}_u(x)e^{-*}_y(x)J_{y}(x^*)=0.
\end{equation}
where $x^*: = (y(u^*),u^*)$. Moreover,  by setting $(p^*,\bar{p} ):=-e^{-*}_y(x)J_{y}(x^*)$ with $(p^*,\bar{p} )\in Y^* $ and  $\bar{p}=p^*(0)$, the first-order optimality condition \eqref{e140} can be expressed as the following system of differential equations
\begin{equation*}
\begin{cases}
B^*p^*=\beta u^*    \quad \text{ in }  L^2(\hat{Q}), \\
p^* = \mathcal{S}_{adj}^{y(u^*)}(-\alpha_1(y(u^*)-y_d),-\alpha_2((y(u^*))(T)-z_d)). \\
\end{cases}
\end{equation*}
Next, we  compute the second derivative of $\mathcal{F}$. Let $(\delta u, \delta v) \in \mathcal{H} \times \mathcal{H}$ be arbitrary,  then using the implicit functions theorem, the second derivative of the operator $u \mapsto y(u)$ from $\mathcal{H}$ to $W(0,T)$ can be written as
\begin{equation}
\label{e145}
y''(u)(\delta u, \delta v)=-e^{-1}_y(x)e_{yy}(x)(y'(u)\delta u,y'(u) \delta v).
\end{equation}
 Now, by using the chain rule and \eqref{e145} as in \cite{MR1871460,MR2516528}, we obtain
 \begin{equation}
 \label{e146}
 \begin{split}
 \mathcal{F}''(u)&(\delta u,\delta v) = \langle J_{yy}(x)y'(u)\delta u, y'(u)\delta v\rangle \\
 &+\langle -e^{-*}_y(x)J_y(x),e_{yy}(x)(y'(u)\delta u,y'(u) \delta v)\rangle +\langle J_{uu}(x)\delta u, \delta v\rangle_{\mathcal{H}}.
 \end{split}
 \end{equation}
Furthermore,  due to the first estimate  in \eqref{e146} and the fact that $J: \mathcal{H}\times W(0,T) \to \mathbb{R}$ and the control-to-state operator are infinitely Fr\'echet differentiable, it follows, clearly, that $\mathcal{F}'': \mathcal{H} \to \mathcal{L}(\mathcal{H},\mathcal{L}(\mathcal{H},\mathbb{R}))$ is locally Lipschitz continuous. Then, the uniformly positiveness of $\mathcal{F}''(u^*)$ can be expressed as
 \begin{equation}
 \label{e149}
 \begin{split}
 \mathcal{F}''(u^*)(v, v):&= \alpha_1\| q^*\|^2_{L^2(Q)}+\alpha_2\| q^*(T)\|^2_{L^2(0,1)}+(p^*,2q^*q^*_x)_{L^2(Q)}+\beta \| v\|_{\mathcal{H}}^2 \\
                                     &\geq \delta_{\inf} \| v\|_{\mathcal{H}}^2           \quad \text{ for all } v \in \mathcal{H},
 \end{split}
 \end{equation}
where $\delta_{\inf}>0$, $p^* := \mathcal{S}_{adj}^{y(u^*)}(-\alpha_1(y(u^*)-y_d),-\alpha_2((y(u^*))(T)-z_d)$ and $q^*:=\mathcal{S}^{y(u^*)}_{lin}(v,0)$.
\begin{remark}
Clearly, the only term in \eqref{e149} that can spoil the uniformly positiveness of  $\mathcal{F}''(u^*)$ is the term involving $p^*$. This term originates from the nonlinear convection term in the state equation. Since
\begin{equation*}
\abs{(p^*,2q^*q^*_x)_{L^2(Q)}} \leq c\| p^*\|_{L^2(0,T;L^{\infty}(0,1))}\| q^*\|^2_{W(0,T)}
\end{equation*}
for a constant $c$, the uniformly positiveness of  $\mathcal{F}''(u^*)$ holds, provided that  $\| p^*\|_{L^2(0,T;L^{\infty}(0,1))}$  is small enough. Indeed, for  $p^* =0$, inequality \eqref{e149} holds for $\delta_{\inf}:=\beta$ . For instance, by setting $y_d = z_d = f = 0$, inequality \eqref{e149} holds for every initial function $y_0$  with sufficiently small $\|y_0\|_{L^2(0,1)}$.
\end{remark}

\section{Numerical Experiments}
In order to validate our theoretical findings in the previous sections,  we report numerical results corresponding to the optimal control problems introduced in the  previous section. We investigate the application of Algorithm \ref{BBa} with respect to  different strategies for selecting step-sizes and  different choices of the discretization parameter $h$,  the control cost parameter $\beta$, and the tolerance $\epsilon$ in the termination condition \eqref{terminCond}.   For Algorithm \ref{BBa}, we consider the cases:
\begin{description}
\item[BB1: ] $\alpha_k :=\alpha^{BB1}_k$ for every  $k\geq1$.
\item[BB2: ] $\alpha_k :=\alpha^{BB2}_k$ for every  $k\geq1$.
\item[ABB: ]$\alpha_k :=\begin{cases}\alpha^{BB1}_k  \text{ for odd } k\geq1,\\
\alpha^{BB2}_k  \text{ for even } k\geq1.\\
\end{cases}$
\end{description}
The last case, which is known as the alternating strategy, has already been introduced  by e.g., \cite{MR1901255,MR1937087}  in the context of finite-dimensional unconstrained optimization.  Moreover, \cite{MR2129700} reports numerical results for the case of finite-dimensional  bound-constrained optimization problems which show that projected ABB works somewhat better than projected BB1. According to \eqref{e150} the value  $\beta$ in all the optimal control problems of the previous section has a direct influence on the spectral condition number of $\mathcal{A}^{\mathcal{F}}_{u^*}$ corresponding to  $\mathcal{F}$.  To be more precise, as the value of $\beta$ increases, the value of $\kappa(\mathcal{A}^{\mathcal{F}}_{u^*})$ is getting smaller. Therefore, as its has been discussed in Remarks \ref{Remark5} and \ref{Remark6}, one expects a larger total number of iterations for a smaller value of $\beta$ and a fixed tolerance $\epsilon$.  Moreover, according to  Remark \ref{Remark4}, the number $\ell$ depends on the behaviour (monotonicity versus nonmonotonicity) of $\{\|\mathcal{G}_k\|\}_k$, and consequently  also on  $\kappa( \mathcal{A}^{\mathcal{F}}_{u^*})$. Hence, the smaller $\beta$ is chosen, the larger the value of $\ell$ is expected to be.  We report the total number of iteration of the optimization Algorithm for different levels of discretization, or equivalently, different values of mesh-sizes. Then, for every example and fixed tolerance $\epsilon$, $\ell$ is reported as the maximum of the pairwise differences of $k_h(\epsilon)$ for different choices of $h$.  We have chosen $u_0 = 0$ and $\alpha_0 = 1$ ($u_1:= -\mathcal{G}(0)$) as the initial iterates. All computations were done in the MATLAB platform.

\begin{example}[Dirichlet optimal control for the Poisson equation]
\label{exp1}
We consider the problem introduced in Subsection \ref{DirichletPoisson} which is posed on the domain $\Omega :=(0,1)^2$. For the  discretization a uniform mesh was generated by triangulation. Then over this mesh, the discretization was done by a conforming linear finite
element scheme using continuous piecewise linear basis functions as described in Subsubsection  \ref{subsubdis}. We set $ f(x)= 10\sin(\pi(x_1+x_2))$  and  $y_d(x) =(x^2_1+x^2_2)^{\frac{1}{3}}$ where $x := (x_1,x_2) \in \Omega$.  Table \ref{table1} shows the number of required iterations $k^*_h(\epsilon)$ for different step-size strategies, and different values of  $\beta$, $\epsilon$,  and the mesh-size $h$.  From Table \ref{table1}, it can be observed that:
\begin{enumerate}
\item For every fixed $h$, $\epsilon$, and choice of step-size, decreasing in the value of $\beta$ implies that the number of required iterations $k^*_h(\epsilon)$ becomes larger and, thus, the convergence is getting slower.  This is in accordance with the fact that there is a trade-off  between the magnitude  of $\beta$ and  the value of  $\kappa(\mathcal{A})$ where $\mathcal{A} = \mathcal{L}^*\mathcal{L}+\beta \mathcal{I}$ with  $\mathcal{L}$ specified in Subsubsection \ref{SubDirichletPoisson}.  More precisely,  $\kappa(\mathcal{A})= \frac{\beta +\delta_{\sup} }{\beta +\delta_{\inf}}$ with  $\delta_{\inf}:=\inf(\sigma(\mathcal{L}^*\mathcal{L}))$ and $\delta_{\sup} :=\sup(\sigma(\mathcal{L}^*\mathcal{L}))$. Hence a larger value of $\beta$ yields a smaller value of $\kappa({\mathcal{A}})$.  That is as expected from the theory, for a larger $\beta$  Algorithm \ref{BBa} requires fewer iterations $k^*_h(\epsilon)$ for every fixed $h$ and  $\epsilon$. This behaviour is clearly illustrated in Figure \ref{Fig1} which depicts  the convergence of  $\| \mathcal{G}^h_k \|_h$  for the choice of BB1 step-sizes, $h =2^{-9}\sqrt{2}$, and different values of $\beta$. As can be seen from Figure $\ref{Fig1}$, the convergence for the cases $\beta=0.5$ and $\beta=0.2$ is Q-linear. For these cases we might conjecture that $\kappa(\mathcal{A}) < 2$ with a smaller value of  convergence rate $\gamma_{\mathcal{A}}$ for $\beta=0.5$ compared to $\beta = 0.2$.  However, for the rest of the cases,  nonmonotonic behaviour occurs, which corresponds to $\kappa(\mathcal{A})\geq 2$.   Apparently, as $\beta$ decreases, the nonmonotonic behaviour in the sequences $\{\|\mathcal{G}_k\| \}_k$ and, consequently, in  $\{\|\mathcal{G}^h_k\|_h \}_k$  becomes stronger. As discussed in Remarks \ref{Remark5} and \ref{Remark6},  if $\kappa(\mathcal{A})$  becomes larger, then the changes in the decreasing components $|g^k_i|$ (for instance $i \in \{0\}\cup \{i: i\geq n^{u}\}$) are getting smaller compared to the nondecreasing components. This explains why a decrease in the value $\beta$ leads to an increase in nonmonotonicity.

\item  Mesh-independence can be observed from Table \ref{table1}. More precisely, we can see that for every fixed  $\beta$, $\epsilon$, and step-size strategy,
the iterations $k_h(\epsilon)$ stay almost constant and do not change as the discretezation levels changes. Moreover, for $\beta = 0.2$, $ \beta= 0.05$, and  $\beta = 0.01$ we can state that  $\ell \approx 1$, $\ell \approx 3$, and $\ell \approx 6$, respectively. This is also due to the dependence of the spectrum of $\mathcal{A}= \mathcal{L^*L}+\beta \mathcal{I}$ on the magnitude of $\beta$ (see Remark \ref{Remark4}).
\end{enumerate}

\begin{figure}[htbp]
    \centering
        \includegraphics[height=5cm,width=7cm]{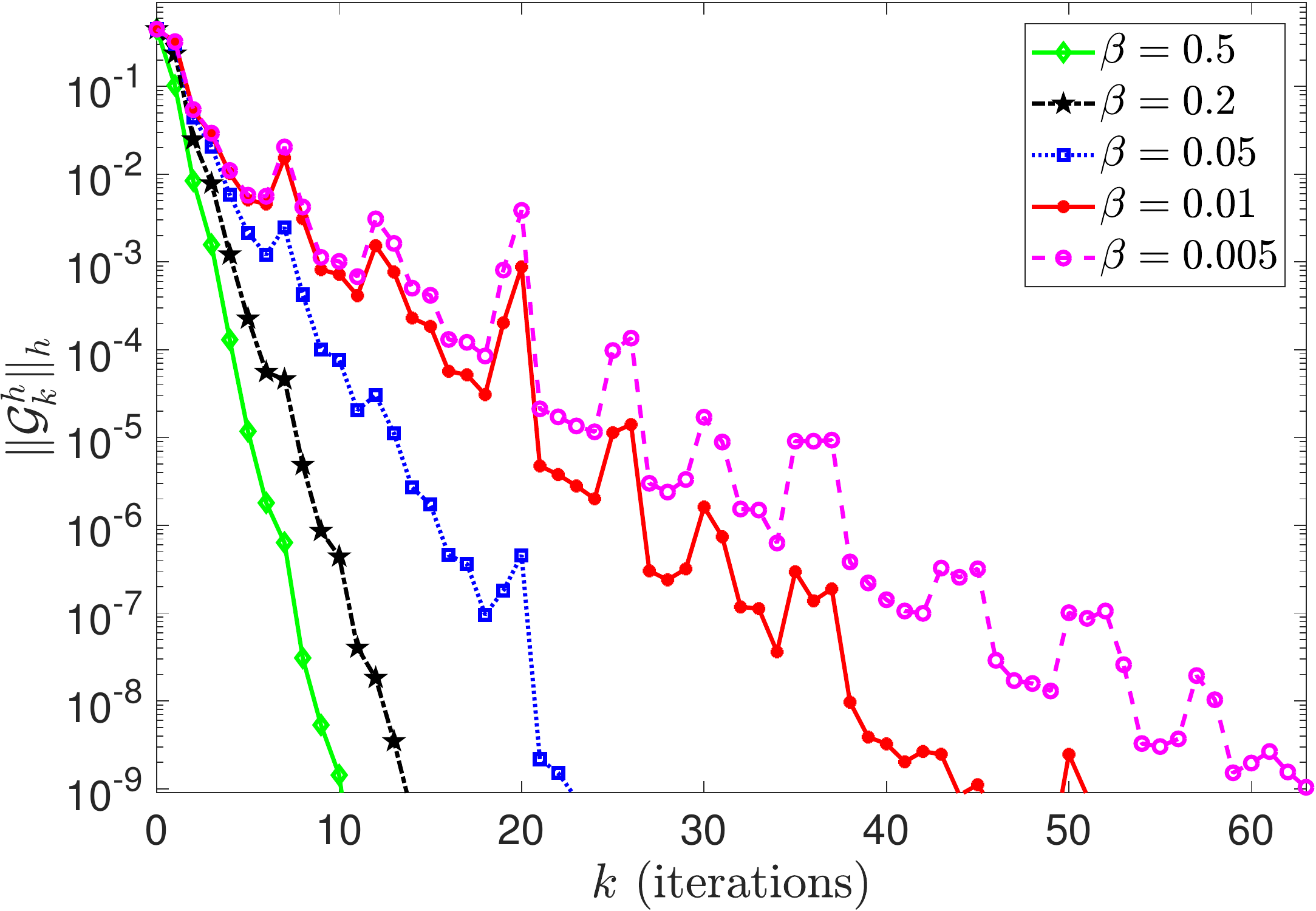}
\caption{Convergence of  $\| \mathcal{G}^h_k \|_h$  for Algorithm \ref{BBa} applied to Example \ref{exp1} with BB1 step-sizes, $h =2^{-9}\sqrt{2}$, and  for different choices of $\beta$}
 \label{Fig1}
\end{figure}
\begin{table}[htpb]
\begin{center}
\scalebox{0.7}{%
\begin{tabular}{ | c | c | c | c |  c |  c | c | c |}
\hline
\multicolumn{8}{|c|}{The number of required iteration $k^*_h(\epsilon)$} \\
\hline
\multicolumn{8}{|c|}{$\beta = 0.2$} \\
\hline
\multirow{6}{*}{BB1}&
\backslashbox{$\epsilon$}{$h$} & \footnotesize{$2^{-5}\sqrt{2}$} &\footnotesize{$2^{-6}\sqrt{2}$}&\footnotesize{$2^{-7}\sqrt{2}$}&\footnotesize{$2^{-8}\sqrt{2}$}&\footnotesize{$2^{-9}\sqrt{2}$}&\footnotesize{$2^{-10}\sqrt{2}$}\\
\cline{2-8}
&$1e-2$ & $3$ & $3$& $3$ & $3$ &$3$ & $3$ \\
\cline{2-8}
&$1e-4$ & $6$ & $6$& $6$ & $6$ &$6$&$6$ \\
\cline{2-8}
&$1e-6$ & $9$ & $9$& $9$ & $9$ &$9$&$9$ \\
\cline{2-8}
&$1e-8$ & $12$ & $13$& $13$ & $13$ &$13$ &$13$\\
\hline
\hline
\multirow{6}{*}{BB2}&
\backslashbox{$\epsilon$}{$h$} & \footnotesize{$2^{-5}\sqrt{2}$} &\footnotesize{$2^{-6}\sqrt{2}$}&\footnotesize{$2^{-7}\sqrt{2}$}&\footnotesize{$2^{-8}\sqrt{2}$}&\footnotesize{$2^{-9}\sqrt{2}$}&\footnotesize{$2^{-10}\sqrt{2}$}\\
\cline{2-8}
&$1e-2$ & $3$ & $3$& $3$ & $3$ &$3$  &$3$\\
\cline{2-8}
&$1e-4$ & $6$ & $6$& $6$ & $6$ &$6$ &$6$ \\
\cline{2-8}
&$1e-6$ & $9$ & $9$& $9$ & $9$ &$9$ &$9$ \\
\cline{2-8}
&$1e-8$ & $11$ & $12$& $12$ & $12$ &$12$ &$12$ \\
\hline
\hline
\multirow{6}{*}{ABB}&
\backslashbox{$\epsilon$}{$h$} & \footnotesize{$2^{-5}\sqrt{2}$} &\footnotesize{$2^{-6}\sqrt{2}$}&\footnotesize{$2^{-7}\sqrt{2}$}&\footnotesize{$2^{-8}\sqrt{2}$}&\footnotesize{$2^{-9}\sqrt{2}$}&\footnotesize{$2^{-10}\sqrt{2}$}\\
\cline{2-8}
&$1e-2$ & $3$ & $3$& $3$ & $3$ &$3$ &$3$ \\
\cline{2-8}
&$1e-4$ & $6$ & $6$& $6$ & $6$ &$6$ &$6$ \\
\cline{2-8}
&$1e-6$ & $9$ & $9$& $9$ & $9$ &$9$ & $9$ \\
\cline{2-8}
&$1e-8$ & $12$ & $12$& $13$ & $13$ &$13$ &$13$ \\
\hline
\hline
\multicolumn{8}{|c|}{$\beta = 0.05$} \\
\hline
\multirow{6}{*}{BB1}&
\backslashbox{$\epsilon$}{$h$} & \footnotesize{$2^{-5}\sqrt{2}$} &\footnotesize{$2^{-6}\sqrt{2}$}&\footnotesize{$2^{-7}\sqrt{2}$}&\footnotesize{$2^{-8}\sqrt{2}$}&\footnotesize{$2^{-9}\sqrt{2}$}&\footnotesize{$2^{-10}\sqrt{2}$}\\
\cline{2-8}
&$1e-2$ & $4$ & $4$& $4$ & $4$ &$4$  &$4$\\
\cline{2-8}
&$1e-4$ & $9$ & $9$& $9$ & $9$ &$10$ &$10$\\
\cline{2-8}
&$1e-6$ & $14$ & $16$& $16$ & $16$ &$16$ &$16$ \\
\cline{2-8}
&$1e-8$ & $21$ & $21$& $21$ & $21$ &$21$ &$21$ \\
\hline
\hline
\multirow{6}{*}{BB2}&
\backslashbox{$\epsilon$}{$h$} & \footnotesize{$2^{-5}\sqrt{2}$} &\footnotesize{$2^{-6}\sqrt{2}$}&\footnotesize{$2^{-7}\sqrt{2}$}&\footnotesize{$2^{-8}\sqrt{2}$}&\footnotesize{$2^{-9}\sqrt{2}$}&\footnotesize{$2^{-10}\sqrt{2}$}\\
\cline{2-8}
&$1e-2$ & $4$ & $4$& $4$ & $4$ &$4$&$4$ \\
\cline{2-8}
&$1e-4$ & $9$ & $9$& $9$ & $9$ &$9$&$9$ \\
\cline{2-8}
&$1e-6$ & $14$ & $15$& $15$ & $15$ &$15$&$15$ \\
\cline{2-8}
&$1e-8$ & $19$ & $21$& $21$ & $21$ &$22$ &$22$\\
\hline
\hline
\multirow{6}{*}{ABB}&
\backslashbox{$\epsilon$}{$h$} & \footnotesize{$2^{-5}\sqrt{2}$} &\footnotesize{$2^{-6}\sqrt{2}$}&\footnotesize{$2^{-7}\sqrt{2}$}&\footnotesize{$2^{-8}\sqrt{2}$}&\footnotesize{$2^{-9}\sqrt{2}$}&\footnotesize{$2^{-10}\sqrt{2}$}\\
\cline{2-8}
&$1e-2$ & $4$ & $4$& $4$ & $4$ &$4$&$4$ \\
\cline{2-8}
&$1e-4$ & $9$ & $9$& $9$ & $9$ &$9$ &$9$\\
\cline{2-8}
&$1e-6$ & $14$ & $14$& $15$ & $16$ &$16$&$16$ \\
\cline{2-8}
&$1e-8$ & $18$ & $21$& $21$ & $21$ &$21$&$21$ \\
\hline
\hline
\multicolumn{8}{|c|}{$\beta = 0.01$} \\
\hline
\multirow{6}{*}{BB1}&
\backslashbox{$\epsilon$}{$h$} & \footnotesize{$2^{-5}\sqrt{2}$} &\footnotesize{$2^{-6}\sqrt{2}$}&\footnotesize{$2^{-7}\sqrt{2}$}&\footnotesize{$2^{-8}\sqrt{2}$}&\footnotesize{$2^{-9}\sqrt{2}$}&\footnotesize{$2^{-10}\sqrt{2}$}\\
\cline{2-8}
&$1e-2$ & $4$ & $4$& $4$ & $5$ &$5$ &$5$ \\
\cline{2-8}
&$1e-4$ & $16$ & $16$& $16$ & $16$ &$16$ &$16$\\
\cline{2-8}
&$1e-6$ & $24$ & $28$& $27$ & $27$ &$27$ &$27$\\
\cline{2-8}
&$1e-8$ & $38$ & $39$& $38$ & $40$ &$38$&$39$ \\
\hline
\hline
\multirow{6}{*}{BB2}&
\backslashbox{$\epsilon$}{$h$} & \footnotesize{$2^{-5}\sqrt{2}$} &\footnotesize{$2^{-6}\sqrt{2}$}&\footnotesize{$2^{-7}\sqrt{2}$}&\footnotesize{$2^{-8}\sqrt{2}$}&\footnotesize{$2^{-9}\sqrt{2}$}&\footnotesize{$2^{-10}\sqrt{2}$}\\
\cline{2-8}
&$1e-2$ & $4$ & $4$& $5$ & $5$ &$5$ &$5$\\
\cline{2-8}
&$1e-4$ & $13$ & $15$& $15$ & $15$ &$15$&$15$ \\
\cline{2-8}
&$1e-6$ & $26$ & $26$& $30$ & $31$& $31$ &$32$ \\
\cline{2-8}
&$1e-8$ & $39$ & $44$& $43$ & $45$ &$45$  &$44$ \\
\hline
\hline
\multirow{6}{*}{ABB}&
\backslashbox{$\epsilon$}{$h$} & \footnotesize{$2^{-5}\sqrt{2}$} &\footnotesize{$2^{-6}\sqrt{2}$}&\footnotesize{$2^{-7}\sqrt{2}$}&\footnotesize{$2^{-8}\sqrt{2}$}&\footnotesize{$2^{-9}\sqrt{2}$}&\footnotesize{$2^{-10}\sqrt{2}$}\\
\cline{2-8}
&$1e-2$ & $4$ & $4$& $5$ & $5$ &$5$ &$5$ \\
\cline{2-8}
&$1e-4$ & $15$ & $15$  & $16$& $16$ & $16$ &$16$ \\
\cline{2-8}
&$1e-6$ & $24$ &$24$ & $28$& $28$ & $28$ &$28$ \\
\cline{2-8}
&$1e-8$ & $43$ & $38$& $43$ & $43$ &$43$ &$40$ \\
\hline
\end{tabular}}
\end{center}
\caption{Numerical results for Example \ref{exp1}}
 \label{table1}
\end{table}
To further study the behaviour of Algorithms \ref{BBa}, we consider Table \ref{table1b} which summarizes the values of  $\| \mathcal{G}^h_k \|_h $   for the choice of $\beta = 0.01$,  at the iterations $k = 37,\dots,45$,  and  different levels of discretization. It can be seen that in any case the sequence $\{\| \mathcal{G}^h_k \|_h \}_k $ has a nonmonotonic behaviour. For every case the members of $\{\| \mathcal{G}^h_k \|_h \}_k $ at which the monotonicity of the sequence is violated, are indicated by bold type. With the superscript star we denote the members corresponding to $k^*_h(\epsilon)$ with $\epsilon = 1e-8$.
\begin{table}[htpb]
\begin{center}
\scalebox{0.7}{%
\begin{tabular}{ | c | c |  c | c |  c |  c | c | c |}
\hline
\multicolumn{8}{|c|}{The value of $\| \mathcal{G}^h_k\|$ at an iteration $k$} \\
\hline
\multirow{10}{*}{BB1}&
\backslashbox{$k$}{$h$} & \footnotesize{$2^{-5}\sqrt{2}$} &\footnotesize{$2^{-6}\sqrt{2}$}&\footnotesize{$2^{-7}\sqrt{2}$}&\footnotesize{$2^{-8}\sqrt{2}$}&\footnotesize{$2^{-9}\sqrt{2}$}&\footnotesize{$2^{-10}\sqrt{2}$}\\
\cline{2-8}
&$37$ & $1.29e-8$              & $1.59e-8$& $3.96e-8$ & $5.25e-7$ &$1.90e-7$ & $1.98e-7$ \\
&$38$ & ${9.22e-9}^*$             & $1.05e-8$& ${4.26e-9}^*$ & $1.03e-7$ &${9.69e-9}^*$&$1.39e-8$ \\
&$39$ & $\mathbf{ 2.93e-8}$ & ${5.17e-9}^*$& $3.78e-9$ & $5.93e-8$ &$3.88e-9$&${8.75e-9}^*$ \\
&$40$ & $\mathbf{1.05e-7}$ & $1.38e-9$& $2.38e-9$ & ${7.21e-9}^*$ &$3.25e-9$ &$3.46e-9$\\
&$41$ & $1.43e-8$           & $1.08e-9$& $\mathbf{6.90e-9}$ & $4.30e-9$ &$2.02e-9$ & $1.90e-9$ \\
&$42$ & $\mathbf{1.46e-8}$ & $\mathbf{1.69e-9}$& $4.72e-9$ & $\mathbf{5.76e-9}$ &$2.66e-9$ & $1.28e-9$ \\
&$43$ & $7.26e-10$         & $\mathbf{1.26e-8}$& $3.74e-9$ & $1.12e-9$ &$2.47e-9$&$6.74e-9$ \\
&$44$ & $6.24e-10$         & $1.55e-9$& $2.63e-10$ &  $8.49e-10$ &$8.46e-10$ &$\mathbf{1.84e-8}$\\
&$45$ & $4.09e-10$         & $1.43e-9$& $1.74e-10$ & $7.33e-10$ &$\mathbf{1.11e-9}$ &$4.10e-10$\\
\hline
\hline
\multirow{10}{*}{BB2}&
\backslashbox{$k$}{$h$} & \footnotesize{$2^{-5}\sqrt{2}$} &\footnotesize{$2^{-6}\sqrt{2}$}&\footnotesize{$2^{-7}\sqrt{2}$}&\footnotesize{$2^{-8}\sqrt{2}$}&\footnotesize{$2^{-9}\sqrt{2}$}&\footnotesize{$2^{-10}\sqrt{2}$}\\
\cline{2-8}
&$37$ & $5.94e-8$ & $2.91e-8$& $7.59e-8$ & $3.38e-7$ &$7.83e-8$ & $2.40e-7$ \\
&$38$ & $2.52e-8$ & $2.71e-8$& $4.59e-8$ & $2.19e-7$ &$6.21e-8$&$1.95e-7$ \\
&$39$ & ${4.97e-9}^*$ & $1.78e-8$& $1.28e-8$ & $2.08e-7$ &$5.32e-8$&$3.46e-8$ \\
&$40$ & $4.06e-9$ & $\mathbf{2.87e-8}$& $\mathbf{1.38e-7}$ & $1.71e-7$ &$4.72e-8$ &$\mathbf{2.61e-7}$\\
&$41$ & $3.52e-9$ & $\mathbf{6.00e-8}$& $4.12e-8$ & $6.06e-8$ &$1.77e-8$ & $4.37e-8$ \\
&$42$ & $\mathbf{4.08e-9}$ & $2.13e-8$& $2.63e-8$ & $\mathbf{1.13e-6}$ &$\mathbf{2.12e-7}$ & $1.40e-8$ \\
&$43$ & $3.36e-9$ & $1.23e-8$& ${1.02e-9}^*$ & $9.80e-8$ &$6.14e-8$&$1.00e-8$ \\
&$44$ &$2.38e-9$ & ${1.88e-9}^*$& $7.07e-10$ & $4.43e-8$ &$1.64e-8$&${2.19e-9}^*$ \\
&$45$ & $2.67e-9$ & $1.78e-9$& $4.10e-10$ & ${9.39e-9}^*$ &${7.97e-9}^*$ &$2.00e-9$\\
\hline
\hline
\multirow{10}{*}{ABB}&
\backslashbox{$k$}{$h$} & \footnotesize{$2^{-5}\sqrt{2}$} &\footnotesize{$2^{-6}\sqrt{2}$}&\footnotesize{$2^{-7}\sqrt{2}$}&\footnotesize{$2^{-8}\sqrt{2}$}&\footnotesize{$2^{-9}\sqrt{2}$}&\footnotesize{$2^{-10}\sqrt{2}$}\\
\cline{2-8}
&$37$ & $5.10e-8$ & $1.57e-8$& $2.93e-8$ & $3.44e-8$ &$3.83e-8$ & $2.55e-8$ \\
&$38$ & $2.70e-8$ & ${7.44e-9}^*$& $1.96e-8$ & $2.49e-8$ &$2.94e-8$&$1.55e-8$ \\
&$39$ & $2.37e-8$ & $5.79e-9$& $\mathbf{3.70e-8}$ & $\mathbf{4.67e-8}$ &$2.20e-8$&$1.15e-8$ \\
&$40$ & $\mathbf{2.50e-8}$ & $4.61e-9$& $\mathbf{7.85e-8}$ & $\mathbf{1.50e-7}$ &$\mathbf{9.20e-8}$ &${2.21e-9}^*$\\
&$41$ & $\mathbf{1.50e-7}$ & $\mathbf{1.25e-8}$& $\mathbf{8.04e-8}$ & $9.71e-8$ &$\mathbf{1.81e-7}$ & $\mathbf{5.17e-8}$ \\
&$42$ & $1.31e-8$ & $\mathbf{1.39e-8}$& $1.06e-8$ & $1.08e-8$ &$1.17e-8$ & $3.26e-8$ \\
&$43$ & ${9.26e-9}^*$ & $8.82e-9$& ${2.66e-9}^*$ & ${6.73e-9}^*$ &${6.90e-9}^*$&$1.36e-8$ \\
&$44$ &$8.76e-9$ & $1.06e-9$& $1.75e-9$ & $1.90e-9$ &$6.11e-9$&$2.33e-9$ \\
&$45$ & $4.43e-9$ & $9.25e-10$& $7.39e-10$ & $1.29e-9$ &$5.62e-9$ &$1.90e-9$\\
\hline
\end{tabular}}
\end{center}
\caption{The values of  $\| \mathcal{G}^h_k \|_h $ related to  Example \ref{exp1}  for  $\beta = 0.01$ and  iterations $k = 37,\dots,45$}
 \label{table1b}
\end{table}
%
%
\end{example}
\begin{example}[Neumann optimal control the for the  linear wave equation]
\label{exp2}
In this example, we deal with the optimal control problem \eqref{e85}-\eqref{e84}. Here, the spatial discretization has been done similarly to the previous example on the domain $\Omega: = (0,1)^2$ with the mesh-size $h$. Further,  for the temporal discretization of the state equation we used  a Petrov-Galerkin scheme based on continuous piecewise linear basis functions for the trial space and piecewise constant test functions. By doing so, the resulting discretized system is equivalent to the system first discretized in space followed by the Crank-Nicolson time stepping method with a step-size $\Delta t$.  Since the temporal test functions have been chosen to be piecewise constant, it is natural to also discretize the adjoint equation and also control by these functions.  This implies that the approximated  gradient is consistent with both continuous functional and the discrete functional.  Here we set $T =1$, $y^1_0(x) =\sin(\pi x_1)\sin(\pi x_2)$,  $y^2_0(x) =0$,  $f(t,x) = \pi^2\sin(\pi x_1t)\sin(\pi x_2t)$, $z_d(x) = 0$, and
\begin{equation*}
y_d(x,t) = \begin{cases} -x_1  & x_1<0.5 ,\\
                              x_1  & x_1\geq 0.5,
\end{cases}
\end{equation*}
where $x := (x_1,x_2) \in \Omega$. The Neumann control is applied on the subset $\Gamma_c \subset \partial \Omega$ given by $\{ (1,x_2) : x_2 \in (0,1)\} \cup \{ (x_1,1) :  x_1 \in (0,1)\}$.
In Figure \ref{Fig2}, we report the behaviour of the gradient norm for Example \ref{exp2} for the choice of BB2 step-sizes, $(\Delta t,h)=(0.0064,2^{-7}\sqrt{2})$, and  for different values of $\beta$. To illustrate the mesh-independence, we reported the values of $k_{\Delta t,h}^*(\epsilon)$ for different levels of temporal and spatial discretization. As it is reported in Table \ref{table2}, we decreased the mesh-size $h$ and step-size $\Delta t$ simultaneously. Clearly, similar observations as in the previous example are also valid  for this example, with the difference that here for $ \beta= 0.5$, and  $\beta = 0.05$, we have $\ell \approx 0$  and $\ell \approx 6$, respectively.

\begin{figure}[htbp]
    \centering
        \includegraphics[height=5cm,width=7cm]{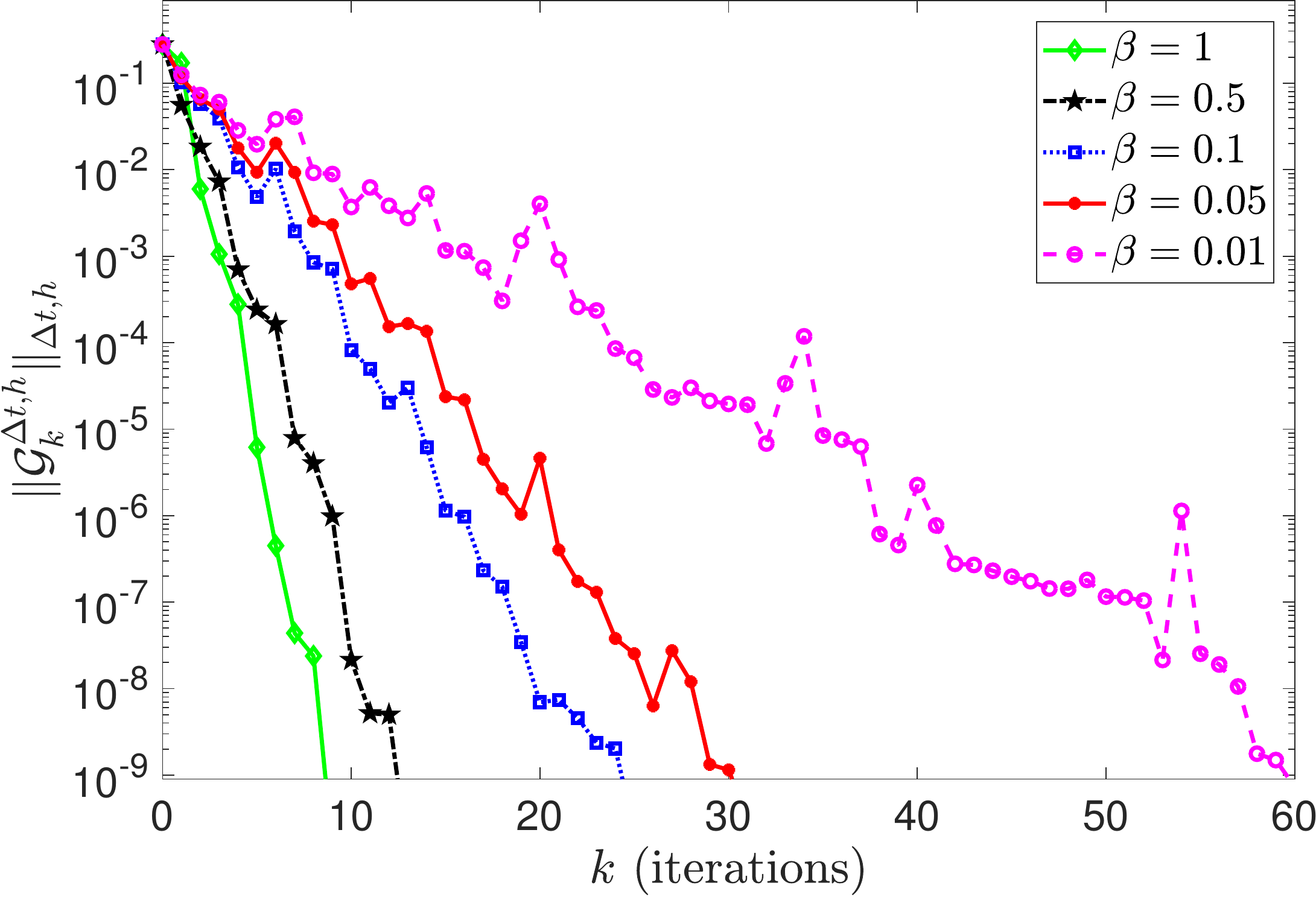}
\caption{Convergence of  $\| \mathcal{G}^{\Delta t,h}_k \|_{\Delta t,h}$  for Algorithm \ref{BBa} applied to Example \ref{exp2} with BB2  step-sizes, $(\Delta t,h)=(0.0064,2^{-7}\sqrt{2})$, and  for different choices of $\beta$}
 \label{Fig2}
\end{figure}
\begin{table}
\begin{center}
\scalebox{0.7}{%
\begin{tabular}{ | c | c | c | c |  c |  c | c |}
\hline
\multicolumn{7}{|c|}{The number of required iteration $k^*_{\Delta t,h}(\epsilon)$} \\
\hline
\multicolumn{7}{|c|}{$\beta = 0.5$} \\
\hline
\multirow{6}{*}{BB1}&\backslashbox{$\epsilon$}{$(\Delta t,h)$}& \footnotesize{$(0.01,2^{-4}\sqrt{2})$} &\footnotesize{$(0.04,2^{-5}\sqrt{2})$}&\footnotesize{$(0.016,2^{-6}\sqrt{2})$}&\footnotesize{$(0.0064,2^{-7}\sqrt{2})$}&\footnotesize{$(0.0026,2^{-8}\sqrt{2})$}\\
\cline{2-7}
&$1e-2$ & $3$ & $3$& $3$ & $3$ &$3$ \\
\cline{2-7}
&$1e-4$ & $7$ & $7$& $7$ & $7$ &$7$ \\
\cline{2-7}
&$1e-6$ & $9$ & $9$& $9$ & $9$ &$9$ \\
\cline{2-7}
&$1e-8$ & $11$ & $11$& $11$ & $11$ &$11$ \\
\hline
\multirow{6}{*}{BB2}&\backslashbox{$\epsilon$}{$(\Delta t,h)$}& \footnotesize{$(0.01,2^{-4}\sqrt{2})$} &\footnotesize{$(0.04,2^{-5}\sqrt{2})$}&\footnotesize{$(0.016,2^{-6}\sqrt{2})$}&\footnotesize{$(0.0064,2^{-7}\sqrt{2})$}&\footnotesize{$(0.0026,2^{-8}\sqrt{2})$}\\
\cline{2-7}
&$1e-2$ & $3$ & $3$& $3$ & $3$ &$3$ \\
\cline{2-7}
&$1e-4$ & $7$ & $7$& $7$ & $7$ &$7$ \\
\cline{2-7}
&$1e-6$ & $9$ & $9$& $9$ & $9$ &$9$ \\
\cline{2-7}
&$1e-8$ & $11$ & $11$& $11$ & $11$ &$11$ \\
\hline
\multirow{6}{*}{ABB}&\backslashbox{$\epsilon$}{$(\Delta t,h)$}& \footnotesize{$(0.01,2^{-4}\sqrt{2})$} &\footnotesize{$(0.04,2^{-5}\sqrt{2})$}&\footnotesize{$(0.016,2^{-6}\sqrt{2})$}&\footnotesize{$(0.0064,2^{-7}\sqrt{2})$}&\footnotesize{$(0.0026,2^{-8}\sqrt{2})$}\\
\cline{2-7}
&$1e-2$ & $3$ & $3$& $3$ & $3$ &$3$ \\
\cline{2-7}
&$1e-4$ & $7$ & $7$& $7$ & $7$ &$7$ \\
\cline{2-7}
&$1e-6$ & $9$ & $9$& $9$ & $9$ &$9$ \\
\cline{2-7}
&$1e-8$ & $11$ & $11$& $11$ & $11$ &$11$ \\
\hline
\hline
\multicolumn{7}{|c|}{$\beta = 0.05$} \\
\hline
\multirow{6}{*}{BB1}&\backslashbox{$\epsilon$}{$(\Delta t,h)$}& \footnotesize{$(0.01,2^{-4}\sqrt{2})$} &\footnotesize{$(0.04,2^{-5}\sqrt{2})$}&\footnotesize{$(0.016,2^{-6}\sqrt{2})$}&\footnotesize{$(0.0064,2^{-7}\sqrt{2})$}&\footnotesize{$(0.0026,2^{-8}\sqrt{2})$}\\
\cline{2-7}
&$1e-2$ & $7$ & $7$& $7$ & $7$ &$7$ \\
\cline{2-7}
&$1e-4$ & $14$ & $14$& $14$ & $14$ &$14$ \\
\cline{2-7}
&$1e-6$ & $21$ & $21$& $21$ & $24$ &$24$ \\
\cline{2-7}
&$1e-8$ & $28$ & $29$& $28$ & $29$ &$29$ \\
\hline
\multirow{6}{*}{BB2}&\backslashbox{$\epsilon$}{$(\Delta t,h)$}& \footnotesize{$(0.01,2^{-4}\sqrt{2})$} &\footnotesize{$(0.04,2^{-5}\sqrt{2})$}&\footnotesize{$(0.016,2^{-6}\sqrt{2})$}&\footnotesize{$(0.0064,2^{-7}\sqrt{2})$}&\footnotesize{$(0.0026,2^{-8}\sqrt{2})$}\\
\cline{2-7}
&$1e-2$ & $5$ & $5$& $5$ & $5$ &$5$ \\
\cline{2-7}
&$1e-4$ & $12$ & $15$& $15$ & $15$ &$15$ \\
\cline{2-7}
&$1e-6$ & $19$ & $21$& $21$ & $21$ &$21$ \\
\cline{2-7}
&$1e-8$ & $24$ & $25$& $25$ & $26$ &$26$ \\
\hline
\multirow{6}{*}{ABB}&\backslashbox{$\epsilon$}{$(\Delta t,h)$}& \footnotesize{$(0.01,2^{-4}\sqrt{2})$} &\footnotesize{$(0.04,2^{-5}\sqrt{2})$}&\footnotesize{$(0.016,2^{-6}\sqrt{2})$}&\footnotesize{$(0.0064,2^{-7}\sqrt{2})$}&\footnotesize{$(0.0026,2^{-8}\sqrt{2})$}\\
\cline{2-7}
&$1e-2$ & $5$ & $7$& $7$ & $7$ &$7$ \\
\cline{2-7}
&$1e-4$ & $14$ & $14$& $14$ & $14$ &$13$ \\
\cline{2-7}
&$1e-6$ & $21$ & $22$& $24$ & $23$ &$23$ \\
\cline{2-7}
&$1e-8$ & $28$ & $31$& $34$ & $29$ &$28$ \\
\hline
\end{tabular}}
\end{center}
\caption{Numerical results for Example \ref{exp2}}
 \label{table2}
\end{table}
\end{example}
\begin{example}[Distributed optimal control for the Burgers equation] \label{exp3}
 We consider the optimal control problem \eqref{e98}-\eqref{e99} posed on the interval $(0,1)$. The spatial discretization was done by the standard Galerkin method based on piecewise linear basis functions with  mesh-size $h$. For temporal discretization, we used the implicit Euler method with a step-size denoted by $\Delta t$. Moreover, the resulting nonlinear systems after the temporal discretization were solved by  Newton's method with the tolerance $\epsilon_n =10^{-13}$. Here the control acts on the  open interval $\hat{\Omega} = (0.1,0.4)$. Moreover we set $\vartheta = 0.01$, $y_0(x) = 5\exp(-20(x-0.5)^2)$,  and  $ y_d(t,x)=z_d(x)= f(t,x) = 0$. Similarly to the previous example, we compute the values of $k_{\Delta t,h}^*(\epsilon)$ for different levels of temporal and spatial discretization. These results are gathered in Table \ref{table3}.  Further, Figure \ref{Fig3} shows the convergence of Algorithm \ref{BBa} applied to Example \ref{exp3}, for the choice of ABB  step-sizes, $(\Delta t,h)=(2^{-7},2^{-8})$, and  different values  of $\beta$. As can be seen from Table \ref{table3}  and Figures \ref{Fig3}, despite the nonlinearity the observations 1 and 2 from Example \ref{exp1} hold also true for this example.

\begin{figure}[htbp]
    \centering
        \includegraphics[height=5cm,width=7cm]{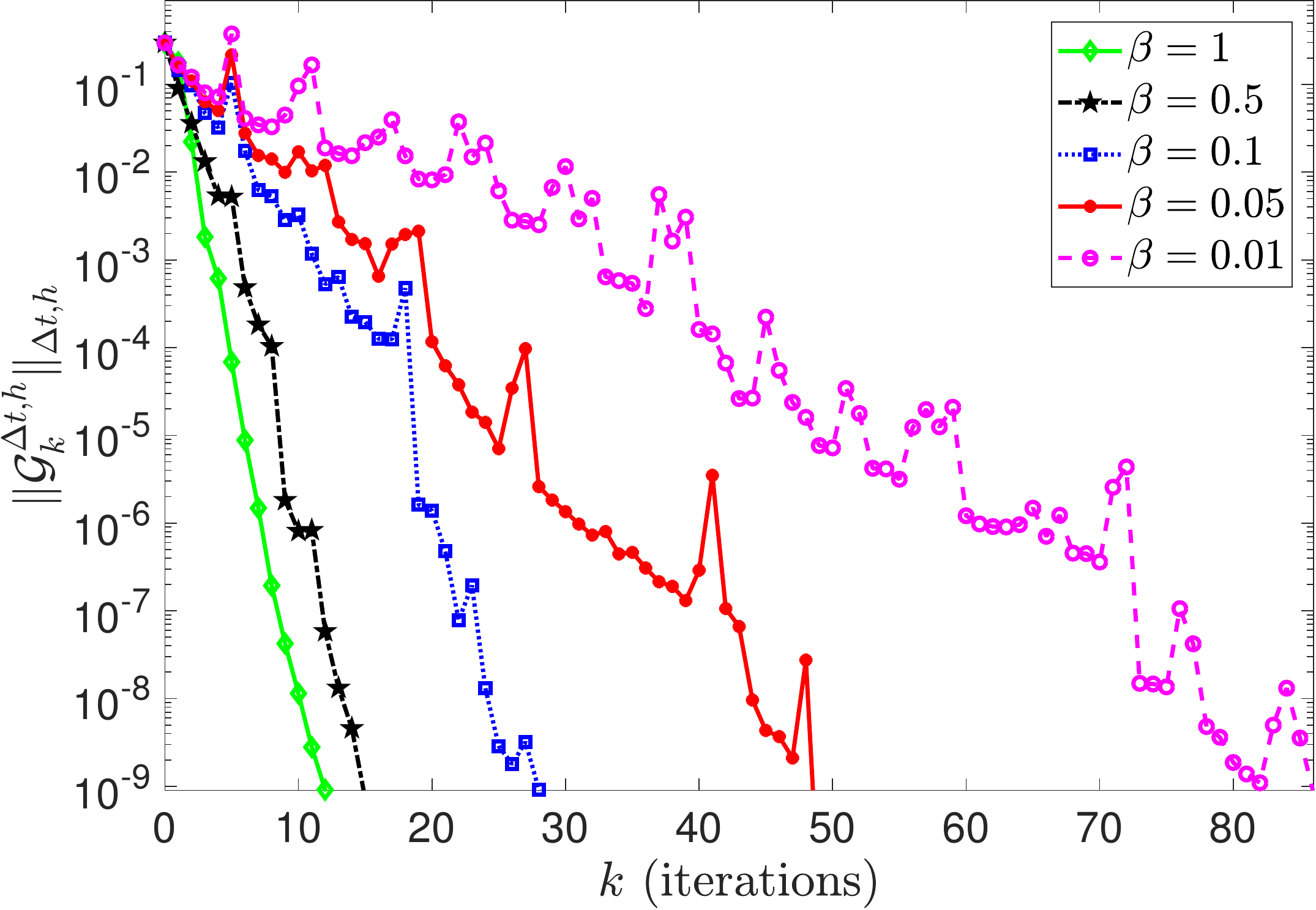}
\caption{Convergence of  $\| \mathcal{G}^{\Delta t,h}_k \|_{\Delta t,h}$  for Algorithm \ref{BBa} applied to Example \ref{exp3} with ABB  step-sizes, $(\Delta t,h)=(2^{-7},2^{-8})$, and  for different choices of $\beta$}
 \label{Fig3}
\end{figure}

\begin{table}[htpb]
\begin{center}
\scalebox{0.7}{%
\begin{tabular}{ | c | c | c | c |  c |  c | c |}
\hline
\multicolumn{7}{|c|}{The number of required iteration $k^*_{\Delta t,h}(\epsilon)$} \\
\hline
\multicolumn{7}{|c|}{$\beta = 0.5$} \\
\hline
\multirow{6}{*}{BB1}&\backslashbox{$\epsilon$}{$(\Delta t,h)$}& \footnotesize{$(2^{-4},2^{-5})$} &\footnotesize{$(2^{-5},2^{-6})$}&\footnotesize{$(2^{-6},2^{-7})$}&\footnotesize{$(2^{-7},2^{-8})$}&\footnotesize{$(2^{-8},2^{-9})$}\\
\cline{2-7}
&$1e-2$ & $4$ & $4$& $4$ & $4$ &$4$ \\
\cline{2-7}
&$1e-4$ & $8$ & $8$& $9$ & $9$ &$9$ \\
\cline{2-7}
&$1e-6$ & $11$ & $11$& $11$ & $11$ &$11$ \\
\cline{2-7}
&$1e-8$ & $13$ & $14$& $14$ & $14$ &$14$ \\
\hline
\multirow{6}{*}{BB2}&\backslashbox{$\epsilon$}{$(\Delta t,h)$}& \footnotesize{$(2^{-4},2^{-5})$} &\footnotesize{$(2^{-5},2^{-6})$}&\footnotesize{$(2^{-6},2^{-7})$}&\footnotesize{$(2^{-7},2^{-8})$}&\footnotesize{$(2^{-8},2^{-9})$}\\
\cline{2-7}
&$1e-2$ & $4$ & $4$& $4$ & $4$ &$4$ \\
\cline{2-7}
&$1e-4$ & $8$ & $8$& $8$ & $8$ &$8$ \\
\cline{2-7}
&$1e-6$ & $10$ & $11$& $11$ & $11$ &$11$ \\
\cline{2-7}
&$1e-8$ & $14$ & $14$& $14$ & $14$ &$14$ \\
\hline
\multirow{6}{*}{ABB}&\backslashbox{$\epsilon$}{$(\Delta t,h)$}& \footnotesize{$(2^{-4},2^{-5})$} &\footnotesize{$(2^{-5},2^{-6})$}&\footnotesize{$(2^{-6},2^{-7})$}&\footnotesize{$(2^{-7},2^{-8})$}&\footnotesize{$(2^{-8},2^{-9})$}\\
\cline{2-7}
&$1e-2$ & $4$ & $4$& $4$ & $4$ &$4$ \\
\cline{2-7}
&$1e-4$ & $8$ & $8$& $9$ & $9$ &$9$ \\
\cline{2-7}
&$1e-6$ & $11$ & $12$& $10$ & $10$ &$10$ \\
\cline{2-7}
&$1e-8$ & $13$ & $13$& $14$ & $14$ &$14$ \\
\hline
\hline
\multicolumn{7}{|c|}{$\beta = 0.05$} \\
\hline
\multirow{6}{*}{BB1}&\backslashbox{$\epsilon$}{$(\Delta t,h)$}& \footnotesize{$(2^{-4},2^{-5})$} &\footnotesize{$(2^{-5},2^{-6})$}&\footnotesize{$(2^{-6},2^{-7})$}&\footnotesize{$(2^{-7},2^{-8})$}&\footnotesize{$(2^{-8},2^{-9})$}\\
\cline{2-7}
&$1e-2$ & $12$ & $12$& $12$ & $12$ &$12$ \\
\cline{2-7}
&$1e-4$ & $25$ & $24$& $23$ & $25$ &$24$ \\
\cline{2-7}
&$1e-6$ & $36$ & $32$& $32$ & $33$ &$32$ \\
\cline{2-7}
&$1e-8$ & $43$ & $40$& $44$ & $42$ &$39$ \\
\hline
\hline
\multirow{6}{*}{BB2}&\backslashbox{$\epsilon$}{$(\Delta t,h)$}& \footnotesize{$(2^{-4},2^{-5})$} &\footnotesize{$(2^{-5},2^{-6})$}&\footnotesize{$(2^{-6},2^{-7})$}&\footnotesize{$(2^{-7},2^{-8})$}&\footnotesize{$(2^{-8},2^{-9})$}\\
\cline{2-7}
&$1e-2$ & $9$ & $9$& $9$ & $9$ &$9$ \\
\cline{2-7}
&$1e-4$ & $23$ & $21$& $24$ & $23$ &$23$ \\
\cline{2-7}
&$1e-6$ & $31$ & $29$& $33$ & $34$ &$35$ \\
\cline{2-7}
&$1e-8$ & $35$ & $41$& $41$ & $38$ &$40$ \\
\hline
\multirow{6}{*}{ABB}&\backslashbox{$\epsilon$}{$(\Delta t,h)$}& \footnotesize{$(2^{-4},2^{-5})$} &\footnotesize{$(2^{-5},2^{-6})$}&\footnotesize{$(2^{-6},2^{-7})$}&\footnotesize{$(2^{-7},2^{-8})$}&\footnotesize{$(2^{-8},2^{-9})$}\\
\cline{2-7}
&$1e-2$ & $13$ & $13$& $13$ & $9$ &$9$ \\
\cline{2-7}
&$1e-4$ & $24$ & $26$& $21$ & $21$ &$21$ \\
\cline{2-7}
&$1e-6$ & $30$ & $32$& $29$ & $31$ &$35$ \\
\cline{2-7}
&$1e-8$ & $36$ & $38$& $42$ & $44$ &$40$ \\
\hline
\end{tabular}}
\end{center}
\caption{Numerical results for Example \ref{exp3}}
\label{table3}
\end{table}
\end{example}

\appendix
\section{Appendix}
\subsection{Proof of Proposition \ref{lem5}}
\label{Ap1}
For every  $k\geq 1$, we consider the sequence  $\{ \hat{\alpha}^k_j \}_j$ associated to $\{ \hat{u}^k_j\}_j$, which is defined by
\begin{equation}
\label{e124}
\hat{\alpha}^k_j:=
\begin{cases}
 \hat{\alpha}^{BB1k}_j &  \text{ if } \alpha_{k+j}=\alpha^{BB1}_{k+j}, \\
 \hat{\alpha}^{BB2k}_j &  \text{ if } \alpha_{k+j}=\alpha^{BB2}_{k+j}. \\
\end{cases}
\end{equation}
for all $j\geq 1$. We will show by induction  that for every $q \in \{ 0,\dots , m \}$,  there exist positive constants $\lambda_{q}$ and $\eta_{q}$ such that
\begin{equation}
\tag{$P_q$}
\label{P}
\begin{cases}
&\text{If $u_k \in \mathcal{B}_{\eta_{q}}(u^*)$, $\alpha_k \in [\alpha_{\inf},\alpha_{\sup}]$, and if for some $\ell \in \{0, \dots, q\}$, the property \eqref{e67} holds,}\\
&\text{then  we have }  u_{k+j} \in \mathcal{B}_{\tau}(u^*), \text{ and } \quad \|u_{k+j}-\hat{u}^k_j\| \leq \lambda_q \|u_k-u^*\|^2 \text{ for all } j \in \{ 0,\dots,\ell\}. \\
\end{cases}
\end{equation}

For the case that $q =\ell =0$ and the choice of  $\eta_0 = \tau$ and arbitrary $\lambda_0>0$, property \eqref{P}  holds clearly since $\hat{u}^k_0 = u_k$.

 For the case that $q= 1$, by using \eqref{e80} and  \eqref{e60}, we have
\begin{equation*}
\|u_{k+1}-u^*\| \leq \|u_k-u^*\| + \| \mathcal{S}_{k} \| \leq  \|u_k-u^*\| + \frac{1}{|\alpha_k|} \| \mathcal{G}_k\| \leq (1+\frac{\alpha_{\sup}}{\alpha_{\inf}}) \| u_k-u^*\|,
\end{equation*}
where $\mathcal{S}_{k}:=u_{k+1}-u_{k}$. Hence, for  $\eta_1 := \frac{\tau}{1+\frac{\alpha_{\sup}}{\alpha_{\inf}}} $, we obtain $u_{k+1} \in \mathcal{B}_{\tau}(u^*)$. In the case $q=1$  we have either $\ell = 0$ or $\ell=1$.  For $\ell=0$, \eqref{e62} holds trivially. Therefore, we need to investigate $\eqref{e62}$ for $\ell= 1$.  By \ref{e45} and using the facts that $u_k = \hat{u}^k_{0} \in \mathcal{B}_{\tau}(u^*)$ and  $\alpha_k = \hat{\alpha}^k_{0}$, we can infer that
\begin{equation*}
 \begin{split}
 \|u_{k+1}-\hat{u}^k_{1}\|  &\leq \| u_{k}-\frac{1}{\alpha_{k} } \mathcal{G}(u_{k})-(\hat{u}^k_{0}-\frac{1}{\hat{\alpha}^k_{0}}\hat{\mathcal{G}}(\hat{u}^k_{0}))\| \\
 & \leq \frac{1}{|\alpha_k|}  \| \mathcal{G}(u_k)-\hat{\mathcal{G}}(\hat{u}^k_{0}) \|  \\
 &\leq  \frac{1}{\alpha_{\inf}} \| \mathcal{G}(u_{k})-  \mathcal{A}^{\mathcal{F}}_{u^*}(u_{k}-u^*) \| \leq \frac{L}{\alpha_{\inf}}\| u_{k}-u^*\|^2,
 \end{split}
\end{equation*}
where  $\hat{\mathcal{G}}(u)=\mathcal{A}^{\mathcal{F}}_{u^*}(u-u^*)$. This ends the justification of the induction basis by  choosing $\lambda_1 := \frac{L}{\alpha_{\inf}}$ and  $\eta_1 := \frac{\tau}{1+\frac{\alpha_{\sup}}{\alpha_{\inf}}}$.

Now, let $p$ be an integer with $ 2\leq p < m $ such that Property \eqref{P} holds for $q=p$ and, constants  $\lambda_{p}$ and $\eta_{p}$. We will show that this property holds for $q = p+1$, a positive constant $\lambda_{p+1} \geq \lambda_{p}$, and for the choice of
\begin{equation}
\label{e69}
\eta_{p+1}: = \min \left\{ \frac{1}{4\lambda_{p}}  , \tau \left(1+ \frac{\alpha_{\sup}}{\alpha_{\inf}} \right)^{-(p+1)}   \right\},
\end{equation}
where due to \eqref{e69}, we obtain $\eta_{p+1}\leq \eta_{p}$.

Now assume that $u_k \in  \mathcal{B}_{\eta_{p+1}}(u^*)$ and $\alpha_k \in [\alpha_{\inf},\alpha_{\sup}]$.  First  we investigate Property \eqref{P}  for $q=p+1$ and  $\ell \leq p$. That is, we assume that \eqref{e67} holds for any given $\ell \leq p$ and we show that \eqref{e62} holds. In this case, since $\eta_{p+1} \leq \eta_{p}$,  we can use the induction hypothesis (Property \eqref{P} for  $q=p$ ) and  conclude, for every  $j \in \{ 0,\dots,\ell\}$ and  $\lambda_{p+1}\geq \lambda_{p}$,  that
\begin{equation}
\label{e141aaa}
u_{k+j} \in \mathcal{B}_{\tau}(u^*)  \quad \text{ and } \quad \|u_{k+j}-\hat{u}^k_j\| \leq \lambda_p \|u_k-u^*\|^2                                                                                                                                                                                   \leq \lambda_{p+1} \|u_k-u^*\|^2,
\end{equation}
and, thus, \eqref{e62} holds. In the remainder of the proof, we consider the case $\ell = p+1$.  In this case  $u_k \in \mathcal{B}_{\eta_{p+1}}(u^*)$, $\alpha_k \in [\alpha_{\inf},\alpha_{\sup}] $, and
 \begin{equation}
 \label{e141}
 \|\hat{u}^k_j-u^*\| \geq \frac{1}{2} \| u_k-u^*\|  \quad \text{ for all  }  j \in \{0, \dots, p\},
 \end{equation}
 and we need to verify that $u_{k+j} \in \mathcal{B}_{\tau}(u^*)$ for $j =\{1,\dots, p+1\}$ and
\begin{equation}
\label{e142}
\|u_{k+j+1}-\hat{u}^k_{j+1}\| \leq \lambda_{p+1} \|u_k-u^*\|^2  \text{ for all } j \in \{0,\dots,p+1\}.
\end{equation}

 First, suppose that  $u_{k+j} \in \mathcal{B}_{\tau}(u^*)$ for $j=1,2,\dots,p$. By  \eqref{e80} and  \eqref{e60}, we have
\begin{equation*}
\begin{split}
\|u_{k+p+1}-u^*\| &\leq \|u_{k+p}-u^*\| + \| \mathcal{S}_{k+p} \| \\
&\leq  \|u_{k+p}-u^*\| + \frac{1}{|\alpha_{k+p}|} \| \mathcal{G}_{k+p}\| \leq (1+\frac{\alpha_{\sup}}{\alpha_{\inf}}) \| u_{k+p}-u^*\|,
\end{split}
\end{equation*}
and, in a similar manner, it can be shown by induction that
\begin{equation*}
\|u_{k+p+1}-u^*\| \leq (1+\frac{\alpha_{\sup}}{\alpha_{\inf}})^{p+1} \| u_{k}-u^*\|.
\end{equation*}
Therefore, due to the definition of $\eta_{p+1}$, it follows that $u_{\ell} \in \mathcal{B}_{\tau}(u^*)$ for every $\ell \in \{1,2,\dots,p+1\}$ and any $u_k \in \mathcal{B}_{\eta_{p+1}}(u^*)$ and $\alpha_k \in [\alpha_{\inf},\alpha_{\sup}] $. It remains to verify \eqref{e142}. In fact, due to  \eqref{e141},  the induction hypothesis,  and the fact that $\eta_{q+1} \leq \eta_{q}$,  \eqref{e142} holds  for any arbitrary $\lambda_{p+1} \geq \lambda_{p}$ and $j \leq p$ .  Hence, it suffices to show that
\begin{equation}
\label{e143}
\|u_{k+p+1}-\hat{u}^k_{p+1}\| \leq \lambda_{p+1} \|u_k-u^*\|^2
 \end{equation}
for some  $\lambda_{p+1}\geq \lambda_{p}$.

  By using \eqref{e1}  and the triangle inequality,  we obtain
 \begin{equation}
 \label{e83}
 \begin{split}
 \|u_{k+p+1}-\hat{u}^k_{p+1}\|  &\leq \| u_{k+p}-\frac{1}{\alpha_{k+p} } \mathcal{G}(u_{k+p})-(\hat{u}^k_{p}-\frac{1}{\hat{\alpha}^k_{p}}\hat{\mathcal{G}}(\hat{u}^k_{p}))\| \\
 & \leq \| u_{k+p}-\hat{u}^k_{p}\|+  \frac{1}{|\hat{\alpha}^k_{p}|}\| \mathcal{G}(u_{k+p})-\hat{\mathcal{G}}(\hat{u}^k_{p}) \|  \\
 &+ \abs{ \frac{1}{\alpha_{k+p}}-\frac{1}{\hat{\alpha}^k_{p}} } \|\mathcal{G}(u_{k+p})\|.
 \end{split}
\end{equation}
From now on, we define $c$ as  a positive generic constant which depends only on $\tau$, $\alpha_{\inf}$, $\alpha_{\sup}$ and $m$ , but not on $k$, and the choice of $\alpha_k$ and $u_k \in \mathcal{B}_{\tau}(u^*)$. We shall show that the following inequalities hold
\begin{align}
 \frac{1}{|\hat{\alpha}^k_{p}|} \| \mathcal{G}(u_{k+p})-\hat{\mathcal{G}}(\hat{u}^k_{p}) \| & \leq c\|u_k-u^*\|^2, \label{e63} \\
\abs{ \frac{1}{\alpha_{k+p}}-\frac{1}{\hat{\alpha}^k_{p}} } \|\mathcal{G}(u_{k+p})\|& \leq c\|u_k-u^*\|^2. \label{e64}
\end{align}
\begin{description}
\item[Verification of inequality \eqref{e63}:] First, by adding and subtracting  $\hat{\mathcal{G}}(u_{k+p})$,  using the triangle inequality, \ref{e45} and \ref{e46},  we obtain
\begin{equation}
\label{e68}
\begin{split}
\| \mathcal{G}(u_{k+p})-\hat{\mathcal{G}}(\hat{u}^k_{p}) \| &\leq \| \mathcal{G}(u_{k+p})-  \hat{\mathcal{G}}(u_{k+p}) \| +  \|\hat{\mathcal{G}}(u_{k+p}) - \hat{\mathcal{G}}(\hat{u}^k_{p}) \|  \\
&\leq \| \mathcal{G}(u_{k+p})-  \mathcal{A}^{\mathcal{F}}_{u^*}(u_{k+p}-u^*) \| +  \| \mathcal{A}^{\mathcal{F}}_{u^*}(u_{k+p} - \hat{u}^k_{p}) \| \\
&\leq L\| u_{k+p}-u^*\|^2+\alpha_{\sup}\| u_{k+p}-\hat{u}^k_{p} \| \leq c\|u_k- u^*\|^2,
\end{split}
\end{equation}
where  $\hat{\mathcal{G}}(u)=\mathcal{A}^{\mathcal{F}}_{u^*}(u-u^*)$. In the last estimate, we have used the induction hypothesis  and  the fact that
\begin{equation}
\label{e92}
\begin{split}
 \|u_{k+p}-u^*\| &\leq \|u_k-u^*\| + \sum^{p-1}_{i=0} \| \mathcal{S}_{k+i} \| \leq  \left(1+ \sum^{p}_{i = 1} \left(\frac{\alpha_{\sup}}{\alpha_{\inf}} \right)^{i} \right) \| u_k-u^*\| \\
 &\leq \left(1+ \sum^{m}_{i = 1} \left(\frac{\alpha_{\sup}}{\alpha_{\inf}} \right)^{i} \right) \| u_k-u^*\|.
\end{split}
 \end{equation}
 Now by using \eqref{e68} and \eqref{e81}, we can infer that
\begin{equation*}
\frac{1}{|\hat{\alpha}^k_{p}|} \| \mathcal{G}(u_{k+p})-\hat{\mathcal{G}}(\hat{u}^k_{p}) \| \leq \frac{1}{\alpha_{\inf}}  \| \mathcal{G}(u_{k+p})-\hat{\mathcal{G}}(\hat{u}^k_{p}) \| \leq c \| u_{k}-u^*\|^2.
\end{equation*}
\item[Verification of inequality \eqref{e64}:] Here we need only to show that
\begin{equation}
\label{e79}
\abs{\frac{1}{\alpha_{k+p}}-\frac{1}{\hat{\alpha}^k_p}}  \leq c\|u_k-u^*\|.
\end{equation}
Then, thanks to \eqref{e60}, \eqref{e92}, and \eqref{e79},  we obtain
 \begin{equation*}
 \abs{ \frac{1}{\alpha_{k+p}}-\frac{1}{\hat{\alpha}^k_{p}} } \|\mathcal{G}(u_{k+p})\| \leq c \alpha_{\sup}\| u_k-u^*\|\|u_{k+p}-u^*\| \leq c \|u_k-u^*\|^2.
\end{equation*}
which implies \eqref{e64}. 

Due to  \eqref{e124}  we have only these two cases :
\begin{enumerate}
\item $\hat{\alpha}_{p}^k = \hat{\alpha}^{BB1, k}_p$ and $\alpha_{k+p}=\alpha^{BB1}_{k+p}$.
\item $\hat{\alpha}_{p}^k = \hat{\alpha}^{BB2, k}_p$ and $\alpha_{k+p}=\alpha^{BB2}_{k+p}$.
\end{enumerate}
We investigate the first case. Due to \eqref{e1a}, we have
\begin{equation}
\label{e144}
\begin{split}
\frac{1}{\alpha_{k+p}} = \frac{(\mathcal{S}_{k+p-1},\mathcal{S}_{k+p-1})}{(\mathcal{S}_{k+p-1},\mathcal{Y}_{k+p-1} )},   \text{  and   } \frac{1}{\hat{\alpha}^k_{p}} = \frac{(\hat{\mathcal{S}}^k_{p-1},\hat{\mathcal{S}}^k_{p-1})}{(\hat{\mathcal{S}}^k_{p-1},\hat{\mathcal{Y}}^k_{p-1} )}.
\end{split}
\end{equation}
Due to \eqref{e141} and the induction hypothesis i.e.,   property \eqref{P} for $q= p$, we have
\begin{equation}
\label{e82}
\|\mathcal{S}_{k+p-1}-\hat{\mathcal{S}}^k_{p -1}\| \leq \|u_{k+p}-\hat{u}^k_{p}\| + \| u_{k+p-1}-\hat{u}^k_{p-1}\|  \leq 2\lambda_{p} \|u_k-u^*\|^2 ,
\end{equation}
and, as a consequence, we obtain
\begin{equation}
\begin{split}
\label{e70}
\abs{\| \mathcal{S}_{k+p -1} \|^2  - \|\hat{\mathcal{S}}^{k}_{p -1} \|^2} & \leq \abs{2(  \mathcal{S}_{k+p -1}, \mathcal{S}_{k+p -1}- \hat{\mathcal{S}}^{k}_{p -1})- \|\hat{\mathcal{S}}^{k}_{p -1}- \mathcal{S}_{k+p -1}\|^2 }\\
&  \leq  c\|u_k-u^*\|^3.
\end{split}
\end{equation}
Further, by \eqref{e81}, \eqref{e141}, we have
\begin{equation}
\label{e71}
 \begin{split}
 \|\hat{\mathcal{S}}^k_{p -1}\| &= \frac{1}{|\hat{\alpha}^k_{p-1}|}\| \hat{\mathcal{G}}^k_{p-1}\| \geq \frac{1}{\alpha_{\sup}} \| \mathcal{A}^{\mathcal{F}}_{u^*} (\hat{u}^k_{p-1}-u^*)\| \geq \frac{\alpha_{\inf}}{\alpha_{\sup}} \| (\hat{u}^k_{p-1}-u^*) \|\\
 & \geq \frac{\alpha_{\inf}}{2\alpha_{\sup}} \| (\hat{u}^k_{0}-u^*) \| = \frac{\alpha_{\inf}}{2\alpha_{\sup}} \| (u_k-u^*) \|.
 \end{split}
 \end{equation}
From \eqref{e70} and \eqref{e71}, it follows that
 \begin{equation}
 \label{e76}
 \abs{1-\frac{\|  \mathcal{S}_{k+p -1} \|^2}{\| \hat{\mathcal{S}}^{k}_{p -1}  \|^2}} \leq c\| u_k-u^*\|.
 \end{equation}
Now observe that
\begin{equation}
\label{e134}
 \begin{split}
 (\mathcal{S}_{k+p-1},\mathcal{Y}_{k+p-1} )-(\hat{\mathcal{S}}^k_{p-1},\hat{\mathcal{Y}}^k_{p-1} )& = (\mathcal{S}_{k+p-1}, \mathcal{Y}_{k+p-1} -\hat{\mathcal{Y}}^k_{p-1})+(\mathcal{S}_{k+p-1}-\hat{\mathcal{S}}^k_{p-1}, \hat{\mathcal{Y}}^k_{p-1}) \\
 & = (\mathcal{S}_{k+p-1}, \mathcal{Y}_{k+p-1} -\hat{\mathcal{Y}}^k_{p-1})+ ( \mathcal{S}_{k+p-1}-\hat{\mathcal{S}}^k_{p-1}, \mathcal{A}^{\mathcal{F}}_{u^*}\hat{\mathcal{S}}^k_{p-1}).
 \end{split}
 \end{equation}
Using  \eqref{e82} and \ref{e46}, we obtain
 \begin{equation}
 \label{e135}
 \begin{split}
 &\abs{(\mathcal{S}_{k+p-1}-\hat{\mathcal{S}}^k_{p-1}, \mathcal{A}^{\mathcal{F}}_{u^*}\hat{\mathcal{S}}^k_{p-1})} \\
 & =\abs{(\mathcal{S}_{k+p-1}-\hat{\mathcal{S}}^k_{p-1}, \mathcal{A}^{\mathcal{F}}_{u^*}\mathcal{S}_{k+p-1})- (\mathcal{S}_{k+p-1}-\hat{\mathcal{S}}^k_{p-1}, \mathcal{A}^{\mathcal{F}}_{u^*}( \mathcal{S}_{k+p-1} -\hat{\mathcal{S}}^k_{p-1}))}\\
&\leq c\|u_k-u^*\|^3,
 \end{split}
 \end{equation}
 and, by \eqref{e68} and the induction hypothesis, we have
 \begin{equation}
 \label{e74}
 \abs{(\mathcal{S}_{k+p -1}, \mathcal{Y}_{k+p -1}- \hat{\mathcal{Y}}^k_{p-1})} \leq \| \mathcal{S}_{k+p -1} \|(\| \mathcal{G}_{k+p }- \hat{\mathcal{G}}^k_{p}\| +\| \mathcal{G}_{k+p-1 }- \hat{\mathcal{G}}^k_{p-1}\|)\leq c \|u_k -u^*\|^3.
 \end{equation}
Hence, using \eqref{e134}, \eqref{e135}, and \eqref{e74}, we have
\begin{equation}
\label{e97}
 \abs{(\mathcal{S}_{k+p-1},\mathcal{Y}_{k+p-1} )-(\hat{\mathcal{S}}^k_{p-1},\hat{\mathcal{Y}}^k_{p-1} )} \leq  c\| u_k-u^* \|^3.
 \end{equation}
 Moreover, by using  \ref{e46}, \eqref{e80},  \eqref{e60},  and the facts that $u_{k+p}, u_{k+p-1} \in \mathcal{B}_{\tau}(u^*)$ and $\alpha_{k} \leq \alpha_{\sup}$ for all $k\geq 1$,  we can write that
 \begin{equation}
 \label{e72}
\begin{split}
(\mathcal{S}_{k+p-1},\mathcal{Y}_{k+p-1} )& = ( \mathcal{S}_{k+p-1}, \mathcal{G}_{k+p}-\mathcal{G}_{k+p -1}) \\
&\geq  \alpha_{\inf} \|\mathcal{S}_{k+p-1}\|^2 = \alpha_{\inf} \abs{\frac{1}{\alpha_{k+p-1}}}^2\| \mathcal{G}_{k+p -1}\|^2 \\
&\geq   \frac{\alpha_{\inf}}{\alpha^2_{\sup}}\| \mathcal{G}_{k+p -1}\|^2 =  \frac{\alpha_{\inf}}{\alpha^2_{\sup}}\| \mathcal{G}(u_{k+p -1})-\mathcal{G}(u^*)\|^2\\
&\geq \frac{\alpha^3_{\inf}}{\alpha^2_{\sup}}\|u_{k+p-1}-u^* \|^2.
\end{split}
\end{equation}
Further, by \ref{e46}, the definition of $\eta_{p+1}$ in \eqref{e69},  \eqref{e141}  and  \eqref{P} with $q = p$, we have
\begin{equation}
 \label{e73}
\begin{split}
\| u_{k+p-1}-u^* \|^2 &\geq \frac{1}{2}\| \hat{u}^k_{p-1}-u^* \|^2-\| u_{k+p-1} -\hat{u}^k_{p -1}\|^2 \\
&\geq \frac{1}{8} \| \hat{u}^k_0-u^* \|^2 -\lambda^2_p \| u_k - u^*\|^4 \\
&\geq (\frac{1}{8}-\lambda^2_p\eta_{p+1}^2) \| u_k - u^* \|^2 = \frac{1}{16} \| u_k - u^* \|^2,
\end{split}
\end{equation}
Combining \eqref{e72} and \eqref{e73} we have
 \begin{equation}
 \label{e75}
(\mathcal{S}_{k+p-1},\mathcal{Y}_{k+p-1} )\geq \frac{\alpha^3_{\inf}}{\alpha^2_{\sup}}\|u_{k+p-1}-u^* \|^2 \geq  \frac{\alpha^3_{\inf}}{16\alpha^2_{\sup}}  \| u_k - u^* \|^2.
\end{equation}
From \eqref{e97} and \eqref{e75} we can write
\begin{equation}
\label{e77}
\abs{1-\frac{(\hat{\mathcal{S}}^k_{p-1},\hat{\mathcal{Y}}^k_{p-1} )}{(\mathcal{S}_{k+p-1},\mathcal{Y}_{k+p-1} )}} \leq c\| u_k-u^*\|.
\end{equation}
Now, observe that by \eqref{e144}
\begin{equation}
\label{e78}
\begin{split}
&\abs{\frac{1}{\alpha_{k+p}}-\frac{1}{\hat{\alpha}^k_{p}}} \leq \abs{ \frac{(\mathcal{S}_{k+p-1},\mathcal{S}_{k+p-1})}{(\mathcal{S}_{k+p-1},\mathcal{Y}_{k+p-1} )} -\frac{(\hat{\mathcal{S}}^k_{p-1},\hat{\mathcal{S}}^k_{p-1} )}{(\hat{\mathcal{S}}^k_{p-1},\hat{\mathcal{Y}}^k_{p-1} )}}
\\
&=\frac{1}{|\hat{\alpha}^k_{p}|}\abs{1- \left( \frac{(\mathcal{S}_{k+p-1},\mathcal{S}_{k+p-1})}{(\hat{\mathcal{S}}^k_{p-1},\hat{\mathcal{S}}^k_{p-1} )} \right)   \left(  \frac{(\hat{\mathcal{S}}^k_{p-1},\hat{\mathcal{Y}}^k_{p-1} )}{(\mathcal{S}_{k+p-1},\mathcal{Y}_{k+p-1} )} \right)}\\
& \leq \frac{1}{\alpha_{\inf}} \abs{1- \left( \frac{(\mathcal{S}_{k+p-1},\mathcal{S}_{k+p-1})}{(\hat{\mathcal{S}}^k_{p-1},\hat{\mathcal{S}}^k_{p-1} )} \right)   \left(  \frac{(\hat{\mathcal{S}}^k_{p-1},\hat{\mathcal{Y}}^k_{p-1} )}{(\mathcal{S}_{k+p-1},\mathcal{Y}_{k+p-1} )} \right)} \\
& = \frac{1}{\alpha_{\inf}}\abs{\phi_1(1-\phi_2)+\phi_2} \leq \frac{1}{\alpha_{\inf}}(|\phi_1|+|\phi_2|+|\phi_1\phi_2|),
\end{split}
\end{equation}
where
\begin{equation}
\label{e91}
\phi_1: = 1-\frac{(\mathcal{S}_{k+p-1},\mathcal{S}_{k+p-1})}{(\hat{\mathcal{S}}^k_{p-1},\hat{\mathcal{S}}^k_{p-1} )} \quad  \text{ and  } \quad  \phi_2:=1-\frac{(\hat{\mathcal{S}}^k_{p-1},\hat{\mathcal{Y}}^k_{p-1} )}{(\mathcal{S}_{k+p-1},\mathcal{Y}_{k+p-1} )}.
\end{equation}
By \eqref{e76},  \eqref{e77}, \eqref{e78}, and \eqref{e91}, we can infer that estimate \eqref{e79} holds for the case that $\alpha_{k+p} = \alpha^{BB1}_{k+p}$ and $\hat{\alpha}^{k}_{p} = \hat{\alpha}^{BB1, k}_p$ are chosen.

Now we deal with the second case, i.e., $\alpha_{k+p} = \alpha^{BB2}_{k+p}$ and $\hat{\alpha}^k_p = \hat{\alpha}^{BB2, k}_p$ . First due to  \eqref{e1a}, we have
\begin{equation*}
\begin{split}
\frac{1}{\alpha_{k+p}} = \frac{(\mathcal{S}_{k+p-1},\mathcal{Y}_{k+p-1})}{(\mathcal{Y}_{k+p-1},\mathcal{Y}_{k+p-1} )},   \text{  and   } \frac{1}{\hat{\alpha}^k_{p}} = \frac{(\hat{\mathcal{S}}^k_{p-1},\hat{\mathcal{Y}}^k_{p-1})}{(\hat{\mathcal{Y}}^k_{p-1},\hat{\mathcal{Y}}^k_{p-1} )}.
\end{split}
\end{equation*}
By using the fact that  $u_{k+p}, u_{k+p-1}, \hat{u}^k_{p-1}, \hat{u}^k_{p} \in \mathcal{B}_{\tau}(u^*) $, and the hypothesis of induction which is applicable due \eqref{e141},  we can write
\begin{equation}
\label{e100}
\begin{split}
\| \mathcal{Y}_{k+p-1}-\hat{\mathcal{Y}}^k_{p-1} \| \leq \| \mathcal{G}_{k+p }- \hat{\mathcal{G}}^k_{p}\| +\| \mathcal{G}_{k+p-1 }- \hat{\mathcal{G}}^k_{p-1}\| \leq c \|u_k-u^*\|^2.
\end{split}
\end{equation}
In addition, by using \eqref{e60}, \eqref{e92},  and the triangle inequality we obtain
\begin{equation}
\label{e101}
\begin{split}
\| \mathcal{Y}_{k+p-1}\| & = \| \mathcal{G}_{k+p}-\mathcal{G}_{k+p-1}\| \leq \| \mathcal{G}_{k+p}-\mathcal{G}(u^*)\|+\| \mathcal{G}_{k+p-1}- \mathcal{G}(u^*)\|   \\ &\leq \alpha_{\sup} \left( \| u_{k+p}-u^*\|+\| u_{k+p-1}-u^*\| \right) \leq c \|u_k-u^*\|.
\end{split}
\end{equation}
From  \eqref{e100}, \eqref{e101}, we deduce
\begin{equation}
\label{e104}
\begin{split}
\abs{\| \mathcal{Y}_{k+p -1} \|^2  - \|\hat{\mathcal{Y}}^{k}_{p -1} \|^2} & \leq \abs{(  \mathcal{Y}_{k+p -1}, \mathcal{Y}_{k+p -1}- \hat{\mathcal{Y}}^{k}_{p -1})+ ( \mathcal{Y}_{k+p -1}- \hat{\mathcal{Y}}^{k}_{p -1},   \hat{\mathcal{Y}}^{k}_{p -1}   ) }\\
&  \leq  \|\mathcal{Y}_{k+p -1}\|\| \mathcal{Y}_{k+p -1}- \hat{\mathcal{Y}}^{k}_{p -1}\|+\| \hat{\mathcal{Y}}^{k}_{p -1}\|\|  \mathcal{Y}_{k+p -1}- \hat{\mathcal{Y}}^{k}_{p -1}\|  \\
& \leq c\|u_k-u^*\|^3,
\end{split}
\end{equation}
where in the last line we have used the fact that
\begin{equation}
\label{e102}
\|\hat{\mathcal{Y}}^k_{p-1}\| \leq  \| \mathcal{Y}_{k+p-1}-\hat{\mathcal{Y}}^k_{p-1} \|  +\| \mathcal{Y}_{k+p-1}\|.
\end{equation}
Furthermore, by using \ref{e45}, \eqref{e81},  \eqref{e71},  and \eqref{e102}, we obtain
\begin{equation}
\label{e103}
\| \hat{\mathcal{Y}}^k_{p-1}\| = \| \mathcal{A}^{\mathcal{F}}_{u^*} \hat{\mathcal{S}}^k_{p-1}\| \geq \alpha_{\inf} \|\hat{\mathcal{S}}^k_{p-1}\| \geq \frac{\alpha^2_{\inf}}{2\alpha_{\sup}} \| (u_k-u^*) \|,
\end{equation}
and, as a consequence, it follows from \eqref{e104} and \eqref{e103} that
\begin{equation}
\label{e105}
 \abs{1-\frac{\|  \mathcal{Y}_{k+p -1} \|^2}{\| \hat{\mathcal{Y}}^{k}_{p -1}  \|^2}} \leq c\| u_k-u^*\|.
 \end{equation}
 Now similarly to the case for $BB1$,  by \eqref{e80} we can write
 \begin{equation}
 \label{e106}
\begin{split}
&\abs{\frac{1}{\alpha_{k+p}}-\frac{1}{\hat{\alpha}^k_{p}}} = \abs{ \frac{(\mathcal{S}_{k+p-1},\mathcal{Y}_{k+p-1})}{(\mathcal{Y}_{k+p-1},\mathcal{Y}_{k+p-1} )} -\frac{(\hat{\mathcal{S}}^k_{p-1},\hat{\mathcal{Y}}^k_{p-1} )}{(\hat{\mathcal{Y}}^k_{p-1},\hat{\mathcal{Y}}^k_{p-1} )}}
\\
&=\frac{1}{|\alpha_{k+p}|}\abs{1- \left( \frac{(\mathcal{Y}_{k+p-1},\mathcal{Y}_{k+p-1})}{(\hat{\mathcal{Y}}^k_{p-1},\hat{\mathcal{Y}}^k_{p-1} )} \right)\left(\frac{(\hat{\mathcal{S}}^k_{p-1},\hat{\mathcal{Y}}^k_{p-1} )}{(\mathcal{S}_{k+p-1},\mathcal{Y}_{k+p-1} )} \right)}\\
& \leq \frac{1}{\alpha_{\inf}} \abs{1- \left( \frac{(\mathcal{Y}_{k+p-1},\mathcal{Y}_{k+p-1})}{(\hat{\mathcal{Y}}^k_{p-1},\hat{\mathcal{Y}}^k_{p-1} )} \right)\left(\frac{(\hat{\mathcal{S}}^k_{p-1},\hat{\mathcal{Y}}^k_{p-1} )}{(\mathcal{S}_{k+p-1},\mathcal{Y}_{k+p-1} )} \right)}\\
& = \frac{1}{\alpha_{\inf}}\abs{\phi_1(1-\phi_2)+\phi_2} \leq \frac{1}{\alpha_{\inf}}(|\phi_1|+|\phi_2|+|\phi_1\phi_2|),
\end{split}
\end{equation}
where
\begin{equation}
\label{e107}
\phi_1: = 1-\frac{(\mathcal{Y}_{k+p-1},\mathcal{Y}_{k+p-1})}{(\hat{\mathcal{Y}}^k_{p-1},\hat{\mathcal{Y}}^k_{p-1} )} \quad  \text{ and  } \quad  \phi_2:=1-\frac{(\hat{\mathcal{S}}^k_{p-1},\hat{\mathcal{Y}}^k_{p-1} )}{(\mathcal{S}_{k+p-1},\mathcal{Y}_{k+p-1} )}.
\end{equation}
By \eqref{e77},  \eqref{e105}, \eqref{e106}, and \eqref{e107}, we can infer that \eqref{e79} holds for the case $BB2$.
\end{description}
Hence, we are finished with the verification of \eqref{e64}. Now from \eqref{e83}, \eqref{e63}, and \eqref{e64}, estimate \eqref{e143} follows and, thus,  the property \eqref{P} holds  for $q = p+1$.  Since $m$ is fixed and finite, we can choose $\lambda$ and $\eta$ independent of $k$ and $\ell$, and, thus the proof is complete.

\bibliographystyle{spmpsci}      

\end{document}